\documentclass[12pt]{article}
\usepackage{amsmath, amsthm, amssymb,amscd, verbatim,hyperref,mathrsfs}
\usepackage{tikz}
\usepackage{tikz-qtree}
\theoremstyle{plain}
\newtheorem{thm}{Theorem}[section]
\newtheorem{thmn}{Theorem}
\newtheorem{lem}[thm]{Lemma}
\newtheorem{cor}[thm]{Corollary}

\newtheorem{prop}[thm]{Proposition}

\theoremstyle{definition}
\newtheorem{defn}[thm]{Definition}

\newtheorem{nota}[thmn]{Notation}
\theoremstyle{remark}
\newtheorem{rem}[thm]{Remark}
\newtheorem*{rema}{{\bf Remark 6.A}}
\newtheorem*{remb}{{\bf Remark 4.1.A}}
\newtheorem*{remc}{{\bf Remark 4.2.B}}

\newcommand{\nc}{\newcommand} 
\nc{\hb}{\mathbb} 
\nc{\M}{\mathcal} 
\nc{\mf}{\mathfrak}
\nc{\mbf}{\mathbf}
\nc{\DMO}{\DeclareMathOperator}

\newbox\noforkbox \newdimen\forklinewidth
\setbox0\hbox{$\textstyle\smile$}
\setbox1\hbox to \wd0{\hfil\vrule width \forklinewidth depth-2pt
 height 10pt \hfil}
\setbox\noforkbox\hbox{\lower 2pt\box1\lower 2pt\box0\relax}
\def\anchor{\mathop{\copy\noforkbox}\limits}
 
\setbox0\hbox{$\textstyle\smile$}
\setbox1\hbox to \wd0{\hfil{\sl /\/}\hfil}
\setbox2\hbox to \wd0{\hfil\vrule height 10pt depth -2pt width
               \forklinewidth\hfil}
\newbox\doesforkbox
\setbox\doesforkbox\hbox{\box1 \lower 2pt\box2\lower2pt\box0\relax}
\def\nanchor{\mathop{\copy\doesforkbox}\limits}
 
\nc{\cA}{{\M A}} \nc{\cB}{{\M B}} \nc{\cC}{{\M C}} \nc{\cD}{{\M D}}
\nc{\cE}{{\M E}} \nc{\cF}{{\M F}} \nc{\cG}{{\M G}} \nc{\cH}{{\M H}}
\nc{\cI}{{\M I}} \nc{\cJ}{{\M J}} \nc{\cK}{{\M K}} \nc{\cL}{{\M L}}
\nc{\cM}{{\M M}} \nc{\cN}{{\M N}} \nc{\cO}{{\M O}} \nc{\cP}{{\M P}}
\nc{\cQ}{{\M Q}} \nc{\cR}{{\M R}} \nc{\cS}{{\M S}} \nc{\cT}{{\M T}}
\nc{\cU}{{\M U}} \nc{\cV}{{\M V}} \nc{\cW}{{\M W}} \nc{\cX}{{\M X}}
\nc{\cY}{{\M Y}} \nc{\cZ}{{\M Z}}
\nc{\Aa}{{\hb A}} \nc{\Cc}{{\hb C}} \nc{\Gg}{{\hb G}}
\nc{\Nn}{{\hb N}} \nc{\Pp}{{\hb P}} 
\nc{\Qq}{{\hb Q}} \nc{\Rr}{{\hb R}} \nc{\Zz}{{\hb Z}}
\nc{\mfa}{{\mf a}} \nc{\mfb}{{\mf b}} \nc{\mfk}{{\mf k}}
\nc{\mfm}{{\mf m}} \nc{\mfp}{{\mf p}} \nc{\mfq}{{\mf q}}
\nc{\mfr}{{\mf r}}
\nc{\fP}{{\mf P}}
\DMO*{\trdeg}{td}
\DMO*{\spec}{Spec}
\DMO*{\fork}{\nanchor}
\DMO*{\dnf}{\anchor}
\DMO{\RU}{RU}
\DMO{\deter}{det}
\DMO{\RM}{RM}
\DMO{\RC}{RC}
\DMO{\Real}{Re}
\DMO{\Imag}{Im}
\DMO{\tr}{tr}
\DMO{\qc}{QC}
\DMO{\Hu}{Hull}
\DMO{\leg}{length}
\DMO{\area}{area}
\DMO{\dia}{diameter}
\DMO{\iso}{Iso}
\DMO{\dis}{dist}
\DMO{\grad}{grad}
\DMO{\vol}{volume}
\DMO{\gra}{grad}
\DMO{\hd}{nbhd}
\DMO{\dv}{div}
\DMO{\Psl}{PSL}
\nc{\Mb}{\mathfrak^{2b/\delta}_{K_x}}
\nc{\Ma}{\mathfrak^{2a/\delta}_{K_x}}
\nc{\dif}{\mathrm{d}}
\nc{\G}{\Gamma}
\nc{\g}{\gamma}
\nc{\D}{\nabla}
\nc{\p}{\partial}
\nc{\DD}{\Delta^2}
\nc{\pp}{\partial^2} 
\nc{\de}{\delta}
\nc{\td}[2]{\trdeg{({#1}/{#2})}}
\nc{\dtd}[2]{\trdeg_{\delta}{({#1}/{#2})}}
\nc{\dspec}[1]{\spec_{\delta}{#1}}
\nc{\ddim}[1]{\dimen_{\delta}{#1}}
\nc{\gens}[1]{\langle {#1} \rangle}        
\nc{\gen}[2]{ {#1} \langle {#2} \rangle } 
\nc{\form}{\Omega}
\nc{\set}[1]{\left\{ {#1} \right\}}
\nc{\mr}{\hat}
\nc{\pr}{\partial}
\nc{\bc}[3]{\cB^{#1}({#2},{#3})=B^{#1}_{#2}(C^{#2}_{#3}(t)+B^{#1}_{#3}(C^{#2}_{#3}(t))} 
\nc{\tuple}[2]{{#1},\ldots,{#2}} \nc{\ptu}[2]{{#1}:\ldots:{#2}}

\nc{\maps}[3]{{#1}\!:\!{#2}\rightarrow{#3}}
\nc{\map}[2]{{#1}\rightarrow {#2}} \nc{\res}[2]{{#1} |_{#2}}
\nc{\imbed}{\hookrightarrow}
\title{All finitely generated Kleinian groups of small Hausdorff dimension are classical Schottky groups.}
\author{ Yong Hou }
\date{Princeton University}
\begin{document}
\maketitle
\begin{abstract}
This is the second part of our works on Hausdorff dimension of Schottky groups \cite{HS}.
In this paper we prove that there exists a universal positive number $\lambda>0$,  
such that up to finite index, any finitely-generated non-elementary Kleinian group with limit set of Hausdorff dimension $<\lambda$
is a classical Schottky group. The proof rely on our previous works in \cite{HS}, \cite{Hou} which provide the foundation to the general result of this paper. Our result can also be considered as a converse to the well-known theorem of Doyle \cite{Doyle} and Phillips-Sarnak \cite{phillips}.
\end{abstract}
\setcounter{tocdepth}{1}
\tableofcontents
\section{Introduction and Main Theorem}
Let $x\in\mathbb{H}^3$ be a point in the hyperbolic $3$-space. For a given Kleinian group (finitely generated, torsion-free, discrete subgroup $\G$ of $\Psl(2,\mathbb{C})$), the limit set $\Lambda_\G$ of $\G$ is defined as
$\Lambda_\G=\overline{\G x}\cap\partial\mathbb{H}^3$. $\Lambda_\G$
is the minimal non-empty closed  subset of $\partial\mathbb{H}^3$ which is invariant under $\G.$ 
The group $\G$ is called \emph{non-elementary} if $\Lambda_\G$ contains more then two points. 
When $\Lambda_\G\not=\partial\mathbb{H}^3$ then the group $\G$ is said to be of second kind, otherwise it
is called first kind. The set
$\Omega_\G=\partial\mathbb{H}^3-\Lambda_\G$ is called region of discontinuity of $\G$, and
it follows that $\G$ acts properly discontinuously on $\Omega_\G.$ \par

Given a collection of disjointed 
closed topological disks $D_i,D_i', 1\le i\le k$ in the Riemann sphere $\partial\mathbb{H}^3=\overline{\mathbb{C}}$ with
boundary curves $\partial D_i=\Delta_i,\partial D_i'=\Delta_i'.$ By definition $\Delta_i$ are closed Jordan curves in Riemann sphere 
$\partial\mathbb{H}^3.$ Let $\{\g_i\}^k_1\subset\Psl(2,\mathbb{C})$
such that $\g_i(\Delta_i)=\Delta'_i$ and $\g_i(D_i)\cap D'_i=\emptyset,$ the group
$\G$ generated by $\{\g_1,...,\g_k\}$ is a free Kleinian
group of rank $k,$ and $\G$ is called marked Schottky group with marking $\{\g_1,...,\g_k\}$. Equivalently, one can also define rank $k$ Schottky group to
be the image of $\rho,$ discrete faithful convex-compact representation of rank $k$ free group $\mathbb{F}_k$ into $\text{PSL}(2,\mathbb{C}),$ with marking
given by $\rho.$ The equivalence of these two definitions simply follows from \v Svarc-Milnor lemma and the classification of finitely generated purely loxodromic
free Kleinian groups. \par
Given a finitely generated Kleinian group $\G$ it is called Schottky group if it is a marked Schottky group for
some markings.
Whenever there exists a
set $\{\g_i,...,\g_k\}$ of generators with all $\Delta_i,\Delta'_i$ as circles then it is 
called a marked classical Schottky group with classical markings $\{\g_i,...,\g_k\}$, and 
$\{\g_1,...,\g_k\}$ are called classical generators.
A Schottky group $\G$ is called classical Schottky group if there exists a classical markings for $\G$.\par
Denote by $\mathfrak{J}_k\subset\Psl(2,\mathbb{C})$ the space of all rank $k$ Schottky groups up to conjugacy,
and $\mathfrak{J}_{k,o}$ be the set of all rank $k$ \emph{classical} Schottky groups up to conjugacy. One can topologies and put analytic structure on 
$\mathfrak{J}_k$ as follows:
Note that one can embed $\mathfrak{J}_k$ into $\mathbb{C}^{3k-3}$ as $(3k-3)$-dimensional complex manifold. Then $\mathfrak{J}_k$ is infinite index covering space,
in fact analytic cover, of the moduli space
of genus $k$ closed Riemann surfaces. Hence, $\mathfrak{J}_k$ can also be considered \emph{fine} moduli space of  genus $k$ Riemann surfaces. The universal covering space of $\mathfrak{J}_k$ is the Techim\"uller space.\par

Equivalently, Schottky space $\mathfrak{J}_k$ can also be viewed as the quasiconformal deformation space $QF(\G)$ of a fixed given rank $k$ classical Schottky group $\G.$ The equivalency of these two definitions simply follows from representation deformation theory of Klenian groups and the definition of Schottky groups.
Note that $QF(\G)$ is independent of choice of $\G$ since given any two ranked $k$ classical Schottky groups $\G,\G'$ marked with $\{\g_i\},\{\g'_i\}, 1\le i\le k$
there exists quasiconformal $f$ with $\{\g'_i\}=\{f\circ\g_i\circ f\}.$ 
The boundary $\partial QF(\G)\subset\Psl(2,\mathbb{C})^{k}$ contains points that are non-Schottky groups which are called \emph{geometrically infinite} Kleinian groups
of first kind (i.e. limit set is $\overline{\mathbb{C}}$). These Kleinian groups have limit set of Hausdorff dimension $2.$ In fact, there are uncountable many
such elements in $\partial QF(\G).$
\par
The set of classical Schottky groups is a non-empty and non-dense subset of set of Schottky
groups \cite{Marden}, also see \cite{Doyle}.\par 
For a given Schottky group $\G$, denote the Hausdorff dimension of $\Lambda_\G$ by $\mathfrak{D}_\G.$ Our main result is:
\begin{thm}\label{main} 
There exists a universal $\lambda>0$, 
such that all finitely generated non-elementary Kleinian group $\G$ with $\mathfrak{D}_\G<\lambda$ are classical Schottky groups.
\end{thm}
Note that Theorem \ref{main} holds up to finite index. Theorem \ref{main} can be viewed as the converse to the result by Phillips-Sarnak \cite{phillips} (for higher dimensional Kleinian groups) and Doyle \cite{Doyle} which states that there exists a universal upper bound on
the Hausdorff dimension of limit sets of all classical Schottky groups. There are also some recent interesting works on bounds of Hausdorff dimensions of limit sets of 
Kleinian groups by Kapovich\cite{kap}, Bowen\cite{Bowen}.\par
Let $\lambda>0$ be given by Theorem \ref{main}. Denote $\mathfrak{J}_{k,o}^\lambda$ as the set of classical Schottky groups of Hausdorff dimension $<\lambda.$ Then by analyticity of Hausdorff dimension function on Schottky space and Theorem \ref{main}, we have the following simple observation:
\begin{cor}\label{main-cor}
$\mathfrak{J}_{k,o}^\lambda$ is $(3k-3)$-dimensional complex manifold.
\end{cor}
\subsection{Strategy of the proof} 

The proof is based on induction on rank $k.$ It follows from main result of \cite{HS}, we know Theorem \ref{main} holds for $k=2.$ We assume Theorem \ref{main} is true for $k-1$ with $k>2.$ The ideas and techniques used are mostly generalizations of methods developed in \cite{HS}.  \par
Note that for a given rank $k$ Schottky group, when we speak of set of generators or generating set, we always mean a set of $k$ elements of free generators.
Hence by definition it is always minimal.
\begin{itemize}
\item 
\emph{Preliminary estimates}:
 We state some preliminary estimates and some of it's generalizations that are given in \cite{HS} on the locations of fixed points, centers of isometric circle of a 
given set of generators of a Schottky group. Detailed proofs of these estimates can be found in \cite{HS}. 
These estimates will give us a sufficient controls on how the fixed points
of a set of generators move in terms of the Hausdorff dimension of the limit set of the group. More precisely, when the Hausdorff dimension
of rank $k$ Schottky group $\G$ is small then, \emph{any} set of $k$ generators will have either the norm of the trace exponentially large or its fixed points
degenerates into single point exponentially fast. The main ingredients 
of the proofs of these estimates rely on the results of \cite{Hou}. This is done in Section $3$ and Section $4.$\par 
\item
\emph{Sufficient condition for classscial-ness}:
For a given sequence of rank $k$ Schottky groups, we will derive a set of sufficient conditions on any sequence of set of generators with respective to Hausdorff dimensions of its limit sets for it to contain classical Schottky group generators. If a given sequence of set of generators satisfies these conditions then, the sequence will contain a subsequence
that will eventually be classical Schottky groups. \par
These sufficient conditions ensures that a given sequence of decreasing Hausdorff dimensions will contain a subsequence of eventually classical Schottky groups are given in Section $5.$ For computational purpose, it turns out that the computations in this part are easier by using the unit ball model of the hyperbolic $3$-space. However all statements are made in the upper-half space of hyperbolic $3$-space. Of course all statements are model independent. In fact, this is our generalized version of conditions given in \cite{HS} to rank $k.$ These conditions in the basic form simply states that, if the norm of the traces of generators growth sufficiently fast compare to the degenerations of it's fixed points, then all the Schottky disks will eventually be disjoint.\par
\item
\emph{Obstructions to classical-ness}:
Let be a sequence of rank $k$ Schottky groups $\G_n$ of decreasing Hausdorff dimensions $\mathfrak{D}_n\to 0.$ Denote it's set of generators by $S_{\Gamma_n},$  
we will classify all \emph{obstructions} for this sequence to contain any classical Schottky group subsequence. 
It follows from Theorem \ref{critical} and Proposition \ref{trace}, there are at least
$k-1$ generators that will have either norms of the trace growth exponentially or it's fixed points degenerates into single point exponentially, with respect to the 
$\mathfrak{D}_n.$ Hence we first assume the sequence is $k-1$ \emph{classical} of standard form, i.e generators have $k-1$ that are classical generators with isometric centers not contained circles of the minimal element of $|\mbox{trace}|.$ Then we define \emph{degenerate types} I, II, III, IV for the given sequence. These are the obstructions to existence of classical Schottky group.\par
\item
\emph{Degenerate types becomes classical when Hausdorff siemsnion is sufficiently small}:
We show that for each of the degenerate types I, II, II, IV given as obstructions to the existence of classical Schottky group in a sequence of Schottky groups, when 
the Hausdorff dimension $\mathfrak{D}_{\G_n}\to 0$ it must become classical Schottky group. In particular, all the obstructions can be removed (degeneracy resolved) when the Hausdorff dimension is sufficiently small. The techniques used to resolve degeneracies is based generator selection process and trace-Hausdorff-dimension growth estimates.
\item
\emph{Non-collapsing fixed points subspaces}:
Based on degeneracies resolutions for sufficiently Hausdorff dimension, we then prove that Theorem \ref{main} holds for $\mathfrak{J}_k(\tau)$ Schottky subspaces for $\tau>0.$ Here the Schottky subspace $\mathfrak{J}_k(\tau)$ is defined as collection of Schottky groups up to conjugation that have a generating sets $S_{\G}$ with uniform lower bound on $Z_{S_\G}>\tau$ (bounded separations of it's fixed point sets). This is proved by induction on $k$ and removal of obstructions to classical-ness when Hausdorff dimension is small.
\item
\emph{Removing uniform lower fixed points bound constraint}:
We classify possible types of fixed points collapses into Collapsing fixed points I and Collapsing fixed points II. Given a sequence of $\G_n$ with 
$\mathfrak{D}_{\G_n}\to 0,$ we show that both types of collapsing fixed points I and II we must have either, uniform lower bound on
$\mathfrak{D}_{\G_n}+Z_{\G_n}$(sum of Hausdorff dimensions and fixed points bounds), or $\G_n$ will contain a subsequence of classical Schottky groups. This is done by analyzing fixed points distributions based on estimates of isometric circles centers of generating sets when the Hausdorff dimensions is sufficiently small.
\item
\emph{Theorem \ref{main} holds for fixed rank}: 
Based on the removability of the constraint of uniform lower bound $Z_\G>\tau$ and result on $\mathfrak{J}_k(\tau),$ we show that Theorem \ref{main} is true for rank $k$ whenever $k-1$ is true.
\item
\emph{Rank extension}:
Finally we prove rank-extension  based on fixed rank result. We show that if Theorem \ref{main} holds for rank k, i.e. there exists bounds $\lambda_k$  then for sufficiently small Hausdorff dimension, $\lambda_k$ must be independent on $k.$ This is proved by using Rank-length growth estimates Lemma \ref{critical-k}.
The Rank-length estimates is growth estimates that relate the rank number $k$ and norms of the traces with respect to other generators with respect to Hausdorff  dimension.
\end{itemize}
\subsection{Organization of the paper}
Section $2$ is used to define and lists global notations that will be used throughout the paper.
In Section $3,$ we state a strong form of Theorem 1.1 given in \cite{Hou} for rank $k$ free groups acting on hyperbolic space.
Note that the original results of \cite{Hou} are state for pinched negatively curved manifolds.
This will be used for generators selection processes, see Corollary \ref{cor-2}. 
In Section $4,$ given any sequence of rank $k$ Schottky groups $\G_n$ with bounded 
$Z_{\G_n}$ (see section $2$),
we shall prove estimates that will enable us to bounds locations of fixed points
of a given sequence of generators of $\G_n$ in terms of the Hausdorff dimensions $\mathfrak{D}_{\G_n}$ of $\G_n.$
These estimates will be used in the generators set selection procedure. In Section $5$, we will derive sufficient conditions for a given set
of $k$ generators
to be a classical generators. This is done in the hyperbolic ball space $\mathbb{B}.$  Sections $6$ to $11$, given a sequence of rank $k$ Schottky groups which is $k-1$ 
classical and of strictly decreasing Hausdorff dimension to $0$, we show that all types of degeneracies can be resolved. Hence all degenerate types will leading to classical Schottky groups.
Section $12$ to $14$, we remove non-collapsing fixed points constraint.
In Section $15$, we extend from fixed rank $k$ to finitely generated by using already established results in earlier sections.
Section $16,$ we finish the proofs of our main theorem by simple standard topological arguments.\par
\centerline{ACKNOWLEDGEMENTS}\par
The author would like to express appreciation to Benson Farb, Peter Sarnak for giving me opportunity to discuss this work in details. I would also like to express gratitude to Peter Shalen, Marc Culler, Dick Canary for many thoughts on this paper.  Finally, I like to express sincere appreciation to the referee.\par
\section{Notations}
Let $\G\subset\Psl(2,\mathbb{C})$ be a rank $k$ Schottky group generated by $<\alpha_1,\beta_2,...,\beta_k>$ with 
$\alpha_1$ having fixed points $0,\infty.$ Note that by our choice here, all generating set are \emph{minimal}, i.e generating set of $k$ elements for  given rank $k$ Schottky group.
Asssume that $\gamma\in\G$ is a loxodromic element having fixed points
$\not=\infty.$ For a $\gamma$ we write in matrix form as
$\gamma=\bigl(\begin{smallmatrix} a & b\\c & d\end{smallmatrix}\bigr),$ with $\det(\gamma)=1.$
We will set the following notations and definitions throughout the rest of the paper.
\begin{nota}\text{}
\begin{itemize}
\item Denote the critical exponent of $\G$ by $\mathfrak{D}_\G,$ which is equal the Hausdorff dimension of $\Lambda_\G$ (see Section $3$).
\item $\mathfrak{R}_\gamma$ the radius of isometric circles of $\gamma$.
\item $\eta_\gamma=-\frac{d}{c}$, $\zeta_\gamma=\frac{a}{c}.$
\item We will define two different ways to denotes the two fixed points of $\gamma:$ 
$\left\{z_{\gamma,l},z_{\gamma,u}\right\}$ the two fixed
points of $\gamma$ in $\mathbb{C}$ with $|z_{\gamma,l}|\le|z_{\gamma,u}|$, and $\left\{z_{\gamma,-},z_{\gamma,+}\right\}$ 
the two fixed points with $z_{\gamma,\pm}$ given by quadratic formula with subscripts $\pm$
corresponding to $\pm\sqrt{\tr^2(\gamma)-4}.$ Note that we always take the principle branch for the square roots of complex numbers.
\item $\mathcal{L}_\gamma$ is the axis of $\gamma.$
\item $T_\gamma$ is the translation length of $\gamma.$
\item $Z_{\beta_i}:=\min\{|z_{\beta_i,-}-z_{\beta_i,+}|,|\frac{1}{z_{\beta_i,-}}-\frac{1}{z_{\beta_i,+}}|\}.$ Geometrically, $Z_{\beta_i}$ measure how far apart do the fixed
points of $\beta_i$ are in the conformal boundary $\hat{\mathbb{C}}$ of $\mathbb{H}^3$. This will be used subsequently for rate of convergence of fixed points with respect to norm of the trace. For instance it is use for estimates in Section $4.$
\item
$Z_{<\alpha_1,\beta_i>}:=\min\{Z_{\beta_i},|z_{\beta_i,+}|,|z_{\beta_i,-}|,|z_{\beta_i,+}|^{-1},|z_{\beta_i,-}|^{-1}\}$ for all $2\le i\le k.$
whenever we need to emphases that $\beta_i$ is element of the generating set we will also denote $Z_{<\alpha_1,\beta_i>}$ by $Z_{<\alpha_1,\beta_2,...,\beta_k>}$ to explicitly write out the generators set. \par
Geometrically, $Z_{<\alpha_1,\beta_i>}$ or $Z_{<\alpha_1,\beta_2,...,\beta_k>}$ measures relative position of the fixed points of $\beta_i$ with respect to the fixed points of $\alpha_1$ in the conformal boundary $\hat{\mathbb{C}}.$ These relative positions will be used to make selection of generating sets. 

\item For $\epsilon>0,$ we say that $Z_\G>\epsilon,$ if there exists a generating set $<\alpha_1,\beta_2,...,\beta_k>$ of $\G$ such that
$Z_{<\alpha_1,\beta_2,...,\beta_k>}>\epsilon.$
\item
Given any two sequences of real numbers $\{p_n,q_n\}$, we use the notation $p_n\asymp q_n$ iff, 
there exists $\sigma>0$ such that, $\sigma^{-1}<\liminf\frac{p_n}{q_n}\le\limsup\frac{p_n}{q_n}<\sigma.$\par
\item
Given two real numbers $a\le b$, set $D_{a,b}:=\{z\in\mathbb{C}|a\le|z|\le b\}.$
\item
Given a loxodromic $\g,$ denote by $D_{\g,\pm},$ Schottky disks of $\g.$ Set $D_{\g}:=D_{\g,+}\cup D_{\g,-}.$
Note that Schottky disks are defined to be circular disks such that $\g(\text{interior}(D_{\g,+}))\cap\g(\text{interior}(D_{\g,-}))=\emptyset.$
Also note that it is true that any loxodromic element always have Schottky disks (see Proposition 5.5), but distinct elements generally don't have disjointed Schottky disks. In addition, Schottky disks may not always be isometric circle disks, since not all loxodromic elements have disjoint isometric circles. But for sufficiently large norm of the trace and control on $Z_\G$ once can always take the Schottky disks to be isometric circular disks (disks bounded by isometric circles). 
\end{itemize}
\end{nota}
\begin{nota}
Let $\{\gamma_n\}\subset\Psl(2,\mathbb{C})$ be a sequence of loxodromic transformations.
Let $\{p_n\}$ be a sequence of complex numbers, and $\{q_n\}$ a sequence of positive real numbers. We write:
\[\left|z_{\gamma_n,\pm}-p_n\right|<q_n\]
if there exists $N$ such that for every $n>N$ we have at least one of the following holds,
\begin{itemize}
\item[(i)]
\[\left|z_{\gamma_n,+}-p_n\right|<q_n,\]
\item[(ii)]
\[\left|z_{\gamma_n,-}-p_n\right|<q_n.\]
\end{itemize}
\end{nota}

\section{$\G$ Actions on Hyperbolic Space}
Let $\G$ be a finitely generated non-elementary Kleinian group, the critical exponent of $\G$ is defined as the unique positive 
number $\delta_\G$ such that the Poincar\'{e} series of $\G$ given by 
$\sum_{\g\in\G}e^{-s\dis(x,\g x)}$ is divergent if $s<\delta_\G$ and
convergent if $s>\delta_\G$. If the Poincar\'{e} series diverges at
$s=\delta_\G$ then $\G$ is said to be divergent. Bishop-Jones showed that
$\delta_\G\le\mathfrak{D}_\G$ for all analytically finite non-elementary Kleinian groups $\G$.
In-fact, if $\G$ is topologically tame ($\mathbb{H}^3/\G$ homeomorphic to the
interior of a compact manifold-with-boundary) then  $\delta_\G=\mathfrak{D}_\G$. Hence it follows
from Ian Agol \cite{agol}, and independently Danny Calegari and Dave Gabai \cite{CG} proof of tameness conjecture 
that $\delta_\G=\mathfrak{D}_\G$ for all finitely generated non-elementary Kleinian groups. Note that Agol's proof
of tamness conjecture extends to pinched negatively curved manifolds.\par
The critical exponent and hence Hausdorff dimension is a geometrically rigid object in the sense that
a decrease in $\mathfrak{D}_\G$ corresponds to a decrease in ``geometric complexity" of the group. Our following results are indications of geometric complexity.\par
We state the following theorem from \cite{Hou}, which provides the
relation between the negatively curved free group action and the critical exponent.
\begin{thm}[Hou\cite{Hou}]\label{critical}
Let $\G$ be a free non-elementary Kleinian group of rank $k$ with free generating
set $\mathcal{S}$, and $x\in\mathbb{H}^3$ then 
\[ \sum_{\g\in\mathcal{S}}\frac{1}{1+\exp(\mathfrak{D}_\G\dis(x,\g x))} \le\frac{1}{2}.\]
In particular we have at least $k-1$ distinct elements $\{\g_{i_j}\}_{1\le j\le k-1}$ 
of $\mathcal{S}$ that satisfies $\dis(x,\g_{i_j} x)\ge\log(3)/\mathfrak{D}_\G,$ and at least one element
$\g_{i_j}$ with $\dis(x,\g_{i_j} x)\ge\log(2k-1)/\mathfrak{D}_\G.$
\end{thm}
The following is a useful corollary of Theorem \ref{critical}, stated
here for the case of $\G$ is a free group of rank $2$.
\begin{cor}\label{cor-2}
Let $\mathcal{S}=\{\g_1,\g_2\}$ be a generating set for a free non-elementary Kleinian group $\G$. Let
$x\in\mathbb{H}^3$. Then
\[ \dis(x,\g_2 x)\ge\frac{1}{\mathfrak{D}_{\G}}\log\left(\frac{e^{\mathfrak{D}_\G\dis(x,\g_1 x)}+3}{e^{\mathfrak{D}_\G\dis(x,\g_1 x)}-1}\right)\]
\end{cor}
\begin{cor}\label{cor-1}
Let $\mathcal{S}=\{\g_1,\g_2\}$ be a generating set for a free non-elementary Kleinian group $\G$. Let
$x\in\mathbb{H}^3$.
Let $m$ be any integer. Then at least one of the elements $\g'$ of 
$\mathcal{S}'=\{\g^m_1\g_2,\g^{m+1}_1\g_2\}$ 
satisfies $\dis(x,\g' x)\ge\log 3/\mathfrak{D}_\G$.
\end{cor}
We will need a much stronger estimates which will relates both the rank of $\G$ and elements of generators. This is needed in order to
prove extension theorem from fixed rank to arbitrary rank.
\begin{lem}[Rank-length growth estimate]\label{critical-k}
Let $\mathcal{S}=\{\g_1,\g_2,...,\g_k\}$ be a generating set for a free non-elementary Kleinian group $\G$ of rank $k$. Let
$x\in\mathbb{H}^3$. Then for every $i$ there exists some $j\not= i$ such that,
\[ \dis(x,\g_{j} x)\ge\frac{1}{\mathfrak{D}_{\G}}\log \left(\frac{(2k-3)e^{\mathfrak{D}_\G\dis(x,\g_i x)}+(2k-1)}{e^{\mathfrak{D}_\G\dis(x,\g_i x)}-1}\right).\]
\end{lem}
\begin{proof}
Assume otherwise, 
\begin{multline*}
\sum_{\g\in\mathcal{S}}\frac{1}{1+e^{\mathfrak{D}_\G\dis(x,\g x)}}\\
>\sum_{1\le l \le k-1}\frac{1}
{1+\left(\frac{(2k-3)e^{\mathfrak{D}_\G\dis(x,\g_i x)}+(2k-1)}{e^{\mathfrak{D}_\G\dis(x,\g_i x)}-1}\right)}+\frac{1}{1+e^{\mathfrak{D}_\G\dis(x,\g_i x)}}\\
>\frac{(e^{\mathfrak{D}_\G\dis(x,\g_i x)}-1)(k-1)}{(e^{\mathfrak{D}_\G\dis(x,\g_i x)}+1)(2k-2)}+\frac{1}{1+e^{\mathfrak{D}_\G\dis(x,\g_i x)}}>\frac{1}{2},
\end{multline*}
which is a contradiction to Theorem \ref{critical}.
\end{proof}
\section{Relations of Trace, Fixed Points, and Hausdorff Dimension}
The relationships of fixed points of generating sets of a given sequence of Schottky groups $\G_n$
and the Hausdorff dimensions of $\Lambda_{\G_n}$ will be stated in this section. 
Proofs of these can be found in \cite{HS} which is stated for general rank $k.$\par
The main is idea is to find a relationship between the distribution function of the fixed points of one of the generators in terms
of the translation length growth of the other generators and the Hausdorff dimension of its limit set. By obtaining
these types of relationships we will be able to construct a new set of generators of the
given generating set of generators so that it will have predetermined distribution of its fixed points. These newly constructed set of
generators will be a crucial part in our proof of the main theorem.\par
Let $\{\G_n\}$ be a sequence of rank $k$ Schottky groups with $\mathfrak{D}_n\to 0$ generated by $\alpha_{1,n},\beta_{i,n}\in\Psl(2,\mathbb{C}), 2\le i\le k$
in $\mathbb{H}^3$ with 
$\alpha_{1,n}=\bigl(\begin{smallmatrix} \lambda_{1,n} & 0\\0 & \lambda^{-1}_{1,n}\end{smallmatrix}\bigr)$, $|\lambda_{1,n}|>1$.
Set $\beta_{i,n}=\bigl(\begin{smallmatrix} a_{i,n} & b_{i,n}\\c_{i,n} & d_{i,n}\end{smallmatrix}\bigr).$ \par
Throughout this section we assume that there exists $M>0$ such that $T_{\alpha_{1,n}}<M$ for all $n$. Set $\mathfrak{D}_n=\mathfrak{D}_{\G_n}.$ Let $\tr=\text{trace}.$\par
First we will state Corollary \ref{cor-2} in the trace form. 
\begin{prop}\label{trace}\cite{HS}
Suppose there exists $\Delta>0$ such that $Z_{\beta_{i,n}}>\Delta.$
There exists $\rho>0$ depending on $\Delta$, such that
\[|\tr(\beta_{i,n})|>\rho\left(\frac{|\lambda_{1,n}|^{2\mathfrak{D}_n}+3}{|\lambda_{1,n}|^{2\mathfrak{D}_n}-1}\right)^{\frac{1}{2\mathfrak{D}_n}}\]
for large $n$.
\end{prop}
\begin{proof}
Let $T_{i,n}$ be the translation length of $\beta_{i,n}$ and $R_{i,n}=\dis(\mathcal{L}_{\alpha_{1,n}},\mathcal{L}_{\beta_{i,n}}).$ 
Let $x_{n}$ be a point on axis of $\alpha_{1,n}$ which is the nearest point of $\mathcal{L}_{\alpha_{1,n}}$ to $\mathcal{L}_{\beta_{i,n}}.$ 
By triangle inequality, 
$T_{i,n}\ge \dis(x_n,\beta_{i,n}x_n)-2R_{i,n}$. Here the triangle inequality is:
$\text{dist}(x_n,\beta_{i,n}x_n)\le\text{dist}(\beta_{i,n}x_n,\beta_{i,n}y_n)+\text{dist}(\beta_{i,n}y_n,y_n)+\text{dist}(x_n,y_n),$ where $y_n$ is the closest point on $\mathcal{L}_{\beta_{i,n}}$ to $\mathcal{L}_{\alpha_{1,n}}.$\par
Note for sufficiently large $T_{i,n}$, we have for some positive constant
$c>0,$ $|\tr^2(\beta_{i,n})|>ce^{T_{i,n}}.$ 
Now for large $n,$ from Corollary \ref{cor-2},
\[|\tr^2(\beta_{i,n})|>c\left(\frac{|\lambda_{1,n}|^{2\mathfrak{D}_n}+3}{|\lambda_{1,n}|^{2\mathfrak{D}_n}-1}\right)^{\frac{1}{\mathfrak{D}_n}}
\left(e^{-2\dis(\mathcal{L}_{\alpha_{1,n}},\mathcal{L}_{\beta_{i,n}})}\right). \] 
The first factor on the right hand side of the last inequality is by taken exponential of inequality $T_{i,n}\ge\text{dist}(x_n,\beta_{i,n}x_n)-2R_{i,n},$ and using Corollary 3.2 for $e^{\text{dist}(x_n,\beta_{i,n}x_n)}.$ The second factor of right hand side comes from
the equality
$e^{-2R_{i,n}}=e^{-2\text{dist}(\mathcal{L}_{\alpha_{1,n}},\mathcal{L}_{\beta_{i,n}})}.$
Now $Z_{\beta_{i,n}}>\Delta$ implies that $\dis(\mathcal{L}_{\alpha_{1,n}},\mathcal{L}_{\beta_{i,n}})<M$ for some $M>0.$ Hence the result follows.
\end{proof}
We will also need to state the rank $k$ form of the above estimates that relates the rank of $\G$ as well. 
\begin{prop}\label{trace-rank}
Let $\G_n$ be a sequence of rank $k$ free group generated by $<\g_{1,n},...,\g_{k,n}>$ with $\g_{i,n}$ having fixed points $0,\infty.$
Suppose there exists $\Delta>0$ such that $Z_{<\g_{1,n},...,\g_{k,n}>}>\Delta.$
There exists $j\not=i$, and $\rho>0$ depending on $\Delta$, such that
\[|\tr(\g_{j,n})|>\rho\left(\frac{(2k-3)|\lambda_{\g_{i,n}}|^{2\mathfrak{D}_n}+(2k-1)}{|\lambda_{\g_{i,n}}|^{2\mathfrak{D}_n}-1}\right)^{\frac{1}{2\mathfrak{D}_n}}\]
for large $n$.
\end{prop}

\begin{cor}\label{trace-rank-dis}
We can state Proposition \ref{trace-rank} with condition $Z_{<\g_{1,n},...,\g_{k,n}>}$ droped as follows,
\[|\tr^2(\g_{j,n})|>c\left(\frac{(2k-3)|\lambda_{\g_{i,n}}|^{2\mathfrak{D}_n}+(2k-1)}{|\lambda_{\g_{i,n}}|^{2\mathfrak{D}_n}-1}\right)^{\frac{1}{\mathfrak{D}_n}}
\left(e^{-2\dis(\mathcal{L}_{\g_{i,n}},\mathcal{L}_{\g_{j,n}})}\right). \] 
\end{cor}
\begin{remb}\label{4A}
Without assuming bounds on $Z_{\beta_{i,n}}$ we can state above Proposition \ref{trace} as,
\[|\tr^2(\beta_{i,n})|>c\left(\frac{|\lambda_{1,n}|^{2\mathfrak{D}_n}+3}{|\lambda_{1,n}|^{2\mathfrak{D}_n}-1}\right)^{\frac{1}{\mathfrak{D}_n}}
\left(e^{-2\dis(\mathcal{L}_{\alpha_{1,n}},\mathcal{L}_{\beta_{i,n}})}\right). \] 
If $\dis(\mathcal{L}_{\alpha_{1,n}},\mathcal{L}_{\alpha_{1,n}\beta_{i,n}})<\epsilon$ then there exists $\delta>0$ such that,
\[|\tr(\beta_{i,n})|>\rho\left(\frac{|\lambda_{1,n}|^{2\mathfrak{D}_n}+3}{|\lambda_{1,n}|^{2\mathfrak{D}_n}-1}\right)^{\frac{1}{2\mathfrak{D}_n}}\]
for large $n$.
\end{remb}
Following lemmas and corollaries are given in \cite{HS} as Lemma $4.2$, Corollary $4.4$ and $4.6$. Note that all these results in \cite{HS} are originally stated for general rank $k.$ The proofs are verbatim to the proofs give in \cite{HS} with obvious notation modifications, hence omitted.\par
The results are estimates of convergence rates of fixed points of all generators of $\G_n$  in terms of its limit set $\Lambda_n$ 
Hausdroff dimension.
Set,
\[\omega(k_{i,n},l_{i,n}) =
\left|z_{\alpha^{k_{i,n}}_{1,n}\beta_{i,n}\alpha^{l_{i,n}}_{1,n},\pm}-\zeta_{\alpha^{k_{i,n}}_{1,n}\beta_{i,n}\alpha^{l_{i,n}}_{1,n}}\right|+
\left|z_{\alpha^{k_{i,n}}_{1,n}\beta_{i,n}\alpha^{l_{i,n}}_{1,n},\mp}-\eta_{\alpha^{k_{i,n}}_{1,n}\beta_{i,n}\alpha^{l_{i,n}}_{1,n}}\right|.\] 
Note that each given $k_{i,n},l_{i,n}$ provides a Nielsen transformation of the generators. We will determine what $k_{i,n},l_{i,n}$ 
are needed during the proof in  later part of the paper. The idea is that, eventually we will need to choose some collection of generators which are Nielsen transformations of $\alpha_{1,n},\beta_{i,n}$ based on the values of $k_{i,n},l_{i,n}.$ Determination on the constraints of $k_{i,n},l_{i,n}$ will require extensive analysis
of both fixed points and centers of circles of generators. Here the definition of $\omega(k_{i,n},l_{i,n})$ is to measures the distance between fixed points
and centers of circles of generators relative to $k_{i,n},l_{i,n}.$ \par
Indices $k_{i,n},l_{i,n}$ are introduced to determine relationship between fixed points convergence rates with respect to Nielsen transformation of $\beta_{i,n}$ by
$\alpha_{1,n}.$ 
Proofs of following statements can be found in \cite{HS}.
\begin{lem}\label{fix-trace}\cite{HS}
Suppose there exists $\Delta>0,M>0$ such that $Z_{<\alpha_{1,n},\beta_{i,n}>}>\Delta$ and $T_{\alpha_{1,n}}<M$ for all $n.$ 
Let $k_{i,n},l_{i,n}$ be any integers such that $T_{{\alpha_{1,n}}^{k_{i,n}}},$ $T_{{\alpha_{1,n}}^{l_{i,n}}}<M$. Then
there exists a constant $\rho>0$ such that, 
\[\omega(k_{i,n},l_{i,n})
<\frac{\rho}{|\tr(\beta_n)||\lambda^{l_n-k_n}_n|}\]
for large $n.$
\end{lem}
Note that in general one could have $T_{\alpha_{i,n}}\to 0,$ which implies that we could have $k_{i,n},l_{i,n}\to\infty.$ In fact, this is one of the degenerate cases that we will consider in Section $7.$ 
\begin{cor}\label{fix-bound-1}\cite{HS}
Suppose there exists $\Delta>0,M>0$ such that $Z_{<\alpha_{1,n},\beta_{i,n}>}>\Delta$ and $T_{\alpha_n}<M$ for all $n.$ 
Let $k_{i,n},l_{i,n}$ be any integers such that $T_{{\alpha_{1,n}}^{k_{i,n}}},$ $T_{{\alpha_{1,n}}^{l_{i,n}}}<M$. Then
for any $\delta>0$ there exists $\epsilon>0$ such that if $\mathfrak{D}_n<\epsilon$ then,
\[\omega(k_{i,n},l_{i,n})
 <\delta(|\lambda_{1,n}|^2-1).\]
\end{cor}
\begin{rem}\label{rem-2}
Note that condition $Z_{<\alpha_{1,n},\beta_{i,n}>}>\Delta$ in Lemma \ref{fix-trace} can be replaced
with conditions $|\tr(\beta_{i,n})|\to\infty$ and $|\tr(\beta_{i,n})|\le|c_{i,n}|.$
\end{rem}
\begin{cor}\label{2-tr-fixed}\cite{HS}
Suppose there exists $\Delta>0,M>0$ such that $Z_{<\alpha_{1,n},\beta_{i,n}>}>\Delta$ and $T_{\alpha_{1,n}}<M$ for all $n.$ 
Let $k_{i,n},l_{i,n}$ to be any integers such that $T_{{\alpha_{1,n}}^{k_{i,n}}},$ $T_{{\alpha_{1,n}}^{l_{i,n}}}<M$.
Then there exists constants $\delta, \rho>0$ such that, if $|\tr(\alpha^{k_{i,n}}_{1,n}\beta_{i,n}\alpha^{l_{i,n}}_{1,n})|$ $\to\infty$
then,
\[\omega(k_{i,n},l_{i,n})
<\frac{\rho}{|\tr(\alpha^{k_{i,n}}_{1,n}\beta_{i,n}\alpha^{l_{i,n}}_{1,n})||\tr(\beta_{i,n})|}\]
for all $\mathfrak{D}_n<\delta$.
\end{cor}
\begin{lem}\label{fix-bound-2}\cite{HS}
Suppose there exists $\Delta>0,M>0$ such that $Z_{<\alpha_{1,n},\beta_{i,n}>}>\Delta$ and $T_{\alpha_{1,n}}<M$ for all $n.$ 
Let $k_{i,n},l_{i,n}$ to be any integers such that $T_{{\alpha_{1,n}}^{k_{i,n}}},$ $T_{{\alpha_{1,n}}^{l_{i,n}}}<M$.
Then there
exists a constant $\sigma_1,\sigma_2>0$ such that,
\[\frac{\sigma_1|\tr(\alpha^{k_{i,n}}_{1,n}\beta_{i,n}\alpha^{l_{i,n}}_{1,n})|}{|\tr(\beta_{i,n})|}\ge
|z_{\alpha^{k_{i,n}}_{1,n}\beta_{i,n}\alpha^{l_{i,n}}_{1,n},+}-z_{\alpha^{k_{i,n}}_{1,n}\beta_{i,n}\alpha^{l_{i,n}}_{1,n},-}|
\ge\frac{\sigma_2|\tr(\alpha^{k_{i,n}}_{1,n}\beta_{i,n}\alpha^{l_{i,n}}_{1,n})|}{|\tr(\beta_{i,n})|},\]
for all $n$ sufficiently large.
\end{lem}
Although next Lemma is not used in the rest of the paper but we include it here to demonstrate relations between fixed points and Hausdorff dimensions.
\begin{lem}\label{fix-bound-3}\cite{HS}
Suppose there exists $\Delta>0,M>0$ such that $Z_{<\alpha_{1,n},\beta_{i,n}>}>\Delta$ and $M^{-1}<T_{\alpha_{1,n}}<M$ for all $n.$ 
Let $k_{i,n},l_{i,n}$ to be any integers such that $T_{{\alpha_{1,n}}^{k_{i,n}}},T_{{\alpha_{1,n}}^{l_{i,n}}}<M$.
Then for at least one $J\in\{0,1\}$, 
and any integers $k'_{i,n}, l'_{i,n} $ with $|(k_{i,n}-k'_{i,n})|+|(l_{i,n}-l'_{i,n})|=J$, we have
$|z_{\alpha^{k'_{i,n}}_{1,n}\beta_{i,n}\alpha^{l'_{i,n}}_{1,n},+}-z_{\alpha^{k'_{i,n}}_{1,n}\beta_{i,n}\alpha^{l'_{i,n}}_{1,n},-}|
\ge\frac{\sigma}{\mathfrak{D}_n|\tr(\beta_{i,n})|}$,
for all $n$ sufficiently large. 
\end{lem} 
The following observation can be seen from the proof of Lemma \ref{fix-trace} \cite{HS}.
For a more rough upper bound of fixed points of $\alpha^{k_{i,n}}_{1,n}\beta_{i,n}\alpha^{l_{i,n}}_{1,n}$ to 
$\zeta_{\alpha^{k_{i,n}}_{1,n}\beta_{i,n}\alpha^{l_{i,n}}_{1,n}},$ and $\eta_{\alpha^{k_{i,n}}_{1,n}\beta_{i,n}\alpha^{l_{i,n}}_{1,n}},$ then we can relax the conditions in Lemma \ref{fix-trace}
and state as follows,
\begin{remc}\label{4B}
Suppose that $|z_{\beta_{i,n},+}-z_{\beta_{i,n},-}|<c$ and $|c_{i,n}\lambda_{1,n}^{l_{i,n}-k_{i,n}}|\to\infty.$ Then
$\{z_{\alpha^{k_{i,n}}_{1,n}\beta_{i,n}\alpha^{l_{i,n}}_{1,n},+},z_{\alpha_{1,n}^{k_{i,n}}\beta_{i,n}\alpha_{1,n}^{l_{i,n}},-}\} \to
\{\zeta_{\alpha^{k_{i,n}}_{1,n}\beta_{i,n}\alpha^{l_{i,n}}_{1,n}},\eta_{\alpha_{1,n}^{k_{i,n}}\beta_{i,n}\alpha_{1,n}^{l_{i,n}}}\}.$
\end{remc}
\section{Sufficient Conditions for $k$-generated}
Given a sequence of Schottky groups with decreasing Hausdorff diemnsions, we will state and prove a set of conditions
that will be sufficient for the given sequence to contain a subsequence which will be classical Schottky groups for small enough Hausdorff dimensions.\par
Denote the unit ball model of hyperbolic $3-$space by $(\mathbb{B}, \dis_B).$
Set $\pi:\mathbb{B}\longrightarrow \mathbb{H}^3$ the stereographic hyperbolic isometry. 
Let $\G_n$ to be a sequences of $k-$generated Schottky groups.
\par
For a loxodromic element $\alpha$ of $\Psl(2,\mathbb{C})$ acting on the unit ball $\mathbb{B}$ model of hyperbolic $3$-space,
denote by $S_{\alpha,r}$ and $S_{\alpha^{-1},r}$ the isometric spheres of
Euclidean radius $r$ of $\alpha.$ We set $\lambda_{\pi_*(\alpha)}$ as the multiplier of $\pi_*(\alpha)$ in 
the upper space model $\mathbb{H}^3.$
For $R>0,$ set $C_R$ as the circle in $\mathbb{C}$ about origin of radius $R.$ We will denote the north pole of $\partial\mathbb{B}$ by ${\bf e}=(0,0,1).$
\begin{prop}\label{proj}\cite{HS}
Let $\alpha$ be a loxodromic element of $\Psl(2,\mathbb{C})$ acting on $\mathbb{B},$
with axis passing through the origin and fixed points on north and south poles. Then
$\pi(S_{\alpha,r}\cap\partial\mathbb{B}),\pi(S_{\alpha^{-1},r}\cap\partial\mathbb{B})$ maps to
$C_{\frac{1}{\lambda_{\pi_*(\alpha)}}},C_{\lambda_{\pi_*(\alpha)}}.$
\end{prop}

\begin{lem}\label{classical}
Let $<\alpha_{1,n},\beta_{2,n},...,\beta_{k,n}>$ be generators for Schottky groups $\G_n$ in the upper-half space model $\mathbb{H}^3$ 
with $|\tr(\beta_{i,n})|\to\infty$ and 
$\alpha_n=\bigl(\begin{smallmatrix} \lambda_{1,n} & 0\\0 & \lambda^{-1}_{1,n}\end{smallmatrix}\bigr)$, $|\lambda_{1,n}|>1$. Suppose one of the following set of conditions
holds:\par
there exists $\Lambda>1$, such that for large $n,$ we have $|\lambda_{1,n}|<\Lambda$ and, 
\begin{itemize}
\item
$|\lambda_{1,n}|^{-1}<|z_{\beta_{i,n},l}|\le|z_{\beta_{i,n},u}|<|\lambda_{1,n}|,$ and
\item 
\[\liminf_n\left\{\frac{1}{(|z_{\beta_{i,n},u}|-|\lambda_{1,n}|)|\tr(\beta_{i,n})|},\frac{1}{(|z_{\beta_{i,n},l}|-|\lambda_{1,n}|^{-1})|\tr(\beta_{i,n})|}
\right\}=0,\]
\item
\[|z_{\beta_{{i,n},\pm}}-z_{\beta_{{i',n},\pm}}||\tr(\beta_{i,n})|\to\infty;\quad |\tr(\beta_{i,n})|\le|\tr(\beta_{i',n})|,\quad i\not=i'.\]
\end{itemize}
or there exists $\kappa>0$ and for large $n,$ we have $|\lambda_{i,n}|>\kappa$ and,
\begin{itemize}
\item
$\kappa^{-1}<|z_{\beta_{1,n},l}|\le|z_{\beta_{1,n},u}|<\kappa,$ and
\item 
\[\liminf_n\left\{\frac{1}{(|z_{\beta_{i,n},u}|-\kappa)|\tr(\beta_{i,n})|},\frac{1}{(|z_{\beta_{i,n},l}|-\kappa^{-1})|\tr(\beta_{1,n})|}
\right\}=0,\]
\item
\[|z_{\beta_{{i,n},\pm}}-z_{\beta_{{i',n},\pm}}||\tr(\beta_{i,n})|\to\infty;\quad |\tr(\beta_{i,n})|\le|\tr(\beta_{i',n})|,\quad i\not=i'.\]

\end{itemize}
then there exists a subsequence such that for $m$ large, $\pi^{-1}<\alpha_{1,n_m},\beta_{2,n_m},...,\beta_{k,n_m}>\pi$ 
are classical generators for $\G_{n_m}$ in the unit ball model $\mathbb{B}$.
\end{lem}
\begin{proof}
Suppose that there exists some subsequence $<\alpha_{1,n_i},\beta_{2,n_i},...,\beta_{k,n_i}>$ which satisfies the first given set of above conditions.
Let us first assume for large $i,$ $\min_{2\le j\le k}\{|z_{\beta_{{j,n_i},u}}-z_{\beta_{{j,n_i},l}}|\}>\delta>0.$  \par
Let $r_i,\rho_{j,i}$  be the Euclidean radii of isometric spheres of generator, $\pi^{-1}\alpha_{1,n_i}\pi$ and 
$\pi^{-1}\beta_{j,n_i}\pi$ respectively for $2\le j\le k.$ 
Since that $4\cosh(T_{\pi^{-1}\beta_{j,n_i}\pi})\ge|\tr^2(\pi^{-1}\beta_{j,n_i}\pi)|$ implies
that there exists some $c'>0$ such that $e^{T_{\pi^{-1}\beta_{j,n_i}\pi}}\ge c'|\tr^2(\pi^{-1}\beta_{j,n_i}\pi)|.$
By $\rho^{-1}_{j,i}=\cosh\dis(o,\mathcal{L}_{\pi^{-1}\beta_{j,n_i}\pi})\sinh(\frac{1}{2}T_{\pi^{-1}\beta_{j,n_i}\pi})$ 
(\cite{Berdon} p$175$), 
we have that $\rho^{-1}_{j,i}\ge\sinh(\frac{1}{2}T_{\pi^{-1}\beta_{j,n_i}\pi})$ so for large $i$ there exists some $c>0$ with that
$\rho^{-1}_{j,i}\ge ce^{\frac{1}{2}T_{\pi^{-1}\beta_{j,n_i}\pi}}.$
Therefore
there exists some $\delta_1>0$ such that $\rho_{j,i}\le\delta_1|\tr(\beta_{j,n_i})|^{-1}$ for large $i.$
Since for $z,w\in\mathbb{C}$,
\[|\pi^{-1}(z)'|=\frac{2}{|z+{\bf e}|^2}\quad\quad\text{and},\]
\[|\pi^{-1}(z)-\pi^{-1}(w)|=|\pi^{-1}(z)'|^{1/2}|\pi^{-1}(w)'|^{1/2}|z-w|.\]
This implies for large $i,$ and $x_{j,i}\in C_{|z_{\beta_{j,n_i},u}|}$ and $y_{j,i}\in C_{|\lambda_{1,n_i}|},$
\begin{eqnarray*}
\frac{\rho_{j,i}}{|\pi^{-1}x_{j,i}-\pi^{-1}y_{j,i}|}&\le& \frac{2\delta_1|x_{j,i}+{\bf e}||y_{j,i}+{\bf e}|}{|\tr(\beta_{j,n_i})||x_{j,i}-y_{j,i}|}.
\end{eqnarray*}
Since $|\lambda_{1,n_i}|<\Lambda,$ there exists $\delta_2>0,$ such that $|x_{j,i}+{\bf e}||y_{j,i}+{\bf e}|<\delta_2$, and
\begin{eqnarray*}
\lim_i\frac{\rho_{j,i}}{|\pi^{-1}x_{j,i}-\pi^{-1}y_{j,i}|}
&\le&\lim_i\frac{2\delta_1\delta_2}{|\tr(\beta_{j,n_i})||x_{j,i}-y_{j,i}|}=0\quad\mbox{for}\quad 2\le j\le k.
\end{eqnarray*}
Similarly there exists $\delta_3>0,$ such that for $w_{j,i}\in C_{|z_{\beta_{j,n_i},l}|}$ and $z_{j,i}\in C_{|\lambda_{1,n_i}|^{-1}},$
\begin{eqnarray*}
\lim_i\frac{\rho_{j,i}}{|\pi^{-1}w_{j,i}-\pi^{-1}z_{j,i}|}
&\le&\lim_i\frac{2\delta_1\delta_3}{|\tr(\beta_{j,n_i})||w_{j,i}-z_{j,i}|}=0\quad\mbox{for}\quad 2\le j\le k.
\end{eqnarray*}
Also there exists $\delta_4>0,$ such that for $j\not=i$ and $f_{j,i}\in C_{|z_{\beta_{j,n_i},u}|}\cup C_{|z_{\beta_{j,n_i},l}|}$ and
$g_{j',i}\in C_{|z_{\beta_{j',n_i},u}|}\cup C_{|z_{\beta_{j',n_i},l}|}$ and $|\tr(\beta_{j,i})|\le|\tr(\beta_{j',i})|,$ 
\begin{eqnarray*}
\lim_i\frac{\rho_{j,i}+\rho_{j',i}}{|\pi^{-1}f_{j,i}-\pi^{-1}g_{j',i}|} &\le&\frac{2\rho_{j,i}}{|\pi^{-1}f_{j,i}-\pi^{-1}g_{j',i}|}\\
&\le&\lim_i\frac{2\delta_1\delta_4}{|\tr(\beta_{j,n_i})||f_{j,i}-g_{j',i}|}=0\quad\mbox{for}\quad 2\le j\le k.
\end{eqnarray*}

Therefore it follows for large $i$ and Proposition \ref{proj} we have the isometric spheres 
$S_{\pi^{-1}\alpha_{1,n_i}\pi,r_i},S_{\pi^{-1}\alpha^{-1}_{1,n_i}\pi,r_i},$
$S_{\pi^{-1}\beta_{j,n_i}\pi,\rho_i},S_{\pi^{-1}\beta^{-1}_{j,n_i}\pi,\rho_i}$ and 
$S_{\pi^{-1}\beta_{j,n_i}\pi,\rho_i},S_{\pi^{-1}\beta^{-1}_{j,n_i}\pi,\rho_i},$ $S_{\pi^{-1}\beta_{j',n_i}\pi,\rho_i},S_{\pi^{-1}\beta^{-1}_{j',n_i}\pi,\rho_i}$ 
are disjoint for all $2\le j<j'\le k.$
\begin{rem}\label{classical-rem-2}
We see from above proof that without the assumption of $|\lambda_{j,n_i}|<\Lambda$ then we can't put bounds on $|x_{j,i}+{\bf e}||y_{j,i}+{\bf e}|.$ 
But since $|x_{j,i}+{\bf e}||y_{j,i}+{\bf e}|\le (|z_{\beta_{j,n_i},u}|+1)(|\lambda_{1,n_i}|+1)$ for $x_{j,i}\in C_{|z_{\beta_{n_i},u}|}, y_{j,i}\in C_{|\lambda_{1,n_i}|}$ and
$|x_{j,i}+{\bf e}||y_{j,i}+{\bf e}|\le (|z_{\beta_{j,n_i},l}|+1)(|\lambda_{1,n_i}|^{-1}+1)$ for $x_{j,i}\in C_{|z_{\beta_{j,n_i},l}|}, y_{j,i}\in C_{|\lambda_{1,n_i}|^{-1}}.$ 
Therefore we can restate the condition as,
\begin{itemize}
\item
$|\lambda_{1,n}|^{-1}<|z_{\beta_{j,n},l}|\le|z_{\beta_{j,n},u}|<|\lambda_{1,n}|,$ and
\item 
\[\liminf_n\left\{\frac{(|z_{\beta_{j,n},u}|+1)(|\lambda_{1,n}|+1)}{(|z_{\beta_{j,n},u}|-|\lambda_{1,n}|)|\tr(\beta_{j,n})|},
\frac{(|z_{\beta_{j,n},l}|+1)(|\lambda_{1,n}|^{-1}+1)}{(|z_{\beta_{j,n},l}|-|\lambda_{1,n}|^{-1})|\tr(\beta_{j,n})|}
\right\}=0,\]
\item
\[|z_{\beta_{{j,n},\pm}}-z_{\beta_{{j',n},\pm}}||\tr(\beta_{j,n})|\to\infty;\quad |\tr(\beta_{j,n})|\le|\tr(\beta_{j',n})|,\quad j\not=j'.\]
\end{itemize}
\end{rem}
Next assume for some $j$ we have that $|z_{\beta_{j,n_i},u}-z_{\beta_{j,n_i},l}|\to 0.$ 
Then with this assumption we can do a much sharper estimate of the lower bounds on $\cosh\dis(J,\mathcal{L}_{\beta_{j,n_i}}),$ which is the distance between 
the point $J$ on the vertical $J$-axis and the translation axis of $\beta_{j,n_i}$ in the space $\mathbb{H}^3.$ Eventhough a much weaker lower bounds will be
sufficient for our case.\par
For any given two points $h_1=(z_1,\theta_1),h_2=(z_2,\theta_2)\in\mathbb{H}^3$ recall that the hyperbolic distance is provided by equation,
\[\cosh\dis(h_1,h_2)=\frac{|z_1-z_2|^2+|\theta_1-\theta_2|^2}{2\theta_1\theta_2}+1.\]
By $|z_{\beta_{j,n_i},u}-z_{\beta_{j,n_i},l}|\to 0,$ also 
$\frac{1}{\Lambda}<|z_{\beta_{j,n_i},l}|\le|z_{\beta_{j,n_i},u}|<\Lambda$ we can estimate $\cosh\dis(J,\mathcal{L}_{\beta_{j,n_i}})$ 
by using the above formula for $(z_1,\theta_1)\in\mathcal{L}_{\alpha_{1,n_i}}$ and $(z_2,\theta_2)\in\mathcal{L}_{\beta_{j,n_i}}.$
Since for large $i$ we have, $|z_1-z_2|\ge|z_{\beta_{j,n_i},l}|,$
$|\theta_1-\theta_2|\ge||z_{\beta_{j,n_i},l}|-\frac{1}{2}(|z_{\beta_{j,n_i},u}-z_{\beta_{j,n_i},l}|)|,$ 
$2\theta_1\theta_2\le |z_{\beta_{j,n_i},u}||z_{\beta_{j,n_i},u}-z_{\beta_{j,n_i},l}|.$
Hence for large $i$ we have,
\[\cosh\dis(J,\mathcal{L}_{\beta_{j,n_i}})\ge\frac{|z_{\beta_{j,n_i},l}|^2+(|z_{\beta_{j,n_i},l}|-\frac{1}{2}(|z_{\beta_{j,n_i},u}-z_{\beta_{j,n_i},l}|))^2}
{|z_{\beta_{j,n_i},u}||z_{\beta_{j,n_i},u}-z_{\beta_{j,n_i},l}|}+1.\]
Since $|\Lambda|^{-1}<|z_{\beta_{j,n_i},l}|\le|z_{\beta_{j,n_i},u}|<\Lambda$ and $|z_{\beta_{j,n_i},u}-z_{\beta_{j,n_i},l}|\to 0,$  we have for large $i$ 
there exits $\sigma>0$
such that 
\[\cosh\dis(J,\mathcal{L}_{\beta_{j,n_i}})\ge\frac{\sigma}{|z_{\beta_{j,n_i},u}-z_{\beta_{j,n_i},l}|}.\]
Also by $\Lambda^{-1}<|z_{\beta_{j,n_i},l}|\le|z_{\beta_{j,n_i},u}|<\Lambda$ and $|\lambda_{1,n_i}|<|\Lambda|,$ there exists $\sigma'>0$ such that
$|\pi^{-1}(z_{\beta_{j,n_i},u})'||\pi^{-1}(z_{\beta_{j,n_i},l})'|>\sigma'.$
Since,
\[ |\pi^{-1}z_{\beta_{j,n_i},u}-\pi^{-1}z_{\beta_{j,n_i},l}|=|\pi^{-1}(z_{\beta_{j,n_i},u})'||\pi^{-1}(z_{\beta_{j,n_i},l})'|
|z_{\beta_{j,n_i},u}-z_{\beta_{j,n_i},l}|,\] we have
\[|\pi^{-1}z_{\beta_{j,n_i},u}-\pi^{-1}z_{\beta_{j,n_i},l}|\ge\sigma'|z_{\beta_{j,n_i},u}-z_{\beta_{j,n_i},l}|,\] 
From $\rho^{-1}_{j,i}=\cosh\dis(o,\mathcal{L}_{\pi^{-1}\beta_{j,n_i}\pi})\sinh(\frac{1}{2}T_{\pi^{-1}\beta_{j,n_i}\pi})$ and
the above estimates, implies that for $i$ large, there exists
$\delta_5>0$ such that $\rho_{j,i}\le\delta_5|\pi^{-1}z_{\beta_{j,n_i},u}-\pi^{-1}z_{\beta_{j,n_i},l}||\tr(\beta_{j,n_i})|^{-1}.$
Hence there exists $\delta_6>0$ such that for $x_{j,i}\in C_{|z_{\beta_{j,n_i},u}|}, y_{j,i}$ $\in C_{|\lambda_{1,n_i}|},$
\begin{eqnarray*}
\lim_i\frac{\rho_{j,i}}{|\pi^{-1}x_{j,i}-\pi^{-1}y_{j,i}|}
&\le&\lim_i\frac{\delta_6|\pi^{-1}z_{\beta_{j,n_i},u}-\pi^{-1}z_{\beta_{j,n_i},l}|}{|\tr(\beta_{j,n_i})||x_{j,i}-y_{j,i}|}=0.
\end{eqnarray*}
Similarly there exists $\delta_7>0,$ such that for $w_{j,i}\in C_{|z_{\beta_{j,n_i},l}|}$ and $z_{j,i}\in C_{|\lambda_{1,n_i}|^{-1}},$
\begin{eqnarray*}
\lim_i\frac{\rho_{j,i}}{|\pi^{-1}w_{j,i}-\pi^{-1}z_{j,i}|}
&\le&\lim_i\frac{\delta_7|\pi^{-1}z_{\beta_{j,n_i},u}-\pi^{-1}z_{\beta_{j,n_i},l}|}{|\tr(\beta_{j,n_i})||w_{j,i}-z_{j,i}|}=0.
\end{eqnarray*}
As before there exists $\delta_8>0,$ such that for $j\not=i$ and $f_{j,i}\in C_{|z_{\beta_{j,n_i},u}|}\cup C_{|z_{\beta_{j,n_i},l}|}$ and
$g_{j',i}\in C_{|z_{\beta_{j',n_i},u}|}\cup C_{|z_{\beta_{j',n_i},l}|}$ and $|\tr(\beta_{j,i})|\le|\tr(\beta_{j',i})|,$ 
\begin{eqnarray*}
\lim_i\frac{\rho_{j,i}+\rho_{j',i}}{|\pi^{-1}f_{j,i}-\pi^{-1}g_{j',i}|} &\le&\frac{2\rho_{j,i}}{|\pi^{-1}f_{j,i}-\pi^{-1}g_{j',i}|}\\
&\le&\lim_i\frac{2\delta_1\delta_4}{|\tr(\beta_{j,n_i})||f_{j,i}-g_{j',i}|}=0\quad\mbox{for}\quad 2\le j\le k.
\end{eqnarray*}
It follows from above calculations and Proposition \ref{proj} one have that for large $i,$ $S_{\pi^{-1}\alpha_{1,n_i}\pi,r_{1,i}},$ $S_{\pi^{-1}\alpha^{-1}_{1,n_i}\pi,r_{1,i}}$
are disjointed from $S_{\pi^{-1}\beta_{j,n_i}\pi,\rho_{j,i}},$ $S_{\pi^{-1}\beta^{-1}_{j,n_i}\pi,\rho_{j,i}}$ also since $|\tr(\beta_{j,n_i})|\to\infty$
implies $S_{\pi^{-1}\beta_{j,n_i}\pi,\rho_{j,i}}$ and $S_{\pi^{-1}\beta^{-1}_{j,n_i}\pi,\rho_{j,i}}$ are disjointed for $i$ is large enough,
and additionally we have $S_{\pi^{-1}\beta_{j,n_i}\pi,\rho_{j,i}}\cup S_{\pi^{-1}\beta^{-1}_{j,n_i}\pi,\rho_{j,i}}$ and
$S_{\pi^{-1}\beta_{j',n_i}\pi,\rho_{j,i}}\cup S_{\pi^{-1}\beta^{-1}_{j',n_i}\pi,\rho_{j,i}}$ are also disjointed for $2\le j< j'\le k.$ Therefore
we get first part of our lemma.\par
For our second part of the lemma it is easily seen it can be proved in the same way as first part.
\end{proof}
\begin{rem}\label{classical-rem}
During the course of our proof we realize that when $|\lambda_{1,n}|\to 1,$ then we can certainly relax our first set of conditions for $2\le j\le k$ in the above lemma:
\begin{itemize}
\item
$|\lambda_{1,n}|^{-1}<|z_{\beta_{j,n},l}|\le|z_{\beta_{j,n},u}|<|\lambda_{1,n}|,$ and
\item 
\[\liminf_n\left\{\frac{|z_{\beta_{j,n},l}-z_{\beta_{j,n},u}|}{(|z_{\beta_{j,n},u}|-|\lambda_{1,n}|)|\tr(\beta_{j,n})|},\frac{|z_{\beta_{j,n},l}-z_{\beta_{j,n},u}|}{(|z_{\beta_{j,n},l}|-|\lambda_{1,n}|^{-1})|\tr(\beta_{j,n})|}
\right\}=0.\]
\end{itemize}
\end{rem}
Next we consider when $|\tr(\beta_{j,n})|<C$ for all $n.$ In this case we can't use the isometric circles  but instead we can employ the following Proposition
to get disjoint circles for $\beta_{i,n}.$
\begin{prop}\label{non-iso}
For any given loxodromic transformation $\gamma$ with fixed points $\not=0,\infty$ and
mutiplier $\lambda^2_\gamma$
we can find disjoint circles $\mathcal{S}_{o,r},\mathcal{S}_{o',r'}$ of center $o$ radius $r$ and center $o'$ radius $r'$ respectively such that, 
\[\gamma(\text{interior}(\mathcal{S}_{o,r}))\cap\text{interior}(\mathcal{S}_{o',r'})=\emptyset,\quad\text{and}\quad
r+r'=|z_{\gamma,+}-z_{\gamma,-}|\frac{2|\lambda_\gamma|}{|\lambda_\gamma|^2-1}.\]
Note that since $|\lambda_\g|>1,$ so by this equality for $r+r'$ we have a upper bound as, $r+r'<|z_{\gamma,-}-z_{\gamma,+}|\frac{|\lambda_\g|+1}{|\lambda_\g|-1}.$
\end{prop}
\begin{proof}
We first conjugates $\gamma$ into
Mobius transformation $\gamma'$ have fixed points $\{0,\infty\}.$  First we look at circles $\mathcal{S}_{0,|\lambda_\gamma|^{-1}},\mathcal{S}_{0,|\lambda_\gamma|}.$
The transformation $\phi(x)=\frac{x-1}{x+1}$ maps the fixed points of $\gamma'$ consists of $\{0,\infty\}$ to $\{-1,1\}$ as fixed points respectively.
Additionally it also maps $\mathcal{S}_{0,|\lambda_\gamma|^{-1}},\mathcal{S}_{0,|\lambda_\gamma|}$ to $S_{z_1,r_1}, S_{z'_1,r'_1}$ respectively. 
Now we will use basic
formulas to determine the values of $z_1,z'_1,r_1,r'_1$ (see page $91$ of \cite{MCD}). By \cite{MCD} we have,
\begin{eqnarray*}
r_1&=&\left|\frac{-|\lambda_\g|^{-2}-1}{-|\lambda_\g|^{-2}+1}-\frac{|\lambda_\g|^{-1}-1}{|\lambda_\g|^{-1}+1}\right|\\
&=&\frac{2|\lambda_\g|}{|\lambda_\g|^2-1}\\
r'_1&=&\left|\frac{-|\lambda_\g|^2-1}{-|\lambda_\g|^2+1}-\frac{|\lambda_\g|-1}{|\lambda_\g|+1}\right|\\
&=&\frac{2|\lambda_\g|}{|\lambda_\g|^2-1}
\end{eqnarray*}
This gives, 
\[r_1+r'_1=\frac{4|\lambda_\g|}{|\lambda_\g|^2-1}.\]
The distance between the centers is,
\[|z_1-z'_1|=\left|\frac{-|\lambda_\g|^{-2}-1}{-|\lambda_\g|^{-2}+1}-\frac{-|\lambda_\g|^2-1}{-|\lambda_\g|^2+1}\right|
=2\frac{|\lambda_\g|^2+1}{|\lambda_\g|^2-1}.\]
Since $(|\lambda_\g|^2+1)-2|\lambda_\g|=(|\lambda_\g|-1)^2>0,$ implies $\mathcal{S}_{z_1,r_1},\mathcal{S}_{z'_1,r'_1}$ are disjoint.
By conjugating $\phi\gamma'\phi^{-1}$ with $\psi(x)=x+\frac{z_{\gamma,+}+z_{\gamma,-}}{z_{\gamma,+}-z_{\gamma,-}}$ we map the fixed points 
$\{-1,1\}$ to $\left\{\frac{2z_{\g,-}}{z_{\g,+}-z_{\g,-}},\frac{2z_{\g,+}}{z_{\g,+}-z_{\g,+}}\right\}.$ 
Since $\psi(x)$ is a translation transformation (i.e euclidian isometry), the circles are then mapped to circles $\mathcal{S}_{z_2,r_2},\mathcal{S}_{z'_2,r'_2}$
which have same radius as before and also preserves the disjointness as before. Then by conjugating $\psi\phi\g\phi^{-1}\psi^{-1}$ with $\theta(x)=\frac{z_{\g,+}-z_{\g,-}}{2}$
sends $\left\{\frac{2z_{\g,-}}{z_{\g,+}-z_{\g,-}},\frac{2z_{\g,+}}{z_{\g,+}-z_{\g,+}}\right\}$ to points $\{z_{\g,-},z_{\g,+}\}$ and also maps the circles to
$\mathcal{S}_{z_3,r_3},\mathcal{S}_{z'_3,r'_3}.$ We have that $r_3,r'_3=|z_{\g,+}-z_{\g,-}|\frac{|\lambda_\g|}{|\lambda_\g|^2-1}$ and also the disjointness is preserved.
Also we have $\g=\theta\psi\phi\g'\phi^{-1}\psi^{-1}\theta^{-1}.$
Therefore we get that the sum of the radius of our disjointed circles are given by,
\[r_3+r_3'=|z_{\gamma,-}-z_{\gamma,+}|\frac{2|\lambda_\g|}{|\lambda_\g|^2-1}\]
\end{proof}
For a given $\g$ we will denote the set of disjoint disks $D_{\g,+}, D_{\g,-},$ such that 
$D_{\g,+}=\p\mathcal{S}_{o,r}, D_{\g,+}=\p\mathcal{S}_{o',r'}$ with 
the convention that $z_{\g,+}\in D_{\g,+}$ and $z_{\g,-}\in D_{\g,-}.$ We also set $\mathcal{R}_{\g}=|z_{\g,-}-z_{\g,+}|\frac{1}{|\lambda_\g|-1}$
for rest of the sections.
\begin{prop}\label{bounded-classical}
Given a sequence $<\alpha_n,\beta_n>$ with $\mathfrak{D}_n\to 0$ such that $|\tr(\alpha_n)|<C.$ Assume that $\alpha_n$ have fixed points $0,\infty$ and
$|\lambda_{\alpha_n}|^{-1}\le|\eta_{\beta_n}|\le|\zeta_{\beta_n}|\le |\lambda_{\alpha_n}|$ then either $<\alpha_n,\beta_n>$ or $<\alpha_n,\alpha^{-1}_n\beta_n>$
is classical Schottky group for large $n.$
\end{prop}
\begin{proof}
This is a direct consequence Theorem $1.1$ of \cite{HS}. And the particular choice of generator is given by the proof of Theorem $6.1$ of \cite{HS}.
\end{proof}
From Proposition \ref{bounded-classical} statement let us assume that its true in $k-1$-rank as well, and we will from the $k-1$-rank prove that it is also
true in rank $k$. To do so we will first state the $k-1$-version of the above Proposition.
\begin{prop}\label{k-1}
Given a sequence of $k-1$-generators $<\alpha_{1,n},...,\alpha_{{k-1},n}>$ with $\mathfrak{D}_n\to 0$ such that $|\tr(\alpha_{1,n})|<C.$ Assume that
$\alpha_{1,n}$ have fixed points $0,\infty$ and $|\lambda_{\alpha_{1,n}}|^{-1}\le|\eta_{\alpha_{j,n}}|\le|\zeta_{\alpha_{j,n}}|\le|\lambda_{\alpha_{1,n}}|$
then, $<\alpha_{1,n},\alpha^{a_1}_{1,n}\alpha_{2,n},...,\alpha^{a_{k-1}}_{1,n}\alpha_{j,n}>$ is classical Schottky group for large $n$ and some $a_j\in\{0,1\}.$
\end{prop}
Denote $\mathcal{G}_{\alpha_n;\beta_n}=\inf\{|z_{\beta_n,\pm}-z|: \quad z\in D_{\alpha_n,\pm}\}.$
Note that for a sequence of generators $<\alpha_n,\beta_n>$ with $z_{\alpha_n,\pm},z_{\beta_n,\pm}\in\mathbb{C}$ and $z_{\beta_n,\pm}\not\in D_{\alpha_n,\pm}.$ Suppose
that $\mathcal{G}_{\alpha_n;\beta_n}\mathcal{R}^{-1}_{\alpha_n}\to\infty,$ then $<\alpha_n,\beta_n>$ generates a classical Schottky group for large $n.$

\section{Degenerate types }
This section we classify all possible behaviors such that a given sequence $\G_n$ of Schottky groups with $D_{\G_n}\to 0$ will not eventually contain a subsequence 
of classical Schottky groups. The \emph{obstructions} that a sequence of $\G_n$ to contain any classical Schottky group will be called \emph{degeneracies.} \par
Analyzing degeneracies rely
on the analysis of the behaviors of \emph{centers of isometric circles} of generators. To do this analysis, we must characterize all possible dynamics (here dynamic refers to behaviors with respect to decreasing Hausdorff dimension) of sequence of generators according to behaviors of their centers of circles. To do so we first consider dichotomized behaviors and then utilize them to show general result. More precisely, we show that each of the dichotomized degenerate behavior will lead into a classical Schottky group for sufficiently large $n.$
The dichotomization of degeneracies are into four types as: 
TypeI, Type II,  Type III, Type IV given in the following. In addition, each of these types of degeneracies will have several sub-types which are stated later on separately.
\par
\begin{defn}
For $\tau>0$, we set: 
\[\mf{J}_k(\tau):=\{[\G]\in\mf{J}_{k}| Z_\G>\tau, \exists\text{ for some } \G\in [\G]\}.\] 
\end{defn}
Recall that $\mf{J}_{k}$ denotes Schottky space
of rank $k$. Here $[\Gamma]$ denotes equivalence class of $\G,$ and $\mf{J}_k(\tau)$ is just collection of Schottky groups that have some generating set with 
uniform lower bounds on $Z_\G$ by $\tau$ under conjugation. \par
Let us assume that there exists a sequence $\{[\G_n]\}\subset\mf{J}_k(\tau)$
of nonclassical Schottky groups with strictly decreasing $\mathfrak{D}_n\to 0$. Set $\G_n=<\alpha_{1,n},\alpha_{2,n},...,\alpha_{k,n}>$ with that
$Z_{<\alpha_{1,n},\alpha_{2,n},...,\alpha_{k,n}>} >\tau$. We arrange the generators so  
$|\tr(\alpha_{i,n})|\le |\tr(\alpha_{i+1,n})|$. There are two possibilities: 
\begin{itemize}
\item
$(I)$ There exists a subsequence 
such that,  $|\tr(\alpha_{1,n})|\to\infty$
\item
$(II)$ $|\tr(\alpha_{1,n})|<M$, for some $M>0.$ 
\end{itemize}
Case $(I)$ is trivial. Since $|\tr(\alpha_{i,n_j})|\to \infty$ as $n\to\infty$ for all $1\le i\le k$,
it follows from $Z_{<\alpha_{i,n_j},...,\alpha_{k,n_j}>}>\tau$, there must exists $N$ such that
$<\alpha_{i,n_j},...,\alpha_{i,n_j}>$ becomes classical Schottky groups for $j>N$. A contradiction.\par
Now we consider case $(II)$. \\We work in upper space model $\mathbb{H}^3$. Conjugate $<\alpha_{1,n},...,\alpha_{k,n}>$ by 
a Mobius transformation into 
$\alpha_{1,n}=\bigl(\begin{smallmatrix} \lambda_{\alpha_{1,n}} & 0\\0 & \lambda^{-1}_{\alpha_{1,n}}\end{smallmatrix}\bigr)$
with $|\lambda_{\alpha_{1,n}}|>1$. Denote 
$\alpha_{j,n}=\bigl(\begin{smallmatrix} a_{j,n} & b_{j,n}\\c_{j,n} & d_{j,n}\end{smallmatrix}\bigr)$. Since $|\tr(\alpha_{1,n})|<M$
implies $|\lambda_{\alpha_{1,n}}|<M'$ for some $M'>0,$ it follows from Proposition \ref{trace} and $\mathfrak{D}_n\to 0$, we have $|\tr(\alpha_{j,n)}|\to\infty.$
By replacing $\alpha_{i,n}$ with $\alpha^{-1}_{i,n}$ if necessary, we can assume $|\zeta_{\alpha_{i,n}}|\le|\eta_{\alpha_{i,n}}|, 2\le j\le k.$\par 
Since $Z_{<\alpha_{1,n},...,\alpha_{k,n}>}>\tau$, there exists $\Delta_1,\Delta_2,\Delta_3,\Delta_4>0$ such that, 
$\Delta_1<|z_{\alpha_{j,n},l}|\le|z_{\alpha_{j,n},u}|<\Delta_2,$ and
$\Delta_3<|z_{\alpha_{j,n},l}-z_{\alpha_{j,n},u}|<\Delta_4.$
It follows from Lemma \ref{fix-trace}, $\Delta_1<|\zeta_{\alpha_{j,n}}|\le|\eta_{\alpha_{j,n}}|<\Delta_2$, and
$\Delta_3<\lim_n|\zeta_{\alpha_{j,n}}-\eta_{\alpha_{j,n}}|<\Delta_4.$
\par
For each $n$ and $i\ge 2$, choose integers $k_{i,n},l_{i,n}$ such that:
\[1\le|\zeta_{\alpha_{i,n}}\lambda^{2k_{i,n}}_{1,n}|<|\lambda^2_{1,n}|, \quad 
1\le|\eta_{\alpha_{i,n}}\lambda^{2l_{i,n}}_{1,n}|<|\lambda^2_{1,n}|.\]
We consider the generating set
$<\alpha_{1,n},..., \alpha^{k_{i,n}}_{1,n}\alpha_{i,n}\alpha^{l_{i,n}}_{1,n} ,...,\alpha^{k_{k,n}}_{1,n}\alpha_{k,n}\alpha^{l_{k,n}}_{1,n}>$. Denote
these new generators by $\beta_{i,n}$ for $i\ge 2.$\par
By passing to subsequence and considering inverses if necessary, we have can assume 
$|\eta_{\beta_{i,n}}\lambda^{2l_{i,n}}_{1,n}|\le|\zeta_{\beta_{i,n}}\lambda^{2k_{i,n}}_{1,n}|.$\par 
\begin{defn}[standard form]\label{standard}
Given a set of generators $S_\G=\{\alpha_{1},...,\alpha_{k}\}$ of Schottky group $\G$, we say $S_\G$ is of \emph{standard form} if (up to conjugation by Mobius transformation), $z_{\alpha_{1},\pm}=\{0,\infty\}$ and
\[1\le|\zeta_{\alpha_{i}}|<|\lambda^2_{\alpha_1}|, \quad 
1\le|\eta_{\alpha_{i}}|<|\lambda^2_{\alpha_1}|, 2\le i\le k.\]
\end{defn}

\begin{defn}[$k-1$ classical]\label{k-1-classical}
Given $\G$ rank-$k$ Schottky group we say $\G$ is $k-1$ \emph{classical} if every rank $k-1$ subgroup of $G$ is a classical Schottky group.
\end{defn}
In Section \ref{t-space} we will show that, given a sequence of $\G_n$ Schottky group which is $k-1$ classical and $\mathfrak{D}_{\G_n}\to 0$, we can always choose a subsequence of
generators $S_{\G_n}$ such that it's $k-1$ classical of standard form. \par
Another $k-1$ classical sequence (called normal sequence) one can choose is given next which is based on location of fixed points. The idea is that when the Hausdorff dimension is sufficiently small then both normal sequence and standard form sequence coincide.  
\begin{defn}\label{normal-def}
Let $\G_{n}$ be a sequence of Schottky groups with generating set $S_{\G_{n}}=<\alpha_{1,n},\alpha_{2,n},...,\alpha_{k,n}>$ such that
$\alpha_{1,n}$ have fixed points $0,\infty.$ We say $S_{\G_{n}}$ is a normal sequence of generating sets if (up to conjugation by Mobius transformation),
all isometric circles of $\alpha_{j,k}$ are strictly bounded between $1$ and $|\lambda_{\alpha_{1,n}}|^2$ for $2\le j\le k$ and large $n$.
\end{defn}

\begin{prop}[Normal sequence]\label{normal-exists}
Let $\G_n$ be a sequence of Schottky groups which is $k-1$ classical with $\mathfrak{D}_{\G_n}\to 0.$ Then there exists a normal subsequence of generating set
$S_{\G_n}$ (we use same index for subsequence ).
\end{prop}
\begin{proof}
Since $\G_n$ is $k-1$ classical, we can choose a sequence of generating set $S_n$ so that every pair $\alpha_{1,n},\alpha_{j,n}$ is classical for $1\le j\le k.$ We assume by same convention that $\alpha_{1,n}$ is of minimal $|\mbox{trace}|.$
Conjugate $\alpha_{1,n}$ to have fixed points $0,\infty$ and denote the new sequence by the same notation. It follows from Corollary \ref{trace-rank-dis} and Remark \ref{4A}, isometric circles of $\alpha_{j,n}$ have radius $\to 0$ for all $2\le j\le k.$ Since all fixed points are strictly bounded between $1$ and $|\lambda_{\alpha_{1,n}}|^2,$ the result follows.
\end{proof}

The behaviors of elements of our sequence of generators are done according to location of their centers of circles of $\alpha_{j,n}$ for $2\le j\le k.$ Lets assume we have a sequence of generating sets which is $k-1$ classical of standard form, then we can
categorize the obstructions of the sequence to contain classical Schottky groups by disjointness of circles of $\alpha_{j,n}$ for $2\le j\le k$ to circles of $\alpha_{1,n}.$ 
\begin{defn}[Degenerate types]\label{cases}
Using above notations, let $\G_n$ be a sequence of rank $k$ Schottky groups which is $k-1$ classical with $\mathfrak{D}_n\to 0$. Let $S_n$ be a sequence of generating sets which is $k-1$ classical of standard form.  Degenerate types are given by:
\begin{itemize}
\item 
Type I: Double boundary degeneracy  
 \[||\zeta_{\alpha_{i,n}}\lambda^{2k_{i,n}}_{1,n}|-|\eta_{\alpha_{i,n}}\lambda^{2l_{i,n}}_{1,n}||\to 0.\]
\item
Type II: Single boundary degeneracy 
\[||\zeta_{\alpha_{i,n}}\lambda^{2k_{i,n}}_{1,n}|-|\lambda_{1,n}|^2|\to 0.\]
\item
Type III: Elliptical degeneracy  
\[|\zeta_{\alpha_{i,n}}\lambda^{2k_{i,n}}_{1,n}|\to 1.\] 
\item 
Type IV: Bounded 
\[||\zeta_{\alpha_{i,n}}\lambda^{2k_{i,n}}_{1,n}|-|\eta_{\alpha_{i,n}}\lambda^{2l_{i,n}}_{1,n}||>C,\text{ and }
||\zeta_{\alpha_{i,n}}\lambda^{2k_{i,n}}_{1,n}|-|\lambda_{1,n}|^2|>C.\]
\end{itemize}
\end{defn}
In the next four sections we will show that each of these degenerate types will lead to classical Schottky group for large $n.$
\section{Bounded Type IV} 
We show this type of degeneracy is classical for large $n.$\par
We use the same notation index for subsequences. For large $n$ there exists 
$1<|\lambda|<|\lambda_{\alpha_{1,n}}|$, $\sigma<1$, such that  
$\max_i|\zeta_{i,n}\lambda^{2k_{i,n}}_{1,n}|\to \sigma|\lambda|^2$.
Let $\psi_n$ be the Mobius transformation that fixes $\{0,\infty\}$ 
defined by $\psi(x)=\frac{x}{\sqrt{\sigma}|\lambda|}$, $x\in\mathbb{C}.$ Let $1\le\chi\le \sigma|\lambda|^2$
be given as $\chi=\min_i\lim_n|\eta_{i,n}\lambda^{2l_{i,n}}_{1,n}|.$
Then $\min_i|\eta_{\psi\beta_{i,n}\psi^{-1}}|\to\frac{\chi}{\sqrt{\sigma}|\lambda|}$
and $\max_i|\zeta_{\psi\beta_{i,n}\psi^{-1}}|\to\sqrt{\sigma}|\lambda|.$ Let $|\tr(\beta_{i^*,n})|=\min|\tr(\beta_{i,n})|.$
It follows from Lemma \ref{fix-trace}, for $n$ large, 
\[\max\{\max_i|z_{\psi\beta_{i,n}\psi^{-1},\mp}-\eta_{\psi\beta_{i,n}\psi^{-1}}|,
\max_i|z_{\psi\beta_{i,n}\psi^{-1},\pm}-\zeta_{\psi\beta_{i,n}\psi^{-1}}|\}
<\frac{\rho}{|\tr(\beta_{i^*,n})|}.\]
Hence there exists $\rho', \rho'',\rho'''>0$ such that,
\[\left|z_{\psi\beta_{i,n}\psi^{-1},+}-z_{\psi\beta_{i,n}\psi^{-1},-}\right|>
\left||\sqrt{\sigma}|\lambda|-\frac{\chi}{\sqrt{\sigma}|\lambda|}\right|-\frac{\rho'}{|\tr(\beta_{i^*,n})|},\]
and
\[\left|\frac{1}{z_{\psi\beta_{i,n}\psi^{-1},+}}-\frac{1}{z_{\psi\beta_{i,n}\psi^{-1},-}}\right|>
\rho''\left||\frac{\sqrt{\sigma}|\lambda|}{\chi}-\frac{1}{\sqrt{\sigma}|\lambda|}\right|-\frac{\rho'''}{|\tr(\beta_{i^*,n})|}.\]
This implies that there exist $\Delta>0$ such that $Z_{\psi\beta_{i,n}\psi^{-1}}>\Delta.$
Hence by applying Proposition \ref{trace} to the generators $<\psi\alpha_{1,n}\psi^{-1},\psi\beta_{i,n}\psi^{-1}>$,
implies $|\tr(\psi\beta_{i,n}\psi^{-1})|\to\infty.$\par
Set $\kappa=\sqrt{\frac{\sigma+1}{2}}|\lambda|.$ Then for sufficiently large $n$ we have 
$\kappa<|\lambda_{1,n}|,$ $\kappa^{-1}<|z_{\psi\beta_{i,n}\psi^{-1},l}|\le
|z_{\psi\beta_{i,n}\psi^{-1},u}|<\kappa.$
And obviously,
\[ \lim_n\frac{1}{(\kappa-|z_{\psi\beta_{i,n}\psi^{-1},u}|)|\tr(\psi\beta_{i,n}\psi^{-1})|}=0,\]
\[\lim_n\frac{1}{(|z_{\psi\beta_{i,n}\psi^{-1},l}|-\kappa^{-1})|\tr(\psi\beta_{i,n}\psi^{-1})|}=0.\]
Also by $\kappa(S_{\G_n})\to\infty,$
\[|z_{\beta_{i,n},\pm}-z_{\beta_{j,n},\pm}||\tr(\beta_{i,n})|\to\infty,\quad\mbox{for}\quad |\tr(\beta_{i,n})|\le|\tr(\beta_{j,n})|\] 
Therefore, $<\psi\alpha_{1,n}\psi^{-1},\psi\beta_{2,n}\psi^{-1},...,\psi\beta_{k,n}\psi^{-1}>$ satisfies the second set of conditions of
Lemma \ref{classical}, and so by Lemma \ref{classical}, these will be classical generators for large $n$, a contradiction.

\section{Double boundary degeneracy Type I}
We show this type of degeneracy is classical for large $n.$\par
This case we need to further break into two possibilities based on whether $\alpha_{1,n}$ converging into a elliptical element or a nonellipitical element. \par
By passing to a subsequence if necessary,
we have two possibilities: 
\begin{itemize}
\item
Type I$_1$: $|\lambda_{\alpha_{1,n}}|^2-1$ is monotonically decreasing to $0$.  
\item
Type I$_2$: There exists $\lambda>1$, such that
$|\lambda_{\alpha_{1,n}}|\ge|\lambda|$ for large $n$.
\end{itemize}
\subsection{Type I$_1$} 
In the degeneration into elliptical case, we need to consider whether the collapsing of fixed points is at must faster rate relative to collapsing of fixed points to circles: $(i),(ii)$.\par
Here we either have 
\[(i)\quad\quad\limsup_n\left|| \zeta_{ \alpha^{k_{i,n}}_{1,n}\beta_{i,n}\alpha^{l_{i,n}}_{1,n}}  |-|\lambda_{\alpha_{1,n}}|^2\right|
\left|\zeta_{ \alpha^{k_{i,n}}_{1,n}\beta_{i,n}\alpha^{l_{i,n}}_{1,n}} -\eta_{ \alpha^{k_{i,n}}_{1,n}\beta_{i,n}\alpha^{l_{i,n}}_{1,n}}\right|^{-1}<\infty \]
or
\[(ii)\quad\quad\liminf_n\left||\zeta_{ \alpha^{k_{i,n}}_{1,n}\beta_{i,n}\alpha^{l_{i,n}}_{1,n}}|-|\lambda_{\alpha_{1,n}}|^2\right|
\left|\zeta_{ \alpha^{k_{i,n}}_{1,n}\beta_{i,n}\alpha^{l_{i,n}}_{1,n}}-\eta_{ \alpha^{k_{i,n}}_{1,n}\beta_{i,n}\alpha^{l_{i,n}}_{1,n}}\right|^{-1}\to\infty.\] 
\subsubsection{ Case $(ii)$} 
In this subcase, we show that under appropriate Nielsen transformation and conjugation, we can reduce to case $(i).$ Hence it would be enough just to consider subcase $(i)$, which is done next. \par
Take Mobius transformation $\psi_{n}$ defined by, 
$\psi_{i,n}(x)=\frac{x}{\zeta_{\alpha^{k_{i,n}}_n\beta_{i,n}\alpha^{l_{i,n}}_{1,n}}}.$ Set $\psi_n(x)$ be the $\psi_{i*,n}$ with maximal $|\psi_{i,n}|, 2\le i\le k.$
We conjugate $\alpha^{k_{i,n}}_{1,n}\beta_{i,n}\alpha^{l_{i,n}}_{1,n}$ by $\psi_n.$
Consider $(\psi_{n}\alpha^{k_{i,n}}_{1,n}\beta_{i,n}\alpha^{l_{i,n}}_{1,n}\psi^{-1}_{n})^{-1}.$  
Since by factor out $\zeta_{\alpha^{k_{i,n}}_{1,n}\beta_{i,n}\alpha^{l_{i,n}}_{1,n}}^2\lambda_{1,n}^{-2}$ in $(ii)$,
\[\liminf_n |\zeta_{\alpha^{k_{i,n}}_{1,n}\beta_{i,n}\alpha^{l_{i,n}}_{1,n}}|^2|\lambda_{1,n}|^{-2}\left||1-|\zeta_{\alpha^{k_{i,n}}_n\beta_{i,n}\alpha^{l_{i,n}}_{1,n}}|^{-1}|\lambda_{1,n}|^2\right|\left|\zeta^{-1}_{\alpha^{k_{i,n}}_{1,n}\beta_{i,n}\alpha^{l_{i,n}}_{1,n}}\lambda^2_{1,n}-\lambda_{1,n}^2\right|^{-1}\to\infty.\] 
\[\liminf_n \left||1-|\zeta_{\alpha^{k_{i,n}}_{1,n}\beta_{i,n}\alpha^{l_{i,n}}_{1,n}}^{-1}\lambda^2_n|\right|\left|\zeta^{-1}_{\alpha^{k_{i,n}}_{1,n}\beta_{i,n}\alpha^{l_{i,n}}_{1,n}}\lambda^2_{1,n}-\lambda_{1,n}^2\right|^{-1}\to\infty.\]
Since $\zeta_{\alpha^{k_{i,n}}_{1,n}\beta_{i,n}\alpha^{l_{i,n}}_{1,n}}^{l_{i,n}-1}\lambda^2_{1,n}= \eta_{\alpha^{k_{i,n}}_n\beta_{i,n}\alpha^{-1}_{1,n}}$ and 
$\zeta_{(\psi_n\alpha^{k_{i,n}}_{1,n}\beta_{i,n}\alpha_{1,n}^{l_{i,n}-1}\psi^{-1}_n)^{-1}} =\eta_{\psi_n\alpha^{k_{i,n}}_{1,n}\beta_{i,n}\alpha^{l_{i,n}-1}_{1,n}\psi^{-1}_n}$ we have,
\[\liminf_n \left||1-|\eta_{\alpha^{k_{i,n}}_{1,n}\beta_{i,n}\alpha^{l_{i,n}-1}_{1,n}}|^2\right|\left|\eta_{\alpha^{k_{i,n}}_{1,n}\beta_{i,n}\alpha_{1,n}^{l_{i,n}-1}}-\lambda_{1,n}^2\right|^{-1}\to\infty,\] 
giving,
\[\limsup_n\left||\zeta_{(\psi_n\alpha^{k_{i,n}}_{1,n}\beta_{i,n}\alpha_{1,n}^{l_{i,n}-1}\psi^{-1}_n)^{-1}}|-|\lambda_{1,n}|^2\right|
\left|\zeta_{(\psi_n\alpha^{k_{i,n}}_{1,n}\beta_{1,n}\alpha_{1,n}^{l_{i,n}-1}\psi^{-1}_n)^{-1}}-1\right|^{-1}<\infty.\]
The generator $(\psi_n\alpha^{k_{i,n}}_{1,n}\beta_{i,n}\alpha^{l_{i,n}-1}_{1,n}\psi^{-1}_n)^{-1}$ satisfies $(i)$. 
Hence replacing the generators if necessary
we can always assume the generators satisfy $(i)$. And without lost of generality we will assume that
$<\alpha_{1,n},...,\alpha^{k_{i,n}}_{1,n}\beta_{i,n}\alpha^{l_{i,n}}_{1,n},...>$ satisfies $(i).$\par
\text{}\\
\subsubsection{Consider $(i).$}
Here we must consider whether the collapsing is bounded from below or not: $(i_1),(i_2).$\par
In this case, we have either:
\[(i_1)\quad\quad\limsup_n\left||\zeta_{ \alpha^{k_{i,n}}_{1,n}\beta_{i,n}\alpha^{l_{i,n}}_{1,n}}|-|\lambda_{\alpha_{1,n}}|^2\right|
\left|\zeta_{ \alpha^{k_{i,n}}_{1,n}\beta_{i,n}\alpha^{l_{i,n}}_{1,n}}-\eta_{ \alpha^{k_{i,n}}_{1,n}\beta_{i,n}\alpha^{l_{i,n}}_{1,n}}\right|^{-1}>\delta>0, \]
or
\[(i_2)\quad\quad\limsup_n\left||\zeta_{ \alpha^{k_{i,n}}_{1,n}\alpha_{i,n}\alpha^{l_{i,n}}_{1,n}}|-|\lambda_{\alpha_{1,n}}|^2\right|
\left|\zeta_{ \alpha^{k_{i,n}}_{1,n}\beta_{i,n}\alpha^{l_{i,n}}_{1,n} } -\eta_{ \alpha^{k_{i,n}}_{1,n}\beta_{i,n}\alpha^{l_{i,n}}_{1,n}}\right|^{-1}=0.\]
First we present the proof for $(i_1).$\par
\subsubsection{ Consider $(i_1).$}
\begin{lem}\label{axes}
There exists $c>0$ such that, 
\[\dis(\mathcal{L}_{\alpha_{1,n}},\mathcal{L}_{\alpha^{k_{i,n}}_{1,n}\alpha_{i,n}\alpha^{l_{i,n}}_{1,n}})<\log(\frac{c}{|\lambda_{1,n}|^2-1}). \] 
\end{lem}
\begin{proof}
We first show that
$\frac{1}{|\tr(\alpha_{i,n})|(|\lambda_{1,n}|^2-1)}\to 0$. From Proposition \ref{trace} we have,
\begin{eqnarray*}
\lim\frac{1}{|\tr(\alpha_{i,n})|(|\lambda_{1,n}|^2-1)}&\le&\lim\rho\left(\frac{|\lambda_{1,n}|^{2\mathfrak{D}_n}-1}
{(|\lambda_{1,n}|^{2\mathfrak{D}_n}+3)(|\lambda_{1,n}|^2-1)^{2\mathfrak{D}_n}}\right)^{\frac{1}{2\mathfrak{D}_n}}
\end{eqnarray*}
and for large $n,$ we have $|\lambda_{1,n}|^{2\mathfrak{D}_n}-1<|\lambda_{1,n}|^2-1$ which implies that for some
$\rho'>0$,
\[\lim\frac{1}{|\tr(\alpha_{i,n})|(|\lambda_{1,n}|^2-1)}\le\lim\rho'(|\lambda_{1,n}|^2-1)^{\frac{1-2\mathfrak{D}_n}{2\mathfrak{D}_n}}=0.\]
It follows from Lemma \ref{fix-trace}, 
\begin{eqnarray*}
\left|\zeta_{ \alpha^{k_{i,n}}_{1,n}\alpha_{i,n}\alpha^{l_{i,n}}_{1,n} }-\eta_{ \alpha^{k_{i,n}}_{1,n}\alpha_{i,n}\alpha^{l_{i,n}}_{1,n} }\right|-
\rho'|\tr(\alpha_{i,n})|^{-1}\le\left|z_{ \alpha^{k_{i,n}}_{1,n}\alpha_{i,n}\alpha^{l_{i,n}}_{1,n},- }-
z_{ \alpha^{k_{i,n}}_{1,n}\alpha_{i,n}\alpha^{l_{i,n}}_{1,n},+ }\right|\\
\le
 \left|\zeta_{ \alpha^{k_{i,n}}_{1,n}\alpha_{i,n}\alpha^{l_{i,n}}_{1,n} }-\eta_{ \alpha^{k_{i,n}}_{1,n}\alpha_{i,n}\alpha^{l_{i,n}}_{1,n} }\right|+
\rho''|\tr(\alpha_{i,n})|^{-1}
\end{eqnarray*}
Since $1\le|\zeta_{ \alpha^{k_{i,n}}_{1,n}\alpha_{i,n}\alpha^{l_{i,n}}_{1,n} }|<|\lambda_{1,n}|^2,$ we have 
\[\frac{\left|z_{ \alpha^{k_{i,n}}_{1,n}\alpha_{i,n}\alpha^{l_{i,n}}_{1,n},-}-z_{ \alpha^{k_{i,n}}_{1,n}\alpha_{i,n}\alpha^{l_{i,n}}_{1,n},+}\right|}
{|\lambda_{1,n}|^2-1}\ge 
\frac{\left|\zeta_{ \alpha^{k_{i,n}}_{1,n}\alpha_{i,n}\alpha^{l_{i,n}}_{1,n} }-\eta_{ \alpha^{k_{i,n}}_{1,n}\alpha_{i,n}\alpha^{l_{i,n}}_{1,n} }\right|}
{|\lambda_{1,n}|^2-1}-\frac{\rho'}{|\tr(\alpha_{i,n})|(|\lambda_{1,n}|^2-1)}. \]
By the condition of $(i_{1})$ we have,
\[\frac{\left|\zeta_{ \alpha^{k_{i,n}}_{1,n}\alpha_{i,n}\alpha^{l_{i,n}}_{1,n} }-\eta_{ \alpha^{k_{i,n}}_{1,n}\alpha_{i,n}\alpha^{l_{i,n}}_{1,n} }\right|}
{|\lambda_{1,n}|^2-1}=
\frac{|\zeta_{ \alpha^{k_{i,n}}_{1,n}\alpha_{i,n}\alpha^{l_{i,n}}_{1,n} }-\eta_{ \alpha^{k_{i,n}}_{1,n}\alpha_{i,n}\alpha^{l_{i,n}}_{1,n} }|}
{|\lambda_{1,n}|^2-|\zeta_{ \alpha^{k_{i,n}}_{1,n}\alpha_{i,n}\alpha^{l_{i,n}}_{1,n} }|+
|\zeta_{ \alpha^{k_{i,n}}_{1,n}\alpha_{i,n}\alpha^{l_{i,n}}_{1,n} }|-1}>\frac{1}{M+1},\]
for some $M>0.$
Hence for large $n$ there exists $\kappa>0$ such that,
\[\frac{\left|z_{ \alpha^{k_{i,n}}_{1,n}\alpha_{i,n}\alpha^{l_{i,n}}_{1,n},-}-z_{ \alpha^{k_{i,n}}_{1,n}\alpha_{i,n}\alpha^{l_{i,n}}_{1,n},+}\right|}
{|\lambda_{1,n}|^2-1}>\frac{1}{M+1}-
\frac{\rho'}{|\tr(\alpha_{i,n})|(|\lambda_{1,n}|^2-1)}>\kappa.\] 
For the upper bounds we have, 
$| z_{ \alpha^{k_{i,n}}_{1,n}\alpha_{i,n}\alpha^{l_{i,n}}_{1,n},-} -z_{ \alpha^{k_{i,n}}_{1,n}\alpha_{i,n}\alpha^{l_{i,n}}_{1,n},+} |
<2|\lambda_{1,n}|+\rho''|\tr(\alpha_{i,n})|^{-1}.$ 
Note that 
\[\dis(\mathcal{L}_{\alpha_{1,n}},\mathcal{L}_{\alpha^{k_{i,n}}_{1,n}\alpha_{i,n}\alpha^{l_{i,n}}_{1,n}})=
\inf\left\{\dis(h_1,h_2)|h_1\in\mathcal{L}_{\alpha_{1,n}},h_2\in\mathcal{L}_{\alpha^{k_{i,n}}_{1,n}\alpha_{i,n}\alpha^{l_{i,n}}_{1,n}}\right\}.\]
Set $h_j=(z_j,\theta_j); j=1,2$ then 
for a upper bound we can take 
\[(z_1,\theta_1)=(0, |\zeta_{ \alpha^{k_{i,n}}_{1,n}\alpha_{i,n}\alpha^{l_{i,n}}_{1,n},l } |+
\frac{1}{2}|z_{ \alpha^{k_{i,n}}_{1,n}\alpha_{i,n}\alpha^{l_{i,n}}_{1,n},u} -z_{ \alpha^{k_{i,n}}_{1,n}\alpha_{i,n}\alpha^{l_{i,n}}_{1,n},l} |),\]
\[(z_2,\theta_2)=(\frac{1}{2}(z_{ \alpha^{k_{i,n}}_{1,n}\alpha_{i,n}\alpha^{l_{i,n}}_{1,n},u} +z_{ \alpha^{k_{i,n}}_{1,n}\alpha_{i,n}\alpha^{l_{i,n}}_{1,n},l} ),
\frac{1}{2}|z_{ \alpha^{k_{i,n}}_{1,n}\alpha_{i,n}\alpha^{l_{i,n}}_{1,n},u}-z_{ \alpha^{k_{i,n}}_{1,n}\alpha_{i,n}\alpha^{l_{i,n}}_{1,n},l} |).\]
By Lemma \ref{fix-trace}, 
\[1-\sigma_1|\tr(\alpha_{i,n})|^{-1}<|z_{ \alpha^{k_{i,n}}_{1,n}\alpha_{i,n}\alpha^{l_{i,n}}_{1,n},l }|
\le|z_{ \alpha^{k_{i,n}}_{1,n}\alpha_{i,n}\alpha^{l_{i,n}}_{1,n},u }|\le|\lambda_{1,n}|^2+\sigma_2|\tr(\alpha_{i,n})|^{-1}\]
and above estimates for $|z_{ \alpha^{k_{i,n}}_{1,n}\alpha_{i,n}\alpha^{l_{i,n}}_{1,n},u } -z_{ \alpha^{k_{i,n}}_{1,n}\alpha_{i,n}\alpha^{l_{i,n}}_{1,n},l } |$ 
we have,
\begin{multline*}
\cosh\dis(\mathcal{L}_{\alpha_{1,n}},\mathcal{L}_{\alpha^{k_{i,n}}_{1,n}\alpha_{i,n}\alpha^{l_{i,n}}_{1,n}})\\
\le
\frac{\frac{1}{4}|z_{ \alpha^{k_{i,n}}_{1,n}\alpha_{i,n}\alpha^{l_{i,n}}_{1,n},u }+z_{ \alpha^{k_{i,n}}_{1,n}\alpha_{i,n}\alpha^{l_{i,n}}_{1,n},l}|^2+
|z_{ \alpha^{k_{i,n}}_{1,n}\alpha_{i,n}\alpha^{l_{i,n}}_{1,n},l} |^2}
{|z_{ \alpha^{k_{i,n}}_{1,n}\alpha_{i,n}\alpha^{l_{i,n}}_{1,n},u}-z_{ \alpha^{k_{i,n}}_{1,n}\alpha_{i,n}\alpha^{l_{i,n}}_{1,n},l}|
(|z_{ \alpha^{k_{i,n}}_{1,n}\alpha_{i,n}\alpha^{l_{i,n}}_{1,n},l} |+\frac{1}{2}|z_{ \alpha^{k_{i,n}}_{1,n}\alpha_{i,n}\alpha^{l_{i,n}}_{1,n},u}-
z_{ \alpha^{k_{i,n}}_{1,n}\alpha_{i,n}\alpha^{l_{i,n}}_{1,n},l} |)}+1\\
<\frac{|\lambda_{1,n}|^2+\rho'''|\tr(\alpha_{i,n})|^{-1}
+\kappa(|\lambda_{1,n}|^2-1)+1}{\kappa(|\lambda_{1,n}|^2-1)},\quad \rho'''>0.
\end{multline*}
This last inequality implies the Lemma.
\end{proof}
\begin{lem}\label{iso-ratio}
\[\lim_n\frac{1}{|\tr(\alpha^{k_{i,n}}_{1,n}\alpha_{i,n}\alpha^{l_{i,n}}_{1,n})|(|\lambda_{1,n}|^2-1)}=0\]
\end{lem}
\begin{proof}
It follows from Proposition \ref{trace} and  
Lemma \ref{axes}, there exists $\rho>0$ such that,
\[\left|\tr(\alpha^{k_{i,n}}_{1,n}\alpha_{i,n}\alpha^{l_{i,n}}_{1,n})\right|\ge\rho\left(\frac{(|\lambda_{1,n}|^{2\mathfrak{D}_n}+3)(|\lambda_{1,n}|^2-1)^{2\mathfrak{D}_n}}
{|\lambda_{1,n}|^{2\mathfrak{D}_n}-1}\right)^{\frac{1}{2\mathfrak{D}_n}}.\]
hence we have,
\[\lim\frac{1}{|\tr(\alpha^{k_{i,n}}_{1,n}\alpha_{i,n}\alpha^{l_{i,n}}_{1,n})|(|\lambda_{1,n}|^2-1)}\le
\lim\rho'\left(\frac{|\lambda_{1,n}|^{2\mathfrak{D}_n}-1}{(|\lambda_{1,n}|^{2\mathfrak{D}_n}+3)(|\lambda_{1,n}|^2-1)^{4\mathfrak{D}_n}}\right)^{\frac{1}{2\mathfrak{D}_n}}.\]
Since $|\lambda_{1,n}|^{2\mathfrak{D}_n}-1<|\lambda_{1,n}|^2-1$ for large $n$ we have,
\[\lim\frac{1}{|\tr(\alpha^{k_{i,n}}_{1,n}\alpha_{i,n}\alpha^{l_{i,n}}_{1,n})|(|\lambda_{1,n}|^2-1)}\le\lim\rho''(|\lambda_{1,n}|^2-1)^\frac{1-4\mathfrak{D}_n}{2\mathfrak{D}_n}=0.\]
\end{proof}
For large $n$ by condition $(i_1),$ we have $\delta(|\lambda_{1,n}|^2-1)<|\lambda_{1,n}|^2-|\zeta_{ \alpha^{k_{i,n}}_{1,n}\alpha_{i,n}\alpha^{l_{i,n}}_{1,n} } |$ 
then by Lemma \ref{fix-trace},
$||z_{ \alpha^{k_{i,n}}_{1,n}\alpha_{i,n}\alpha^{l_{i,n}}_{1,n},u } |-|\zeta_{ \alpha^{k_{i,n}}_{1,n}\alpha_{i,n}\alpha^{l_{i,n}}_{1,n} }||
<\frac{\chi}{|\tr(\alpha^{k_{i,n}}_{1,n}\alpha_{i,n}\alpha^{l_{i,n}}_{1,n})|}$ for some $\chi>0$ we have,
\begin{eqnarray*}
|\lambda_{1,n}|^2-|z_{ \alpha^{k_{i,n}}_{1,n}\alpha_{i,n}\alpha^{l_{i,n}}_{1,n},u } |&>&|\lambda_{1,n}|^2-
|\zeta_{ \alpha^{k_{i,n}}_{1,n}\alpha_{i,n}\alpha^{l_{i,n}}_{1,n} } |-\frac{\chi}{|\tr(\alpha^{k_{i,n}}_{1,n}\alpha_{i,n}\alpha^{l_{i,n}}_{1,n})|}\\
&>&\delta(|\lambda_{1,n}|^2-1)-\frac{\chi}{|\tr(\alpha^{k_{i,n}}_{1,n}\alpha_{i,n}\alpha^{l_{i,n}}_{1,n})|}.
\end{eqnarray*}
Set $\epsilon_{n}=\min_{2\le i\le k}\left(\delta(|\lambda_{1,n}|^2-1)-\frac{\chi}{\left|\tr(\alpha^{k_{i,n}}_{1,n}\alpha_{i,n}\alpha^{l_{i,n}}_{1,n})\right|}\right).$
$z_{u,n}=\max_{2\le i\le k}\left|z_{ \alpha^{k_{i,n}}_{1,n}\alpha_{i,n}\alpha^{l_{i,n}}_{1,n},u }\right|.$
Define Mobius transformations by 
$\psi_{n}(x)=\left(1+\frac{\epsilon_{n}}{2z_{u,n}}\right)\lambda^{-1}_{1,n}(x).$ Then,
\begin{eqnarray*}
|\lambda_{i,n}|-|z_{\psi_n\alpha^{k_{i,n}}_{1,n}\alpha_{i,n}\alpha^{l_{i,n}}_{1,n}\psi^{-1}_n,u}|&=&|\lambda_{i,n}|-
(1+\frac{\epsilon_n}{2z_{u,n}})|\lambda_{1,n}|^{-1}
|z_{ \alpha^{k_{i,n}}_{1,n}\alpha_{i,n}\alpha^{l_{i,n}}_{1,n},u } |\\
&\ge&(|\lambda_{1,n}|^2-|z_{ \alpha^{k_{i,n}}_{1,n}\alpha_{i,n}\alpha^{l_{i,n}}_{1,n},u} |-\frac{\epsilon_n}{2}|)|\lambda_{1,n}|^{-1}\\
&>&(\epsilon_n-\frac{\epsilon_n}{2})|\lambda_{1,n}|^{-1}=\frac{\epsilon_n}{2}|\lambda_{1,n}|^{-1}.
\end{eqnarray*}
Also by Lemma \ref{fix-trace} we have, $||z_{\alpha^{k_{i,n}}_{1,n}\alpha_{i,n}\alpha^{l_{i,n}}_{1,n},l}|-|\eta_{ \alpha^{k_{i,n}}_{1,n}\alpha_{i,n}\alpha^{l_{i,n}}_{1,n} }||
<\frac{\chi'}{|\tr(\alpha_{i,n})|}.$ 
Since $|\eta_{\alpha^{k_{i,n}}_{1,n}\alpha_{i,n}\alpha^{l_{i,n}}_{1,n}}|\ge 1$ for all $2\le i\le k$ this implies,
\begin{eqnarray*}
|z_{\psi\alpha^{k_{i,n}}_{1,n}\alpha_{i,n}\alpha^{l_{i,n}}_{1,n}\psi^{-1},l}|-|\lambda_{1,n}|^{-1}&=&
(1+\frac{\epsilon_n}{2z_{u,n}})
|z_{\alpha^{k_{i,n}}_{1,n}\alpha_{i,n}\alpha^{l_{i,n}}_{1,n},l}||\lambda_{1,n}|^{-1}-|\lambda_{1,n}|^{-1}\\
&=&|\lambda_{1,n}|^{-1}\left(|z_{\alpha^{k_{i,n}}_{1,n}\alpha_{i,n}\alpha^{l_{i,n}}_{1,n},l}|-1+
\frac{\epsilon_n|z_{\alpha^{k_{i,n}}_{1,n}\alpha_{i,n}\alpha^{l_{i,n}}_{1,n},l}|}{2z_{u,n}}\right)\\
&>&|\lambda_{1,n}|^{-1}\left(\frac{\epsilon_n|z_{\alpha^{k_{i,n}}_{1,n}\alpha_{i,n}\alpha^{l_{i,n}}_{1,n},l}|}{2z_{u,n}}-\frac{\chi'}
{|\tr(\alpha_{i,n})|}\right)\\
&>&\frac{1}{|\lambda_{1,n}||\tr(\alpha_{i,n})|}\left(\frac{|\tr(\alpha_{i,n})|\epsilon_n|z_{\alpha^{k_{i,n}}_{1,n}\alpha_{i,n}\alpha^{l_{i,n}}_{1,n},l}|}
{2z_{u,n}}-\chi'\right) 
\end{eqnarray*}
\[
>\frac{1}{|\lambda_{1,n}||\tr(\alpha_{i,n})|}\left(\frac{\delta|\tr(\alpha_{i,n})|(|\lambda_{1,n}|^2-1)|
z_{\alpha^{k_{i,n}}_{1,n}\alpha_{i,n}\alpha^{l_{i,n}}_{1,n},l}|-\chi|z_{\alpha^{k_{i,n}}_{1,n}\alpha_{i,n}\alpha^{l_{i,n}}_{1,n},l}|}
{2z_{u,n}}-\chi'\right)
\]
By Lemma \ref{iso-ratio} and above inequality we have, $|z_{\psi_n z_{\alpha^{k_{i,n}}_{1,n}\alpha_{i,n}\alpha^{l_{i,n}}_{1,n},l}\psi^{-1}_n,l}|-|\lambda_{1,n}|^{-1}>0$ 
and $|\tr(\alpha^{k_{i,n}}_{1,n}\alpha_{i,n}\alpha^{l_{i,n}}_{1,n})|(|z_{\alpha^{k_{i,n}}_{1,n}\alpha_{i,n}\alpha^{l_{i,n}}_{1,n},l}|-|\lambda_{1,n}|^{-1})\to\infty$ and 
$|\tr(\alpha^{k_{i,n}}_{1,n}\alpha_{i,n}\alpha^{l_{i,n}}_{1,n})|(|\lambda_{1,n}|-|z_{\psi_n\alpha^{k_{i,n}}_{1,n}\alpha_{i,n}\alpha^{l_{i,n}}_{1,n},u}|)\to\infty.$ Hence,
\begin{itemize}
\item
$|\lambda_{1,n}|^{-1}<|z_{\psi_n\alpha^{k_{i,n}}_{1,n}\alpha_{i,n}\alpha^{l_{i,n}}_{1,n}\psi^{-1}_n,l}|
\le|z_{\psi_n\alpha^{k_{i,n}}_{1,n}\alpha_{i,n}\alpha^{l_{i,n}}_{1,n}\psi^{-1}_n,u}|<|\lambda_{1,n}|,$ and
\item 
\[0=\begin{cases}
\lim_n\frac{1}{\left(\left|z_{\psi_n\alpha^{k_{i,n}}_{1,n}\alpha_{i,n}\alpha^{l_{i,n}}_{1,n}\psi_n^{-1},u}\right|-\left|\lambda_{1,n}\right|\right)
\left|\tr\left(\alpha^{k_{i,n}}_{1,n}\alpha_{i,n}\alpha^{l_{i,n}}_{1,n}\right)\right|}\\
\lim_n\frac{1}{\left(\left|z_{\psi_n\alpha^{k_{i,n}}_{1,n}\alpha_{i,n}\alpha^{l_{i,n}}_{1,n}\psi_n^{-1},l}\right|-\left|\lambda_{1,n}\right|^{-1}\right)
\left|\tr\left(\alpha^{k_{i,n}}_{1,n}\alpha_{i,n}\alpha^{l_{i,n}}_{1,n}\right)\right|}.
\end{cases}\]
\end{itemize}
Also by $\kappa({S_{\G_n}})\to\infty,$ implies that
\begin{multline*}
\left|z_{\psi_{n;1}\alpha^{k_{i,n}}_{1,n}\alpha_{i,n}\alpha^{l_{i,n}}_{1,n}\psi^{-1}_{n;1} ,\pm}-
z_{\psi_{n;1}\alpha^{k_{j,n}}_{1,n}\alpha_{j,n}\alpha^{l_{j,n}}_{1,n}\psi^{-1}_{n;1},\pm}\right|\\
=\left|\left(1+\frac{\epsilon_{n}}{2z_{u,n}}\right)\lambda^{-1}_{1,n}\right|\left|z_{\alpha^{k_{i,n}}_{1,n}\alpha_{i,n}\alpha^{l_{i,n}}_{1,n},\pm}-
z_{\alpha^{k_{j,n}}_{1,n}\alpha_{j,n}\alpha^{l_{j,n}}_{1,n},\pm}\right|\to\infty,
\end{multline*}
for $i\not=j, 2\le i,j\le k$ and $\left|\tr\left( \alpha^{k_{i,n}}_{1,n}\alpha_{i,n}\alpha^{l_{i,n}}_{1,n}\right)\right|
\le\left|\tr\left(\alpha^{k_{j,n}}_{1,n}\alpha_{j,n}\alpha^{l_{j,n}}_{1,n}\right)\right|.$ \par
Hence the generators 
$<\psi_n\alpha_{1,n}\psi^{-1}_n,\psi_n\alpha^{k_{2,n}}_{1,n}\alpha_{2,n}\alpha^{l_{2,n}}_{1,n}\psi^{-1}_n,...,
\psi_n\alpha^{k_{k,n}}_{1,n}\alpha_{k,n}\alpha^{l_{k,n}}_{1,n}\psi^{-1}_n >$ satisfies conditions of Lemma \ref{classical} for large $n.$
For future reference we set the Mobius transformations defined earlier as $\psi_{n;1}(x).$\par
Next we present the proof for $(i_2).$ This case requires a bit more computation. This is due to the fact that we need to address the possibility that collapsing of
fixed points imply that the distance between axises of $\beta_{i,n}$ and $\alpha_{1,n}$ will go to $\infty.$ So we can't conclude that $|\tr(\beta_{i,n})|\to\infty.$ Hence 
we will use different set of circles given by Proposition $5.5$ to compare rates.
\subsubsection{ Case $(i_2)$.}
There exists $|\rho_{\alpha^{k_{i,n}-1}_{1,n}\alpha_{i,n}\alpha^{l_{i,n}}_{1,n} }|\to 1$ such that 
$\zeta_{\alpha^{k_{i,n}-1}_{1,n}\alpha_{i,n}\alpha^{l_{i,n}}_{1,n} }\rho_{\alpha^{k_{i,n}-1}_{1,n}\alpha_{i,n}\alpha^{l_{i,n}}_{1,n}}
=\eta_{\alpha^{k_n-1}_n\beta_n}$.
If $\limsup_n|\rho_{\alpha^{k_{i,n}-1}_{1,n}\alpha_{i,n}\alpha^{l_{i,n}}_{1,n} }-1|>0$, 
then (with the same index notation for subsequence there exists a subsequence)
such that 
\[\liminf_n|\zeta_{\alpha^{k_{i,n}-1}_{1,n}\alpha_{i,n}\alpha^{l_{i,n}}_{1,n} }-\eta_{\alpha^{k_{i,n}}_{1,n}\alpha_{i,n}\alpha^{l_{i,n}}_{1,n} }|>0.\]
This implies by Lemma \ref{fix-trace}, 
$\liminf_n|z_{\alpha^{k_{i,n}-1}_{1,n}\alpha_{i,n}\alpha^{l_{i,n}}_{1,n} ,+}-z_{\alpha^{k_{i,n}}_{1,n}\alpha_{i,n}\alpha^{l_{i,n}}_{1,n} ,-}|>0.$ 
Hence by Proposition \ref{trace}, there exists $\rho>0$, such that for large $n$,
\[|\tr(\alpha^{k_{i,n}-1}_{1,n}\alpha_{i,n}\alpha^{l_{i,n}}_{1,n} )|\ge\rho\left(\frac{|\lambda_{1,n}|^{2\mathfrak{D}_n}+3}{|\lambda_{1,n}|^{2\mathfrak{D}_n}-1}\right)^{\frac{1}{2\mathfrak{D}_n}}.\] In particular,
we have $|\tr(\alpha^{k_{i,n}-1}_{1,n}\alpha_{i,n}\alpha^{l_{i,n}}_{1,n} )|\to\infty.$ Note 
\[\lim_n\frac{|\rho_{\alpha^{k_{i,n}-1}_{1,n}\alpha_{i,n}\alpha^{l_{i,n}}_{1,n}}|-1}{|\lambda_{1,n}|^2-1}=0.\]
This can be seen as follows: since $|\lambda_{1,n}|^2-|\zeta_{\alpha^{k_{i,n}}_{1,n}\alpha_{i,n}\alpha^{l_{i,n}}_{1,n}}|\le|\lambda_{1,n}|^2-1$ we have either, 
$\frac{|\lambda_{1,n}|^2-|\zeta_{\alpha^{k_{i,n}}_{1,n}\alpha_{i,n}\alpha^{l_{i,n}}_{1,n}}|}{|\lambda_{1,n}|^2-1}\to 0$ or 
$1\ge\frac{|\lambda_{1,n}|^2-|\zeta_{\alpha^{k_{i,n}}_{1,n}\alpha_{i,n}\alpha^{l_{i,n}}_{1,n}}|}{|\lambda_{1,n}|^2-1}>\epsilon>0.$\par
Assume that the latter inequality holds. This is equivalent to $(i_1)$ and we follows the same idea used in $(i_1).$
Set Mobius transformations $\psi_{n;2}(x)=\lambda_{1,n}^{-1}(1-\epsilon_n)^{-1}x,$ with $\epsilon_n=\frac{\epsilon(|\lambda_{1,n}|^2-1)}{2|\lambda_{1,n}|^2}.$ Then,
\begin{eqnarray*}
|\lambda_{1,n}|-|\zeta_{\psi_{n;2}\alpha^{k_{i,n}}_{1,n}\alpha_{i,n}\alpha^{l_{i,n}}_{1,n}\psi^{-1}_{n;2}}|&=&|\lambda_{1,n}|-|\lambda_{1,n}^{-1}
(1-\epsilon_n)^{-1}\zeta_{\alpha^{k_{i,n}}_{1,n}\alpha_{i,n}\alpha^{l_{i,n}}_{1,n}}|\\
&=&|\lambda_{1,n}|^{-1}(1-\epsilon_n)^{-1}\left(|\lambda_{1,n}|^2(1-\epsilon_n)-|\zeta_{\alpha^{k_{i,n}}_{1,n}\alpha_{i,n}\alpha^{l_{i,n}}_{1,n}}|\right)
\end{eqnarray*}
Since $\quad |\lambda_{1,n}|^2-|\zeta_{\alpha^{k_{i,n}}_{1,n}\alpha_{i,n}\alpha^{l_{i,n}}_{1,n}}|>\epsilon(|\lambda_{1,n}|^2-1)$ we have,
\[|\lambda_{1,n}|-|\zeta_{\psi_{n;2}\alpha^{k_{i,n}}_{1,n}\alpha_{i,n}\alpha^{l_{i,n}}_{1,n}\psi^{-1}_{n;2}}|
> |\lambda_{1,n}|^{-1}(1-\epsilon_n)^{-1} (\epsilon(|\lambda_{1,n}|^2-1)-|\lambda_{1,n}|^2\epsilon_n)
=\frac{\epsilon(|\lambda_{1,n}|^2-1)}{2|\lambda_{1,n}|(1-\epsilon_n)}.\]
Since that $\epsilon_n\to 0,$ it follows from last inquality that for large $n$ we have,
$\frac{|\lambda_{1,n}|-|\zeta_{\psi_{n;2}\alpha^{k_{i,n}-1}_{1,n}\alpha_{i,n}\alpha^{l_{i,n}}_{1,n}\psi^{-1}_{n;2}}|}{|\lambda_{1,n}|^2-1}>\epsilon'>0.$
And $|\eta_{\psi_{n;2}\alpha^{k_{i,n}}_{1,n}\alpha_{i,n}\alpha^{l_{i,n}}_{1,n}\psi^{-1}_{n;2}}|\ge|\lambda_{1,n}|^{-1}(1-\epsilon_n)^{-1}$ we have, 
\[\frac{|\eta_{\psi_{n;2}\alpha^{k_{i,n}}_{1,n}\alpha_{i,n}\alpha^{l_{i,n}}_{1,n}\psi^{-1}_{n;2}}|-|\lambda_n|^{-1}}{|\lambda|^2-1}=\frac{\epsilon}{2|\lambda_n|^3(1-\epsilon_n)}>\epsilon''>0.\]
Since $|\tr(\alpha^{k_{i,n}}_{1,n}\alpha_{i,n}\alpha^{l_{i,n}}_{1,n})(|\lambda_{1,n}|^2-1)|\to\infty$, it follows that for large $n$,
\begin{itemize}
\item
$|\lambda_{1,n}|^{-1}<|\eta_{\psi_{n;2}\alpha^{k_{i,n}}_{1,n}\alpha_{i,n}\alpha^{l_{i,n}}_{1,n}\psi^{-1}_{n;2}}|
\le|\zeta_{\psi_{n;2}\alpha^{k_{i,n}}_{1,n}\alpha_{i,n}\alpha^{l_{i,n}}_{1,n}\psi^{-1}_{n;2}}|<|\lambda_{1,n}| ,$ and
\item 
\[
0=\lim_n\begin{cases}
\frac{1}{\left|\tr\left(\alpha^{k_{i,n}}_{1,n}\alpha_{i,n}\alpha^{l_{i,n}}_{1,n}\right)\right|
\left(\left|\eta_{\psi_{n;2}\alpha^{k_{i,n}}_{1,n}\alpha_{i,n}\alpha^{l_{i,n}}_{1,n}\psi_{n;2}^{-1}}\right|-\left|\lambda_{1,n}\right|^{-1}\right)}\\
\frac{1}{\left|\tr\left(\alpha^{k_{i,n}}_{1,n}\alpha_{i,n}\alpha^{l_{i,n}}_{1,n}\right)\right|\left(\left|\zeta_{\psi_{n;2}\alpha^{k_{i,n}}_{1,n}
\alpha_{i,n}\alpha^{l_{i,n}}_{1,n}\psi_{n;2}^{-1}}\right|-\left|\lambda_{1,n}\right|\right)}
\end{cases}\]
\end{itemize}
and since $\kappa(S_{\G_n})\to\infty,$ implies that
\begin{multline*}
\left|z_{\psi_{n;2} \alpha^{k_{i,n}}_{1,n}\alpha_{i,n}\alpha^{l_{i,n}}_{1,n}\psi_{n;2}^{-1} ,\pm}-
z_{\psi_{n;2}\alpha^{k_{j,n}}_{1,n}\alpha_{j,n}\alpha^{l_{j,n}}_{1,n}\psi_{n;2}^{-1},\pm}\right|\\
=\left|\lambda_{1,n}^{-1}(1-\epsilon_n)^{-1}\right|\left|z_{\alpha^{k_{i,n}}_{1,n}\alpha_{i,n}\alpha^{l_{i,n}}_{1,n},\pm}-
z_{\alpha^{k_{j,n}}_{1,n}\alpha_{j,n}\alpha^{l_{j,n}}_{1,n},\pm}\right|\to\infty,
\end{multline*}
for $i\not=j, 2\le i,j\le k$ and $\left|\tr\left( \alpha^{k_{i,n}}_{1,n}\alpha_{i,n}\alpha^{l_{i,n}}_{1,n}\right)\right|
\le\left|\tr\left(\alpha^{k_{j,n}}_{1,n}\alpha_{j,n}\alpha^{l_{j,n}}_{1,n}\right)\right|.$ \par
Hence by Lemma \ref{fix-trace}, $<\psi_{n;2}\alpha_{1,n}\psi_{n;2}^{-1},...,\psi_{n;2}\alpha^{k_{k,n}}_{1,n}\alpha_{k,n}\alpha^{l_{k,n}}_{1,n}\psi_{n;2}^{-1}>$ satisfies Lemma \ref{classical}.\par
If the former holds then 
\begin{multline*}
|\lambda_{1,n}|^2(|\eta_{\alpha^{k_{i,n}}_{1,n}\alpha_{i,n}\alpha^{l_{i,n}}_{1,n}}|-|\zeta_{\alpha^{k_{i,n}}_{1,n}\alpha_{i,n}\alpha^{l_{i,n}}_{1,n}}\lambda_{1,n}^{-2}|)
=|\lambda_{1,n}|^2(|\eta_{\alpha^{k_{i,n}}_{1,n}\alpha_{i,n}\alpha^{l_{i,n}}_{1,n} }|-|\zeta_{\alpha^{k_{i,n}-1}_{1,n}\alpha_{i,n}\alpha^{l_{i,n}}_{1,n} }|)\\
=|\lambda_{1,n}|^2(|\zeta_{\alpha^{k_{i,n}-1}_{1,n}\alpha_{i,n}\alpha^{l_{i,n}}_{1,n} }\rho_{\alpha^{k_{i,n}-1}_{1,n}\alpha_{i,n}\alpha^{l_{i,n}}_{1,n} }| 
-|\zeta_{\alpha^{k_{i,n}-1}_{1,n}\alpha_{i,n}\alpha^{l_{i,n}}_{1,n}}|)\\
=|\lambda_{1,n}|^2|\zeta_{\alpha^{k_{i,n}-1}_{1,n}\alpha_{i,n}\alpha^{l_{i,n}}_{1,n}}|(|\rho_{\alpha^{k_{i,n}-1}_{1,n}\alpha_{i,n}\alpha^{l_{i,n}}_{1,n} }|-1)
\end{multline*}
The last equation implies that $\lim_n\frac{|\rho_{\alpha^{k_{i,n}-1}_{1,n}\alpha_{i,n}\alpha^{l_{i,n}}_{1,n} }|-1}{|\lambda_{1,n}|^2-1}=0,$ 
and $|\tr( \alpha^{k_{i,n}-1}_{1,n}\alpha_{i,n}\alpha^{l_{i,n}}_{1,n} )(|\lambda_{1,n}|^2-1)|\to\infty$,
it follows that for large $n$,
\begin{itemize}
\item
$|\lambda_{1,n}|^{-1}<|\zeta_{\alpha^{k_{i,n}-1}_{1,n}\alpha_{i,n}\alpha^{l_{i,n}}_{1,n} }|
\le|\eta_{\alpha^{k_{i,n}-1}_{1,n}\alpha_{i,n}\alpha^{l_{i,n}}_{1,n} }|<|\lambda_{1,n}|,$ and
\item 
\[0=\lim_n\begin{cases}
\frac{1}{\left|\tr\left( \alpha^{k_{i,n}-1}_{1,n}\alpha_{i,n}\alpha^{l_{i,n}}_{1,n} \right)\right|
\left(\left|\zeta_{\alpha^{k_{i,n}-1}_{1,n}\alpha_{i,n}\alpha^{l_{i,n}}_{1,n}}\right|-\left|\lambda_{1,n}\right|^{-1}\right)}\\
\frac{1}{\left|\tr\left( \alpha^{k_{i,n}-1}_{1,n}\alpha_{i,n}\alpha^{l_{i,n}}_{1,n} \right)\right|\left(\left|\eta_{\alpha^{k_{i,n}-1}_{1,n}
\alpha_{i,n}\alpha^{l_{i,n}}_{1,n}}\right|-\left|\lambda_{1,n}\right|\right)}
\end{cases}\]
\item
\[\left|z_{ \alpha^{k_{i,n}-1}_{1,n}\alpha_{i,n}\alpha^{l_{i,n}}_{1,n} ,\pm}-z_{\alpha^{k_{j,n}-1}_{1,n}\alpha_{j,n}\alpha^{l_{j,n}}_{1,n},\pm}\right|
\left|\tr\left(\alpha^{k_{i,n}-1}_{1,n}\alpha_{i,n}\alpha^{l_{i,n}}_{1,n}\right)\right|\to\infty;\quad i\not=j\]
\end{itemize} 
Hence by Lemma \ref{fix-trace}, $<\alpha_{1,n},...,\alpha^{k_{k,n}-1}_{1,n}\alpha_{k,n}\alpha^{l_{k,n}}_{1,n}>$ satisfies Lemma \ref{classical}.\par
\noindent Consider the case that, $\rho_{\alpha^{k_{i,n}-1}_{1,n}\alpha_{i,n}\alpha^{l_{i,n}}_{1,n}}\to 1$.\par
We first consider the growth conditions that relates between 
$\alpha_{1,n},\alpha^{k_{i,n}-1}_{1,n}\alpha_{i,n}\alpha^{l_{i,n}}_{1,n}$ and $\alpha^{k_{i,n}}_{1,n}\alpha_{i,n}\alpha^{l_{i,n}}_{1,n}.$ To do so we do some two generators
estimates. Fix a $i$ that satisfies $\rho_{\alpha^{k_{i,n}-1}_{1,n}\alpha_{i,n}\alpha^{l_{i,n}}_{1,n}}\to 1$. We first normalize so that $\eta_{\alpha^{k_{i,n}}_{1,n}\alpha_{i,n}\alpha^{l_{i,n}}_{1,n}}=1.$
Since our estimates will be in terms of traces so for simplicity we can take $l_{i,n}=0$ and set 
$\eta_{\alpha^{k_{i,n}-1}_{1,n}\alpha_{i,n}\alpha^{l_{i,n}}_{1,n}}=1.$
by conjugation with a desired Mobius transformation. \par
Let $f_{i,n}$ be the Mobius transformations given by,
\[f_{i,n}(x)=\frac{1}{\sigma_{i,n}}\frac{z_{ \alpha^{k_{i,n}-1}_{1,n}\alpha_{i,n},-}x-z_{\alpha^{k_{i,n}-1}_{1,n}\alpha_{i,n},-}z_{\alpha^{k_{i,n}-1}_{1,n}\alpha_{i,n},+}}
{x-z_{ \alpha^{k_{i,n}-1}_{1,n}\alpha_{i,n} ,-}}, \quad x\in\mathbb{C}.\]
where $\sigma_{i,n}=\frac{1}{z_{\alpha^{k_{i,n}-1}_{1,n}\alpha_{i,n},-}z_{\alpha^{k_{i,n}-1}_{1,n}\alpha_{i,n},+}-z^2_{\alpha^{k_{i,n}-1}_{1,n}\alpha_{i,n},-}}.$
Let $\phi_{i,n}$ be the Mobius transformations defined by,
$\phi_{i,n}(x)=\eta^{-1}_{f_{i,n}\alpha^{k_{i,n}}_{1,n}\alpha_{i,n} f^{-1}_{i,n}}x.$ Then
$\zeta_{ \phi_{i,n} f_{i,n} \alpha^{k_{i,n}}_{1,n}\alpha_{i,n}f^{-1}_{i,n} \phi^{-1}_{i,n} }$ is given by,
\begin{eqnarray*}
\zeta_{ \phi_{i,n} f_{i,n}\alpha^{k_{i,n}}_{1,n}\alpha_{i,n}f^{-1}_{i,n}\phi^{-1}_{i,n} }=
\left(\frac{1-z_{\alpha^{k_{i,n}-1}_{1,n}\alpha_{i,n},-}}{1-z_{\alpha^{k_{i,n}-1}_{1,n}\alpha_{i,n},+}}\right)
\left(\frac{\alpha^{k_{i,n}}_{1,n}\alpha_{i,n}(z_{\alpha^{k_{i,n}-1}_{1,n}\alpha_{i,n},-})-z_{\alpha^{k_{i,n}-1}_{1,n}\alpha_{i,n},+}}
{\alpha^{k_{i,n}}_{1,n}\alpha_{i,n}(z_{\alpha^{k_{i,n}-1}_{1,n}\alpha_{i,n},+})-z_{\alpha^{k_{i,n}-1}_{1,n}\alpha_{i,n},-}}\right)
\end{eqnarray*}
To see this we do a simple computation. Set $\tilde{\alpha}_{i,n}=\phi_n\psi_n\alpha^{k_{i,n}-1}_{1,n}\alpha_{i,n}\alpha^{l_{i,n}}_{1,n}\psi^{-1}_n\phi^{-1}_n$ and 
$\tilde{\beta}_{i,n}=\phi_n\psi_n\alpha^{k_{i,n}}_{1,n}\alpha_{i,n}\alpha^{l_{i,n}}_{1,n}\psi^{-1}_n\phi^{-1}_n.$ 
Writing in matrix $\tilde{\beta}_{i,n}=\bigl(\begin{smallmatrix} \tilde{a}_n & \tilde{b}_n\\ \tilde{c_n} & \tilde{d}_n\end{smallmatrix}\bigr).$
Note that by our choice of $\phi_n$, we have $\eta_{\tilde{\beta}_{i,n}}=1,$ so $\tilde{c}_n=-\tilde{d}_n$ and $\zeta_{\tilde{\beta}_{i,n}}=\frac{\tilde{a}_n}{-\tilde{d}_n}.$
By straight forward matrix multiplications we have,
\begin{eqnarray*}
\tilde{a}_n=
-z^2_{\alpha^{k_n-1}_n\beta_n,-}\lambda_n^{k_n}a_n&-&z_{\alpha^{k_n-1}_n\beta_n,-}\lambda_n^{k_n}b_n+\\
&&z_{\alpha^{k_n-1}_n\beta_n,-}z_{\alpha^{k_n-1}_n\beta_n,+}
(z_{\alpha^{k_n-1}_n\beta_n,-}\lambda^{-k_n}_nc_n+\lambda_n^{-k_n}d_n)
\end{eqnarray*}
\begin{eqnarray*}
\tilde{a}_n&=&\frac{
\left(\frac{-z_{\alpha^{k_n-1}_n\beta_n,-}(\lambda_n^{k_n}a_nz_{\alpha^{k_n-1}_n\beta_n,-}+\lambda^{k_n}_nb_n)}
{z_{\alpha^{k_n-1}_n\beta_n,-}\lambda^{-k_n}_nc_n+\lambda^{-k_n}_nd_n}+z_{\alpha^{k_n-1}_n\beta_n,-}z_{\alpha^{k_n-1}_n\beta_n,+}\right)}
{(z_{\alpha^{k_n-1}_n\beta_n,-}\lambda^{-k_n}_nc_n+\lambda^{-k_n}_nd_n)^{-1}}\\
&=&\frac{-z_{\alpha^{k_n-1}_n\beta_n,-}(\alpha^{k_n}_n\beta_n(z_{\alpha^{k_n-1}_n\beta_n,-})-z_{\alpha^{k_n-1}_n\beta_n,+})}
{(z_{\alpha^{k_n-1}_n\beta_n,-}\lambda_n^{-k_n}c_n+\lambda^{-k_n}_nd_n)^{-1}}
\end{eqnarray*}
\begin{eqnarray*}
\tilde{d}_n=z_{\alpha^{k_n-1}_n\beta_n,-}z_{\alpha^{k_n-1}_n\beta_n,+}\lambda^{k_n}_na_n&+&z_{\alpha^{k_n-1}_n\beta_n,-}\lambda^{k_n}_nb_n-\\
&&z^2_{\alpha^{k_n-1}_n\beta_n,-}(z_{\alpha^{k_n-1}_n\beta_n,+}\lambda_n^{k_n}c_n\lambda_n^{-k_n}d_n)
\end{eqnarray*}
\begin{eqnarray*}
\tilde{d}_n&=&\frac{
\left(\frac{z_{\alpha^{k_n-1}_n\beta_n,-}(\lambda_n^{k_n}a_nz_{\alpha^{k_n-1}_n\beta_n,+}+\lambda^{k_n}_nb_n)}
{z_{\alpha^{k_n-1}_n\beta_n,+}\lambda^{-k_n}_nc_n+\lambda^{-k_n}_nd_n}-z^2_{\alpha^{k_n-1}_n\beta_n,-}\right)}
{(z_{\alpha^{k_n-1}_n\beta_n,+}\lambda^{-k_n}_nc_n+\lambda^{-k_n}_nd_n)^{-1}}\\
&=&\frac{z_{\alpha^{k_n-1}_n\beta_n,-}(\alpha^{k_n}_n\beta_n(z_{\alpha^{k_n-1}_n\beta_n,+})-z_{\alpha^{k_n-1}_n\beta_n,-})}
{(z_{\alpha^{k_n-1}_n\beta_n,+}\lambda_n^{-k_n}c_n+\lambda^{-k_n}_nd_n)^{-1}}
\end{eqnarray*}
Now by $\eta_n=1$ we have,
\[\frac{ z_{\alpha^{k_n-1}_n\beta_n,+} \lambda_n^{-k_n}c_n+\lambda^{-k_n}_nd_n}{z_{\alpha^{k_n-1}_n\beta_n,-}\lambda_n^{-k_n}c_n+\lambda^{-k_n}_nd_n}=
\frac{1-z_{\alpha^{k_n-1}_n\beta_n,+}}{1-z_{\alpha^{k_n-1}_n\beta_n,-}}.\]
Since $\zeta_{\tilde{\beta}_{i,n}} =\frac{\tilde{a}_n}{-\tilde{d}_n},$ the formula for $\zeta_{\tilde{\beta}_{i,n}}$ follows from above equations for
$\tilde{a}_n$ and $\tilde{d}_n.$\par
Let $\tilde{\lambda}_n$, denote the multiplier of $\tilde{\alpha}_{1,n}.$ \par
First, we need to get a estimate of the growth of $|\tr(\tilde{\beta}_{i,n})|$ in terms of $|\tr(\beta_{i,n})|.$ Note that 
$|\tr(\tilde{\beta}_{i,n})|=|\tr( \alpha^{k_{i,n}}_{1,n}\alpha_{i,n}\alpha^{l_{i,n}}_{1,n} )|.$  
\begin{rema}\label{tilde-trace}
There exists $\sigma,\sigma'>0$ such that 
\[\sigma'|\tr(\alpha^{k_{i,n}}_{1,n}\alpha_{i,n}\alpha^{l_{i,n}}_{1,n})|>|\tr(\tilde{\beta}_{i,n})|>\sigma|\tr(\alpha^{k_{i,n}}_{1,n}\alpha_{i,n}\alpha^{l_{i,n}}_{1,n})|(|\lambda_{1,n}|^2-1)\quad\quad\text{for $n$ large}.\]
\end{rema}
In fact we only need the lower bound for $|\tr(\tilde{\beta}_{i,n})|$.
\begin{proof} 
Since, 
\[ \frac{|\lambda_{1,n}|^2-1}{|\zeta_{\alpha^{k_{i,n}}_{1,n}\alpha_{i,n}\alpha^{l_{i,n}}_{1,n}}-1|}+ \frac{|\zeta_{\alpha^{k_{i,n}}_{1,n}\alpha_{i,n}\alpha^{l_{i,n}}_{1,n}}|-
|\lambda_{1,n}|^2}{|\zeta_{\alpha^{k_{i,n}}_{1,n}\alpha_{i,n}\alpha^{l_{i,n}}_{1,n}}-1|}
=\frac{|\zeta_{\alpha^{k_{i,n}}_{1,n}\alpha_{i,n}\alpha^{l_{i,n}}_{1,n}}|-1}{|\zeta_{\alpha^{k_{i,n}}_{1,n}\alpha_{i,n}\alpha^{l_{i,n}}_{1,n}}-1|}\le1,\]
and by condition $(i_2),$ $\frac{|\zeta_{\alpha^{k_{i,n}}_{1,n}\alpha_{i,n}\alpha^{l_{i,n}}_{1,n}}|-|\lambda_{1,n}|^2}{|\zeta_{\alpha^{k_{i,n}}_{1,n}\alpha_{i,n}\alpha^{l_{i,n}}_{1,n}}-1|}\to 0$ we have,
\[ \frac{|\zeta_{\alpha^{k_{i,n}}_{1,n}\alpha_{i,n}\alpha^{l_{i,n}}_{1,n}}-1|}{|\lambda_{1,n}|^2-1}>\epsilon>0,\quad\quad\text{for large $n$}.\]
Recall $\eta_{\alpha^{k_{i,n}}_{1,n}\alpha_{i,n}\alpha^{l_{i,n}}_{1,n}}=1$, and by Lemma \ref{fix-bound-2}, more precisely by equation $(1)$ in the proof of Lemma \ref{fix-bound-2} and 
$|\tr(\alpha_{i,n})|\asymp |c_{i,n}|,$
\[\frac{|\tr(\alpha^{k_{i,n}}_{1,n}\alpha_{i,n}\alpha^{l_{i,n}}_{1,n})|}{|\tr(\alpha_{i,n})|}>\sigma' |\zeta_{\alpha^{k_{i,n}}_{1,n}\alpha_{i,n}\alpha^{l_{i,n}}_{1,n}}-
\eta_{\alpha^{k_{i,n}}_{1,n}\alpha_{i,n}\alpha^{l_{i,n}}_{1,n}}|
>\sigma'\epsilon(|\lambda_{1,n}|^2-1), \quad\text{for large $n$}.\]
The upper bound is trivial.
\end{proof}
We define a fuction of generators $\alpha_{i,n},\alpha_{1,n}$ as $f(\alpha_{i,n},\alpha_{1,n})=\frac{|\zeta_{\tilde{\alpha}_{i,n}}-1|}{|\tilde{\lambda}_{1,n}|^2-1}.$ \par
In next Lemma, we will prove $(i_2)$ by: if $\lim_jf(\alpha_{i,n_j},\alpha_{1,n_j})=0$ then we have classical Schottky generators. In Lemma $7.6$, we will prove $(i_2)$ for the case $f(\alpha_{i,n_j},\alpha_{1,n_j})>M$ for some $M>0.$ Hence we will be done for $(i_2).$
\begin{lem}\label{m=0}
Assume that there exists a subsequence such that,
\[\lim_jf(\alpha_{i,n_j},\alpha_{1,n_j})=0.\]
Then $S_j=<\alpha^{k_{i,n_j}-1}_{1,n_j}\alpha_{i,n_j}\alpha^{l_{i,n_j}}_{1,n_j},...,
\alpha^{k_{i,n_j}}_{1,n_j}\alpha_{i,n_j}\alpha^{l_{i,n_j}}_{1,n_j},..., 
\alpha^{k_{k,n_j}}_{1,n_j}\alpha_{k,n_j}\alpha^{l_{k,n_j}}_{1,n_j}>$ 
classical Schottky generators for large $n.$ 
i.e with generator $\alpha_{1,n_j}$ replaced by $\alpha^{k_{i,n_j}-1}_{1,n_j}\alpha_{i,n_j}\alpha^{l_{i,n_j}}_{1,n_j}$ and rest generators remains the same.
\end{lem}
\begin{proof}
We will show that $S_j$ satisfies conditions of Lemma \ref{classical} with Remark \ref{classical-rem}.\par
Since by Remark 6.A, $\sigma'|\tr(\alpha_{i,n})|>|\tr(\tilde{\beta}_{i,n_j})|>\sigma|\tr(\alpha_{i,n_j})|(|\lambda_{1,n_j}|^2-1).$ In particular $|\tr(\tilde{\beta}_{i,n_j})|\to\infty$, we have
\[\lim_j\frac{\left|z_{\tilde{\beta}_{i,n_j},+}-z_{\tilde{\beta}_{i,n_j},-}\right|}{\left|\zeta_{\tilde{\beta}_{i,n_j}}-1\right|}=
\lim_j\left|\frac{\sqrt{|\tr(\tilde{\beta}_{i,n_j})|^2-4}}{\tr(\tilde{\beta}_{i,n_j})}\right|=1\]
Since $|\tr(\tilde{\beta}_{i,n_j})|\to\infty$ implies that the isometric circles of $\tilde{\beta}_{i,n_j}$ are disjoint for large $n.$ Since
$\zeta_{\tilde{\beta}_{i,n_j}},\eta_{\tilde{\beta}_{i,n_j}}$ are centers of these isometric circles and so by disjointness, the radius of these
isometric circles must be
$<\frac{|\zeta_{\tilde{\beta}_{i,n}}-\eta_{\tilde{\beta}_{i,n_j}}|}{2}$ for large $j.$ In addition, each isometric circle contains one of the fixed point
$z_{\tilde{\beta}_{i,n_j},l}, z_{\tilde{\beta}_{i,n_j},u}.$ By our convention $z_{\tilde{\beta}_{i,n_j},l},z_{\tilde{\beta}_{i,n_j},u}$ are contained within
the isometric circles with centers $\zeta_{\tilde{\beta}_{i,n_j}},\eta_{\tilde{\beta}_{i,n_j}}$ respectively.
Note that $\eta_{\tilde{\beta}_{i,n_j}}=1.$ Hence for large $n_j$ we have,
\[\frac{|z_{\tilde{\beta}_{i,n_j},l}-\zeta_{\tilde{\beta}_{i,n_j}}|}{|\zeta_{\tilde{\beta}_{i,n_j}}-1|},
\frac{|z_{\tilde{\beta}_{i,n_j},u}-1|}{|\zeta_{\tilde{\beta}_{i,n_j}}-1|}<\frac{1}{2}.\]
From these bounds we have,
\begin{align*}
\lim_j\frac{|z_{\tilde{\beta}_{i,n_j},l}-z_{\tilde{\beta}_{i,n_j},u}|}{(|z_{\tilde{\beta}_{i,n_j},u}|-|\tilde{\lambda}_{i,n_j}|)|\tr(\tilde{\beta}_{i,n_j})|}
&=\lim_j\frac{\left|\zeta_{\tilde{\beta}_{i,n_j}}-1\right|}{\left|(|z_{\tilde{\beta}_{i,n_j},u}|-1)+(1-|\tilde{\lambda}_{i,n_j}|)\right|\left|\tr(\tilde{\beta}_{i,n_j})\right|}\\
&\le\lim_j\frac{1}{\left(\left|\frac{|\tilde{\lambda}_{i,n_j}|-1}{\zeta_{\tilde{\beta}_{i,n_j}}-1}\right|-\frac{1}{2}\right)|\tr(\tilde{\beta}_{i,n_j})|}=0
\end{align*}
Similarly we have,
\begin{equation*}
\begin{split}
\lim_j\frac{|z_{\tilde{\beta}_{i,n_j},l}-z_{\tilde{\beta}_{i,n_j},u}|}{(|z_{\tilde{\beta}_{i,n_j},l}|-|\tilde{\lambda}_{i,n_j}|^{-1})|\tr(\tilde{\beta}_{i,n_j})|}
\quad\quad\quad\quad\quad\quad\quad\quad\quad\quad\quad\quad\quad\quad\quad\quad\quad\quad\quad\text{        }\\
=\lim_j\frac{\left|\zeta_{\tilde{\beta}_{i,n_j}}-1\right|}{\left|(|z_{\tilde{\beta}_{i,n_j},u}|-|\zeta_{\tilde{\beta}_{i,n_j}}|)+
(|\zeta_{\tilde{\beta}_{i,n_j}}|-1)+(1-|\tilde{\lambda}_{i,n_j}|^{-1})\right|\left|\tr(\tilde{\beta}_{i,n_j})\right|}\quad\quad\quad\quad\quad\text{        }\\
\le\lim_j\frac{1}{\left(\frac{1}{|\tilde{\lambda}_{i,n_j}|}\left|\frac{|\tilde{\lambda}_{i,n_j}|-1}{\zeta_{\tilde{\beta}_{i,n_j}}-1}\right|
+\left|\frac{|\zeta_{\tilde{\beta}_{i,n_j}}|-1}{\zeta_{\tilde{\beta}_{i,n_j}}-1}\right| -\frac{1}{2}\right)|\tr(\tilde{\beta}_{i,n_j})|}=0
\quad\quad\quad\quad\quad\quad\quad\quad\quad\quad\text{        }
\end{split}
\end{equation*}
Hence by Remark \ref{classical-rem}, $\tilde{\alpha}_{n_j},\tilde{\beta}_{n_j}$ have disjoint circles for large $n_j$ which implies that
$\alpha_{i,n_j},\beta_{i,n_j}$ have disjoint cricles. By $\kappa(S_{\G_{n_j}})\to\infty$ we have 
\[|z_{\alpha^{k_{i,n_j}-1}_{1,n_j}\alpha_{i,n_j}\alpha^{l_{i,n_j}}_{1,n_j},\pm}-z_{\alpha^{k_{j,n_j}}_{1,n_j}\alpha_{j,n_j}\alpha^{l_{j,n_j}}_{1,n_j},\pm}|
|\tr(\alpha^{k_{i,n_j}}_{1,n_j}\alpha_{i,n_j}\alpha^{l_{i,n_j}}_{1,n_j})|\to\infty,\]
for $i\not=j.$ This implies $<\alpha^{k_{i,n_j}-1}_{1,n_j}\alpha_{i,n_j}\alpha^{l_{i,n_j}}_{1,n_j},...,\alpha^{k_{i,n_j}}_{1,n_j}\alpha_{i,n_j}\alpha^{l_{i,n_j}}_{1,n_j},...,
\alpha^{k_{k,n_j}}_{1,n_j}\alpha_{k,n_j}\alpha^{l_{k,n_j}}_{1,n_j}>$ is classical Schottky group for large $n_j.$
\end{proof}
We are done for the subcase $\inf_nf(\alpha_{i,n},\alpha_{1,n})=0.$ Next we will deal the remain subcase of $M<f(\alpha_{i,n},\alpha_{1,n})$ for some $M>0.$ First
we need to estimate the norm of the trace of Nielsen transformed generators $\tilde{\beta}_{i,n}$ with respect to the rate the generators $\tilde{\alpha}_{1,n}$ 
degenerates into elliptical elements, i.e. $|\tilde{\lambda}_{1,n}|$ decreases to 1. The required estimate is given by next Lemma.
\begin{lem}\label{axis-m=0}
Assume that there exists a $M>0$ such that,
\[M<f(\alpha_{i,n},\alpha_{1,n}).\]
Then there exists $\mathcal{N},\sigma>0$ such that, $|\tilde{\lambda}_{1,n}|^{2\mathfrak{D}_n}-1>\sigma|\tr(\tilde{\beta}_{i,n})|^{\frac{2\mathfrak{D}_n}{2\mathfrak{D}_n-1}},$
for $n>\mathcal{N}.$
\end{lem}
\begin{proof}
Use matrix representation we can write, 
$\tilde{\beta}_{i,n}=\bigl(\begin{smallmatrix} \tilde{a}_{i,n} & \tilde{b}_{i,n}\\ \tilde{c_{i,n}} & \tilde{d}_{i,n}\end{smallmatrix}\bigr).$
Note that $\eta_{\tilde{\beta}_{i,n}}=1.$ So $\tilde{c}_{i,n}=-\tilde{d}_{i,n}$ we have,
$\tilde{\beta}_{i,n}=\bigl(\begin{smallmatrix} \tilde{a}_{i,n} & \tilde{b}_{i,n}\\ -\tilde{d_{i,n}} & \tilde{d}_{i,n}\end{smallmatrix}\bigr).$
By $|\zeta_{\tilde{\beta}_{i,n}}|\le 1$ and our assumption of the Lemma 
$M<\frac{|\zeta_{\tilde{\beta}_{{i,n}}}-1|}{|\tilde{\lambda}_{{1,n}}|^2-1},$ 
we have, 
\[ M(|\tilde{\lambda}_{1,n}^2|-1)<\left|\zeta_{\tilde{\beta}_{i,n}}-1\right|<M'.\]
By $|\tr(\tilde{\beta}_{i,n})|\to\infty$ and as in the proof of Lemma \ref{m=0} we have,
\[\lim_n\frac{\left|z_{\tilde{\beta}_{{i,n}},+}-z_{\tilde{\beta}_{{i,n}},-}\right|}{\left|\zeta_{\tilde{\beta}_{{i,n}}}-1\right|}=
\lim_n\left|\frac{\sqrt{|\tr(\tilde{\beta}_{{i,n}})|^2-4}}{\tr(\tilde{\beta}_{{i,n}})}\right|=1\]
Since $|z_{\tilde{\beta}_{i,n},+}-z_{\tilde{\beta}_{i,n},-}|=\left|\tilde{d}_{i,n}\right|^{-1}\left|\sqrt{\tr^2{\tilde{\beta}_{i,n}}-4}\right|$
it follows that we have for some $\sigma, \sigma'>0$ and large $n$ that,
\[\sigma|\tr(\tilde{\beta}_{i,n})|<|\tilde{d}_{i,n}|<\sigma'|\tr(\tilde{\beta}_{i,n})|(|\tilde{\lambda}_{1,n}|^2-1)^{-1}.\]
Let $e$ be the Euler number. If there exists a subsequence such that $\lim_j|\zeta_{\tilde{\beta}_{{i,n_j}}}|=e^{-1},$ 
then $1-e^{-1}<|z_{\tilde{\beta}_{{i,n_j}},+}-z_{\tilde{\beta}_{{i,n_j}},-}|<1+e^{-1}$ for large $n_j.$ Otherwise we let $m_n>0$ be integers defined as,
\[ 1+\frac{1}{m_n+1}\le|\tilde{\lambda}_{1,n}|\le 1+\frac{1}{m_n}.\]
From this definition we have, $\lim_n|\tilde{\lambda}_{1,n}|^{m_n}=\lim_n(1+m_n^{-1})^{m_n}=e.$
Then there exists $N,\delta>0$ such that for $n>N$ we have,
\begin{eqnarray*}
|\zeta_{\tilde{\alpha}^{m_n}_{1,n}\tilde{\beta}_{i,n}}-\eta_{\tilde{\alpha}^{m_n}_{1,n}\tilde{\beta}_{i,n}}|&\le&
\left||\zeta_{\tilde{\beta}_{i,n}}||\tilde{\lambda}_{1,n}|^{m_n}+1\right|\\
&\le& \delta(e+1) \\
\text{and,}\\
|\zeta_{\tilde{\alpha}^{m_n}_{1,n}\tilde{\beta}_{i,n}}-\eta_{\tilde{\alpha}^{m_n}_{1,n}\tilde{\beta}_{i,n}}|&\ge&
\left||\zeta_{\tilde{\beta}_{i,n}}||\tilde{\lambda}_{1,n}|^{m_n}-1\right|\\
&\ge& \delta(e-1). 
\end{eqnarray*}
Hence it follows from Lemma \ref{fix-trace}, there exists $\kappa>0$ such that,
\[ \kappa^{-1}<|z_{\tilde{\alpha}^{m_n}_{1,n}\tilde{\beta}_{i,n},+}-z_{\tilde{\alpha}^{m_n}_{1,n}\tilde{\beta}_{i,n},-}|<\kappa.  \]  
Therefore by setting $m_n=0$ for the subsequence with $\lim_j|\zeta_{\tilde{\beta}_{{i,n_j}}}|=e^{-1},$ we can always assume that for large $n$
\[ \kappa^{-1}<|z_{\tilde{\alpha}^{m_n}_{1,n}\tilde{\beta}_{i,n},+}-z_{\tilde{\alpha}^{m_n}_{1,n}\tilde{\beta}_{i,n},-}|<\kappa.  \]  
Since
\[|z_{\tilde{\alpha}^{m_n}_{1,n}\tilde{\beta}_{i,n},+}-z_{\tilde{\alpha}^{m_n}_{1,n}\tilde{\beta}_{i,n},-}|=
\frac{\left|\sqrt{\tr^2(\tilde{\alpha}^{m_n}_{1,n}\tilde{\beta}_{i,n})-4}\right|}{|\tilde{d}_{i,n}\tilde{\lambda}^{m_n}_{1,n}|},\]
we have
\[\sigma''|\tr(\tilde{\beta}_{i,n})|<|\tr(\tilde{\alpha}^{m_n}_{1,n}\tilde{\beta}_{i,n})|<\frac{\sigma'''|\tr(\tilde{\beta}_{i,n})|}{(|\tilde{\lambda}_{1,n}|^2-1)}.\]
Since $\kappa^{-1}<|z_{\tilde{\alpha}^{m_n}_{1,n}\tilde{\beta}_{i,n},+}-z_{\tilde{\alpha}^{m_n}_{1,n}\tilde{\beta}_{i,n},-}|<\kappa,$ we have
$\dis(\mathcal{L}_{\tilde{\alpha}_{1,n}},\mathcal{L}_{\tilde{\alpha}^{m_n}_{1,n}\tilde{\beta}_{i,n}})<\delta$ for some $\delta>0.$
By Remark 4.1.A of Proposition \ref{trace} applied to $<\tilde{\alpha}_{1,n},\tilde{\alpha}_{1,n}^{m_n}\tilde{\beta}_{i,n}>$ we have,
\[ |\tr(\tilde{\alpha}^{m_n}_{1,n}\tilde{\beta}_{i,n})|>
\rho\left(\frac{|\tilde{\lambda}_{1,n}|^{2\mathfrak{D}_n}+3}{|\tilde{\lambda}_{1,n}|^{2\mathfrak{D}_n}-1}\right)^{\frac{1}{2\mathfrak{D}_n}}\]
and above bound for $|\tr(\tilde{\alpha}^{m_n}_{1,n}\tilde{\beta}_{i,n})|$ we have,
\begin{align*}
|\tr(\tilde{\beta}_{i,n})|&>\rho\sigma'''^{-1}(|\tilde{\lambda}_{1,n}|^2-1)\left(\frac{|\tilde{\lambda}_{1,n}|^{2\mathfrak{D}_n}+3}{|\tilde{\lambda}_{1,n}|^{2\mathfrak{D}_n}-1}\right)^{\frac{1}{2\mathfrak{D}_n}}>
\rho'\frac{|\tilde{\lambda}_{1,n}|^2-1}{\left(|\tilde{\lambda}_{1,n}|^{2\mathfrak{D}_n}-1\right)^{\frac{1}{2\mathfrak{D}_n}}}\\
&>\rho''\left(|\tilde{\lambda}_{1,n}|^{2\mathfrak{D}_n}-1\right)^{1-\frac{1}{2\mathfrak{D}_n}}
\end{align*}
The last inequality implies that,
\[|\tilde{\lambda}_{1,n}|^{2}-1>|\tilde{\lambda}_{1,n}|^{2\mathfrak{D}_n}-1>\rho'''|\tr(\tilde{\beta}_{i,n})|^{\frac{2\mathfrak{D}_n}{2\mathfrak{D}_n-1}}.\]
\end{proof}
Next we are ready to take care the subcase when $f(\alpha_{i,n_j},\alpha_{1,n_j})>M.$
\begin{prop}\label{classical-i-2}
Suppose that there exists a $M>0$ such that,
\[M<f(\alpha_{i,n},\alpha_{1,n}).\]
Then $<\alpha_{1,n}, \alpha^{k_{2,n}-1}_{1,n}\alpha_{2,n}\alpha^{l_{2,n}}_{1,n} ,...,\alpha^{k_{k,n}-1}_{1,n}\alpha_{k,n}\alpha^{l_{k,n}}_{1,n}>$ are classical generators for large $n.$
\end{prop} 
To prove Proposition \ref{classical-i-2} when $\limsup_n|\tr(\alpha^{k_{i,n}-1}_{1,n}\beta_{i,n})|<\infty$ we use disjoint non-isometric circles 
for $\alpha^{k_{i,n}-1}_{1,n}\beta_{i,n}$ based on the following Lemma.
\begin{proof}{(Proposition \ref{classical-i-2})}
First assume that $\limsup_n|\tr( \alpha^{k_{i,n}-1}_{1,n}\alpha_{i,n}\alpha^{l_{i,n}}_{1,n} )|<\infty.$
Let $\mathcal{S}_{o_{i,n},r_{i,n}},\mathcal{S}_{o'_{i,n},r'_{i,n}}$ be the disjoint circles for $\alpha^{k_{i,n}-1}_{1,n}\alpha_{i,n}\alpha^{l_{i,n}}_{1,n} $ given by Proposition \ref{non-iso}.  We will
show that $\lim_n\frac{r_n+r'_n}{|\lambda_{1,n}|^2-1}=0.$\par
Note that
\begin{align*}
\lim_n\frac{r_{i,n}+r'_{i,n}}{|\lambda_{1,n}|^2-1}&\le\lim_n|z_{\alpha^{k_{i,n}-1}_{1,n}\alpha_{i,n}\alpha^{l_{i,n}}_{1,n} ,+}-
z_{\alpha^{k_{i,n}-1}_{1,n}\alpha_{i,n}\alpha^{l_{i,n}}_{1,n},-}|
\frac{2|\tilde{\lambda}_{i,n}|}{(|\tilde{\lambda}_{i,n}|^2-1)(|\lambda_{1,n}|^2-1)}\\
&\text{by $|\tr(\alpha^{k_{i,n}-1}_{1,n}\alpha_{i,n}\alpha^{l_{i,n}}_{1,n} )|<C$ for some $C>0$ we have,}\\
&\le\lim_n|z_{\alpha^{k_{i,n}-1}_{1,n}\alpha_{i,n}\alpha^{l_{i,n}}_{1,n},+}-z_{\alpha^{k_{i,n}-1}_{1,n}\alpha_{i,n}\alpha^{l_{i,n}}_{1,n},-}|
\frac{2C}{(|\tilde{\lambda}_{i,n}|^2-1)(|\lambda_{1,n}|^2-1)}
\end{align*}
Since $\limsup_n|\tr(\alpha^{k_{i,n}-1}_{1,n}\alpha_{i,n}\alpha^{l_{i,n}}_{1,n})|<\infty$ and $\limsup_n|\lambda_{1,n}|^{k_{i,n}-1}<\infty,$ we have by Lemma \ref{fix-bound-2},
\[|z_{\alpha^{k_{i,n}-1}_{1,n}\alpha_{i,n}\alpha^{l_{i,n}}_{1,n} ,+}-z_{\alpha^{k_{i,n}-1}_{1,n}\alpha_{i,n}\alpha^{l_{i,n}}_{1,n} ,-}|
\asymp\frac{1}{|\tr(\alpha_{i,n})|}.\]
By Proposition \ref{trace}, 
\[|\tr(\alpha_{i,n})|^{2\mathfrak{D}_n}>\rho^{\mathfrak{D}_n}\frac{|\lambda_{1,n}|^{2\mathfrak{D}_n}+3}{|\lambda_{1,n}|^{2\mathfrak{D}_n}-1}>\frac{4\rho^{\mathfrak{D}_n}}{|\lambda_{1,n}|^2-1}.\]
The last inequality follows from $|\lambda_{1,n}|^{2\mathfrak{D}_n}-1<|\lambda_{1,n}|^2-1$ for large $n.$\par
By our assumptiom that $M<\frac{|\zeta_{\tilde{\beta}_{i,n}}-1|}{|\tilde{\lambda}_{i,n}|^2-1},$ and Lemma \ref{axis-m=0} we have,
\begin{align*}
\lim_n\frac{r_{i,n}+r'_{i,n}}{|\lambda_{1,n}|^2-1}&\le\lim_n (4\rho^{\mathfrak{D}_n})^{-1}M'|\tr(\alpha_{i,n})|^{-1}|\tr(\alpha_{i,n})|^{2\mathfrak{D}_n}|\tr(\tilde{\beta}_{i,n})|^{\frac{2\mathfrak{D}_n}{1-2\mathfrak{D}_n}}\\
&\text{since $\quad |\tr(\tilde{\beta}_{i,n})|<\sigma|\tr(\alpha_{i,n})|$ we have,}\\
&<\lim_n (4\rho^{\mathfrak{D}_n})^{-1}\sigma^{\frac{2\mathfrak{D}_n}{1-2\mathfrak{D}_n}}M'|\tr(\alpha_{i,n})|^{\frac{\mathfrak{D}_n(6-4\mathfrak{D}_n)-1}{1-2\mathfrak{D}_n}}=0
\end{align*}
Now the circles $\mathcal{S}_{o_{i,n},r_{i,n}}$ contains one of $z_{\alpha^{k_{i,n}-1}_{1,n}\alpha_{i,n}\alpha^{l_{i,n}}_{1,n},-},
z_{ \alpha^{k_{i,n}-1}_{1,n}\alpha_{i,n}\alpha^{l_{i,n}}_{1,n} ,+}$ and $\mathcal{S}_{o'_{i,n},r'_{i,n}}$ contains
the other fixed point, and since $|\eta_{\alpha^{k_n-1}_n\beta_n}|\to1$ and $|\zeta_{\alpha^{k_n-1}_n\beta_n}|\to 1,$ 
it follows from Lemma \ref{fix-trace}, we must have $\mathcal{S}_{o_{i,n},r_{i,n}}, \mathcal{S}_{o'_{i,n},r'_{i,n}}$
contained in the region between $\frac{1}{|\lambda_{1,n}|}$ and $|\lambda_{1,n}|$ for large $n.$ Hence we have classical generators for large $n.$\par
Now if there exists a subsequence such that $|\tr( \alpha^{k_{i,n}-1}_{1,n}\alpha_{i,n}\alpha^{l_{i,n}}_{1,n} )|\to\infty,$ 
then $<\alpha_{1,n},...,\alpha^{k_{k,n}-1}_{1,n}\alpha_{k,n}\alpha^{l_{k,n}}_{1,n}>$ satisfies
conditions of Lemma \ref{classical}.
\end{proof}
Hence the case $(i_2)$ is complete. Next we consider Type I$_2.$ This is relatively straightforward and parallel to case Type I$_1$  but we provide another short proof in the following.
\subsection{Type I$_2$}
For the Type I$_2$ it can be seen that we can certainly follow similarly procedure as given in our proof of the Type I$_1$ 
and just simply do some modifications. In fact many of our previous estimates will be 
became much simpler by $|\lambda_{1,n}|>\lambda$ hence all estimates that requires bounds on $(|\lambda_{1,n}|^2-1)^{-1}$ are trivial. 
To avoid much repetition of our previous proof of the Type $I_1$  we therefore present a alternative and somewhat much shorter  proof of Type I$_2.$ \par

\begin{lem}\label{delta-z}
Let $\alpha_n,\beta_n$ be a sequence of loxodromic transformations with fixed points of $\alpha_n$ be $0,\infty.$ 
Suppose $|z_{\beta_n,+}-z_{\beta_n,-}|\to 0$ and $C^{-1}<|z_{\beta_n,\pm}|<C.$
Then there exists $\delta>0$ such that 
\[\dis(\mathcal{L}_{\alpha_n},\mathcal{L}_{\beta_n})
<\log\left(\frac{\delta}{|z_{\beta_n,+}-z_{\beta_n,-}|}\right).\]
\end{lem}
\begin{proof}
By using hyperbolic distance formula, and since 
$|z_{\beta_n,+}-z_{\beta_n,-}|\to 0$ and
$|z_{\beta_n,\pm}|\to 1.$ we have for large $n$,
\begin{eqnarray*}
\cosh\dis(\mathcal{L}_{\alpha_n},\mathcal{L}_{\beta_n})&<&\frac{|z_{\beta_n,u}|^2+\frac{1}{4}(|z_{\beta_n,u}+z_{\beta_n,l}|)^2}
{\frac{1}{2}(|z_{\beta_n,l}|+\frac{1}{2}|z_{\beta_n,u}-z_{\beta_n,l}|) (|z_{\beta_n,u}-z_{\beta_n,l}| )}+1\\
&<&\frac{\rho}{|z_{\beta_n,u}-z_{\beta_n,l}|}, \quad\text{for some $\rho>0$}
\end{eqnarray*}
\end{proof}
\begin{proof}{(Type I$_2$)}
Suppose that there is a subsequence (use same index for subsequence) such that $|\tr( \alpha^{k_{i,n}}_{1,n}\alpha_{i,n}\alpha^{l_{i,n}}_{1,n} )|\to\infty.$ 
Then by our assumption of
type I$_2,$ $|\zeta_{\alpha^{k_{i,n}}_{1,n}\alpha_{i,n}\alpha^{l_{i,n}}_{1,n}}|\to 1$ and $|\lambda_{1,n}|>\lambda>1$ we have for large $n$,  
\begin{itemize}
\item
$\lambda^{-1}< |\eta_{\alpha^{k_{i,n}}_{1,n}\alpha_{i,n}\alpha^{l_{i,n}}_{1,n}}| \le|\zeta_{\alpha^{k_{i,n}}_{1,n}\alpha_{i,n}\alpha^{l_{i,n}}_{1,n}}|
<\lambda,$ and
\item
\[
0=\lim_n\begin{cases}
\frac{1}{\left|\tr\left(\alpha^{k_{i,n}}_{1,n}\alpha_{i,n}\alpha^{l_{i,n}}_{1,n}\right)\right|
\left(\left|\eta_{\psi_{n;2}\alpha^{k_{i,n}}_{1,n}\alpha_{i,n}\alpha^{l_{i,n}}_{1,n}\psi_{n;2}^{-1}}\right|-\left|\lambda_{1,n}\right|^{-1}\right)}\\
\frac{1}{\left|\tr\left( \alpha^{k_{i,n}}_{1,n}\alpha_{i,n}\alpha^{l_{i,n}}_{1,n} \right)\right|\left(\left|\zeta_{\psi_{n;2}\alpha^{k_{i,n}}_{1,n}
\alpha_{i,n}\alpha^{l_{i,n}}_{1,n}\psi_{n;2}^{-1}}\right|-\left|\lambda_{1,n}\right|\right)}
\end{cases}\]

\end{itemize}
By Lemma \ref{fix-trace}, $<\alpha_{1,n},\alpha^{k_{2,n}}_{1,n}\alpha_{2,n}\alpha^{l_{2,n}}_{1,n},...,\alpha^{k_{k,n}}_{1,n}\alpha_{k,n}\alpha^{l_{k,n}}_{1,n}>$ satisfies the second set of conditions of Lemma \ref{classical}, hence
classical.\par
Otherwise we have $|\tr(\alpha^{k_{i,n}}_{1,n}\alpha_{i,n}\alpha^{l_{i,n}}_{1,n})|<C$ for some $C>0.$ Since 
$|\zeta_{\alpha^{k_{i,n}}_{1,n}\alpha_{i,n}\alpha^{l_{i,n}}_{1,n}}|$ $\to 1$ and $\eta_{\alpha^{k_{i,n}}_{1,n}\alpha_{i,n}\alpha^{l_{i,n}}_{1,n}}\to 1,$
by Lemma \ref{fix-trace} we have $|z_{\alpha^{k_{i,n}}_{1,n}\alpha_{i,n}\alpha^{l_{i,n}}_{1,n},\pm}|\to 1.$ Now by Remark \ref{4A}.1.A and 
$|\tr(\alpha^{k_{i,n}}_{1,n}\alpha_{i,n}\alpha^{l_{i,n}}_{1,n})|<C$ 
we must have $\dis(\mathcal{L}_{\alpha_{1,n}},\mathcal{L}_{\alpha^{k_{i,n}}_{1,n}\alpha_{i,n}\alpha^{l_{i,n}}_{1,n}})$ $\to\infty.$ This implies that 
$|z_{\alpha^{k_n}_n\beta_n,+}-z_{\alpha^{k_n}_n\beta_n,-}|\to 0.$ More precisely we have,
Let $\mathcal{S}_{o,r_{i,n}},\mathcal{S}_{o',r'_{i,n}}$ be the circles given by Proposition \ref{non-iso}.
\begin{prop}\label{delta-z-r}
$|\tr(\alpha^{k_{i,n}}_{1,n}\alpha_{i,n}\alpha^{l_{i,n}}_{1,n})|<C$ for some $C>0.$ 
Then we must have $(r_{i,n}+r'_{i,n})\to 0.$
\end{prop}
\begin{proof}
First note that as we have showed $|\tr(\alpha^{k_{i,n}}_{1,n}\alpha_{i,n}\alpha^{l_{i,n}}_{1,n})|<C$ implies,
\[|z_{\alpha^{k_{i,n}}_{1,n}\alpha_{i,n}\alpha^{l_{i,n}}_{1,n},+}-z_{\alpha^{k_{i,n}}_{1,n}\alpha_{i,n}\alpha^{l_{i,n}}_{1,n},-}|\to 0\quad\text{and}
\quad |z_{\alpha^{k_{i,n}}_{1,n}\alpha_{i,n}\alpha^{l_{i,n}}_{1,n},\pm}|\to 1. \]
Set $\psi_{i,n}(x)=\frac{x-z_{\alpha^{k_{i,n}}_{1,n}\alpha_{i,n}\alpha^{l_{i,n}}_{1,n},+}}{x-z_{\alpha^{k_{i,n}}_{1,n}\alpha_{i,n}\alpha^{l_{i,n}}_{1,n},-}}.$ 
Let $\lambda_{\alpha^{k_{i,n}}_{1,n}\alpha_{i,n}\alpha^{l_{i,n}}_{1,n}}$ be the mutiplier of
$\psi_{i,n}\alpha^{k_{i,n}}_{1,n}\alpha_{i,n}\alpha^{l_{i,n}}_{1,n}\psi^{-1}_{i,n}.$ By Remark \ref{4A}.1.A applied to 
$\psi_{i,n}<\alpha_{i,n},\alpha^{k_{i,n}}_{1,n}\alpha_{i,n}\alpha^{l_{i,n}}_{1,n}>\psi^{-1}_{i,n}$ we have,
\[|\lambda_{\alpha_{1,n}}|^2>\left(\frac{|\lambda_{\alpha^{k_{i,n}}_{1,n}\alpha_{i,n}\alpha^{l_{i,n}}_{1,n}}|^{2\mathfrak{D}_n}+3}
{|\lambda_{\alpha^{k_{i,n}}_{1,n}\alpha_{i,n}\alpha^{l_{i,n}}_{1,n}}|^{2\mathfrak{D}_n}-1}\right)^{\frac{1}{\mathfrak{D}_n}}
\left(e^{-2\dis(\mathcal{L}_{\psi_{i,n}\alpha_{1,n}\psi_{i,n}^{-1}},\mathcal{L}_{\psi_{i,n}\alpha^{k_{i,n}}_{1,n}\alpha_{i,n}\alpha^{l_{i,n}}_{1,n}\psi^{-1}_{i,n}})}\right). \] 
Since $\{z_{\psi_{i,n}\alpha_{1,n}\psi^{-1}_{i,n},+},z_{\psi_{i,n}\alpha_{1,n}\psi^{-1}_{i,n},-}\}=
\{1,\frac{z_{\alpha^{k_{i,n}}_{1,n}\alpha_{i,n}\alpha^{l_{i,n}}_{1,n},+}}{z_{\alpha^{k_{i,n}}_{1,n}\alpha_{i,n}\alpha^{l_{i,n}}_{1,n},-}}\}$
we have,
\[|z_{\psi_{i,n}\alpha_{1,n}\psi_{i,n}^{-1},\pm}|\to 1,\quad\text{and}\] 
\[\left|z_{\psi_{i,n}\alpha_{1,n}\psi^{-1}_{i,n},+}-z_{\psi_{i,n}\alpha_{1,n}\psi^{-1}_{i,n},-}\right|=
\left|1-\frac{z_{\alpha^{k_{i,n}}_{1,n}\alpha_{i,n}\alpha^{l_{i,n}}_{1,n},+}}{z_{\alpha^{k_{i,n}}_{1,n}\alpha_{i,n}\alpha^{l_{i,n}}_{1,n},-}}\right|\to 0.\]
By Lemma \ref{delta-z} we have for large $n$,
\[\dis(\mathcal{L}_{\psi_{i,n}\alpha_{1,n}\psi_{i,n}^{-1}},\mathcal{L}_{\psi_{i,n}\alpha^{k_{i,n}}_{1,n}\alpha_{i,n}\alpha^{l_{i,n}}_{1,n}\psi^{-1}_{i,n}})
<\log\left(\frac{\delta}
{|z_{\psi_{i,n}\alpha_{1,n}\psi^{-1}_{i,n},+}-z_{\psi_{i,n}\alpha_{1,n}\psi^{-1}_{i,n},-}|}\right).\]
This implies that for large $n$,
\begin{eqnarray*}
|\lambda_{\alpha_{1,n}}|&>&\delta^{-1}\left(\frac{ |\lambda_{\alpha^{k_{i,n}}_{1,n}\alpha_{i,n}\alpha^{l_{i,n}}_{1,n}}|^{2\mathfrak{D}_n}+3} 
{|\lambda_{\alpha^{k_{i,n}}_{1,n}\alpha_{i,n}\alpha^{l_{i,n}}_{1,n}}|^{2\mathfrak{D}_n}-1}\right)^{\frac{1}{2\mathfrak{D}_n}}
 |z_{\psi_{i,n}\alpha_{1,n}\psi^{-1}_{i,n},+}-z_{\psi_{i,n}\alpha_{1,n}\psi^{-1}_{i,n},-}|\\
&>&\delta^{-1}\left(|\lambda_{\alpha^{k_{i,n}}_{1,n}\alpha_{i,n}\alpha^{l_{i,n}}_{1,n}}|^{2\mathfrak{D}_n}+3\right)^{\frac{1}{2\mathfrak{D}_n}}
\left(\frac{|z_{\psi_{i,n}\alpha_{1,n}\psi^{-1}_{i,n},+}-z_{\psi_{i,n}\alpha_{1,n}\psi^{-1}_{i,n},-}|}
{|\lambda_{\alpha^{k_{i,n}}_{1,n}\alpha_{i,n}\alpha^{l_{i,n}}_{1,n}}|^2-1}\right)\\
&>&\delta^{-1}\left(|\lambda_{\alpha^{k_{i,n}}_{1,n}\alpha_{i,n}\alpha^{l_{i,n}}_{1,n}}|^{2\mathfrak{D}_n}+3\right)^{\frac{1}{2\mathfrak{D}_n}}\frac{(r_n+r'_n)}
{2C'|z_{\alpha^{k_{i,n}}_{1,n}\alpha_{i,n}\alpha^{l_{i,n}}_{1,n},-}|}
\end{eqnarray*}
The last inequality follows from Proposition \ref{non-iso} and $|\lambda_{\alpha^{k_{i,n}}_{1,n}\alpha_{i,n}\alpha^{l_{i,n}}_{1,n}}|<C'.$ The second inequality in the above calculations follows from that 
$|\lambda_{\alpha^{k_{i,n}}_{1,n}\alpha_{i,n}\alpha^{l_{i,n}}_{1,n}}|<C'$
and for large $n$ we have, 
\[\left(|\lambda_{\alpha^{k_{i,n}}_{1,n}\alpha_{i,n}\alpha^{l_{i,n}}_{1,n}}|^{2\mathfrak{D}_n}-1\right)^{\frac{1}{2\mathfrak{D}_n}}
\le|\lambda_{\alpha^{k_{i,n}}_{1,n}\alpha_{i,n}\alpha^{l_{i,n}}_{1,n}}|^{2\mathfrak{D}_n}-1\le|\lambda_{\alpha^{k_{i,n}}_{1,n}\alpha_{i,n}\alpha^{l_{i,n}}_{1,n}}|^2-1.\]
Since $|\lambda_{\alpha_{1,n}}|<M$ for some $M,$ hence by we have,
\[r_{i,n}+r'_{i,n}<\frac{2C'M\delta|z_{\alpha^{k_{i,n}}_{1,n}\alpha_{i,n}\alpha^{l_{i,n}}_{1,n},-}|}
{\left(|\lambda_{\alpha^{k_{i,n}}_{1,n}\alpha_{i,n}\alpha^{l_{i,n}}_{1,n}}|^{2\mathfrak{D}_n}+3\right)^{\frac{1}{2\mathfrak{D}_n}}}<\frac{2C'M\delta'}{4^{\frac{1}{2\mathfrak{D}_n}}}\to 0.\]
\end{proof}
Now we can continue and finish the proof for $|\tr(\alpha^{k_{i,n}}_{1,n}\alpha_{i,n}\alpha^{l_{i,n}}_{1,n})|<C.$ By Proposition \ref{delta-z-r} and 
$|\lambda_{1,n}|>\lambda>1$ (this is condition of type I$_2$) we have, $\frac{r_{i,n}+r'_{i,n}}{|\lambda_{1,n}|^2-1}\to 0.$ 
Since the circles $\mathcal{S}_{o_{i,n},r_{i,n}}$ contains one of $z_{\alpha^{k_{i,n}}_{1,n}\alpha_{i,n}\alpha^{l_{i,n}}_{1,n},-},
z_{ \alpha^{k_{i,n}}_{1,n}\alpha_{i,n}\alpha^{l_{i,n}}_{1,n} ,+}$ and $\mathcal{S}_{o'_{i,n},r'_{i,n}}$ contains
the other fixed point, and $\eta_{\alpha^{k_{i,n}}_{1,n}\alpha_{i,n}\alpha^{l_{i,n}}_{1,n}}\to 1$ and by condition I$_2$ we have
$|\zeta_{\alpha^{k_{i,n}}_{1,n}\alpha_{i,n}\alpha^{l_{i,n}}_{1,n}}|\to 1,$ it follows from Lemma \ref{fix-trace},
 we must have $\mathcal{S}_{o_{i,n},r_{i,n}}, \mathcal{S}_{o'_{i,n},r'_{i,n}}$
contained in the region between $\frac{1}{|\lambda_{1,n}|}$ and $|\lambda_{1,n}|$ for large $n.$ Hence we have classical generators for large $n.$
This completes the proof for type I$_2$.\par
\end{proof}
\section{Single boundary degeneracy Type II}
We show this type of degeneracy is classical by reducing to type I for large $n.$\par
Set Mobius transformations $\psi_{n;3}(x)=\min_{2\le i\le k}|\zeta^{-\frac{1}{2}}_{\alpha_{i,n}}\lambda^{1-k_{i,n}}_{\lambda_{1,n}}| x$ and consider the generators
$\psi_{n;3}<\alpha_{1,n},...,\alpha^{k_{k,n}-1}_{1,n}\alpha_{k,n}\alpha_{1,n}^{l_{k,n}}>\psi^{-1}_{n;3}.$ Then we have 
$1\le|\zeta_{\psi_{n;3}\alpha_{1,n}^{k_{i,n}-1}\alpha_{i,n}\psi_{n;3}^{-1}}| 
\le|\eta_{\psi_{n;3}\alpha_{1,n}^{k_{i,n}-1}\alpha_{i,n}\psi_{n;3}^{-1}}|\le|\lambda_{1,n}|^2.$ Since $1\le |\zeta_{\alpha_{i,n}\lambda_{1,n}^{2k_{i,n}}}|<|\lambda_{1,n}|^2$ and
$|\zeta_{\alpha_{i,n}\lambda_{1,n}^{2k_{i,n}}}|-|\lambda_{1,n}^2|\to 0$ (condition of $(B)$) we have, 
\[|\zeta_{\alpha_{i,n}}\lambda_{1,n}^{2k_{i,n}}|(1-|\zeta^{-1}_{\alpha_{i,n}}\lambda_{1,n}^{2-2k_{i,n}}|)\to 0.\]
Since $|\lambda_{1,n}|<M$ for some $M>0,$ we have $|\eta_{\psi_{n;3}\alpha^{k_{i,n}-1}_{1,n}\alpha_{i,n}\psi^{-1}_{n;3}}|\to 1.$ Hence by considering
$<\psi_{n;3}\alpha_{1,n}\psi^{-1}_{n;3},...,\psi_{n;3}\alpha^{-l_{k,n}}_{1,n}\alpha^{-1}_{k,n}\alpha^{1-k_{k,n}}_{1,n}\psi^{-1}_{n;3}>,$ we have
\[\left||\zeta_{\psi_{n;3}\alpha_{1,n}^{-l_{i,n}}\alpha^{-1}_{i,n}\alpha^{1-k_{i,n}}_{1,n}\psi^{-1}_{n;3}}|-
|\eta_{\psi_{n;3}\alpha_{1,n}^{-l_{i,n}}\alpha^{-1}_{i,n}\alpha^{1-k_{i,n}}_{1,n}\psi^{-1}_{n;3}}|\right|\to 0,\] and 
$|\eta_{\psi_{n;3}\alpha^{-l_{i,n}}_{1,n}\alpha^{-1}_{i,n}\alpha^{1-k_{i,n}}_{1,n}\psi^{-1}_{n;3}}|\to 1.$
We have reduced Type II to Type I.\par 
\section{Elliptical degeneracy Type III}
We show this type of degeneracy is classical by reducing to type I for large $n.$\par
Since $1\le|\eta_{\alpha^{k_{i,n}}_{1,n}\alpha_{i,n}\alpha^{l_{i,n}}_{1,n}}|\le|\zeta_{\alpha^{k_{i,n}}_{1,n}\alpha_{i,n}\alpha^{l_{i,n}}_{1,n}}|,$ 
Type III implies Type I.
In this can trivially reduce to Type I.\par
Now we have completed the proof for Types I, II, III, IV. next we will link them together and consider the general case without any restrictions on $\alpha_{1,n}.$

\section{Non-Collapsing fixed points subspaces}
The main result in this section is the following,  when we have a sequence of $\Gamma_n$ of Schottky groups which is $k-1$ classical and $Z_{\G_n}>\tau$ for some $\tau>0$, such that $\mathfrak{D}_{\G_n}\to 0$ then, $\G_n$ will contain a subsequence which will be classical Schottky groups for large enough $n.$
Given a set of generators $<\alpha_{1},...,\alpha_{k}>$ we define Nielsen maps $F_{i,l,m}$ for integers $1\le i\le k$ and natural number $l,m$ as:
$F_{i,l,m}\alpha_{j,k}=\alpha^l_i\alpha_{j}\alpha^m_i.$\par
First we show that given $\G_n$ sequence of $k-1$ classical Schottky groups with $\mathfrak{D}_n\to 0$, we can always choose generators so that the $k-1$ classical generators are of standard form, which means that centers of isometric circles are bounded between $1$ and $\alpha_{1,n}$ for $n$ sufficiently large.

\begin{lem}\label{k-1-form}
Let $\G_n$ be a sequence of rank $k$ Schottky group which is $k-1$ classical and $\mathfrak{D}_{\G_n}\to 0.$ Then we can find a subsequence of generating set $S_n$ of 
$\G_n$, up to Mobius conjugation, such that $\{\alpha_{2,n},...,\alpha_{k,n}\}\subset S_n$ is a sequence of $k-1$ classical generators and, $S_n$ is of standard form. 
\end{lem}
\begin{proof}
Since $\G_n$ is $k-1$ classical and $\mathfrak{D}_{\G_n}\to 0$, it follows from Prop \ref{normal-def}, we have a normal subsequence of generating sets (we using same index for subsequence) $<\alpha_{1,n},\alpha_{2,n},...,\alpha_{k,n}>$ of $\G_n.$ As before there exists integers $l_{1,j,n},m_{1,j,n}$ such that we can put these sequence of generating sets in to standard form by $F_{1,l_{1,j,n},m_{1,j,n}}$ applied to $\alpha_{j,n}$ for $2\le j\le k.$  We can assume that the all centers of isometric circles of $F_{1,l_{1,j,n},m_{1,j,n}}\alpha_{j,n}$ are strictly bounded away from $|\lambda_{1,n}|^2.$ Note that the new sequence will not necessarily be $k-1$ classical generators. \par
Suppose this is the case. Since $\G_n$ is $k-1$ classical sequence, there exists classical generating set for the $k-1$ generators not including $\alpha_{1,n}.$ 
In fact, there exists $F_{i,l_{i,j,n},m_{i,j,n}}$ such that elements of this classical generators will be of the form $F_{i,l_{i,j,n},m_{i,j,n}}\alpha_{j,n}$ for $i\not= 1$ and integers $l_{i,j,n},m_{i,j,n}$ depends on indices $i,j,n.$ Note that if it's not of standard form then, by repeating this process for a fixed $n$, we have generating set with generators with $|\mbox{trace}|\to\infty.$ The generators constructed will have leading term of the form 
$F_{i,l_{i,j,n},m_{i,j,n}}F_{1,l_{1,j,n},m_{1,j,n}}$ where $i\not=1$ and some integers $l_{i,j,n},m_{i,j,n}.$ In particular continue this process, 
we have generators of arbitrarily small isometric circles. And since we only terminate when we have elements of leading term of the form $F_{1,l_{1,j,n},m_{1,j,n}}F_{i,l_{i,j,n},m_{i,j,n}}$ becomes $k-1$ classical generators, we have every 
$F_{1,l_{1,j,n},m_{i,j,n}}F_{i,l_{1,j,n},m_{i,j,n}}$ is distinct for the repetition. This implies all fixed points of these $k-1$ generators are arbitrarily close to centers of circles. \par
For sufficiently many repetitions, $<F_{i,l_{i,j,n},m_{i,j,n}}\alpha'_{j,n}>$ for $i\ge2$ are classical generators of rank $k-1$ Schottky group $\G$ with each generator's fixed points and isometric centers arbitrarily close. Since $<\alpha'_{j,n}>$ is also generators of $\G$ which must also of fixed points and isometric centers arbitrarily close,
and $<F_{i,l_{i,j,n},m_{i,j,n}}\alpha'_{j,n}>$ is classical, we must have fixed points of $<\alpha'_{j,n}>$ contained within isometric circles of $F_{i,l_{i,j,n},m_{i,j,n}}\alpha'_{j,n}.$ 
Since isometric circle centers of $\alpha'_{j,n}$ are bounded within $1, |\lambda_{1,n}|^2$ and strictly away from $|\lambda_{1,n}|^2$, we must have isometric centers of $F_{i,l_{i,j,n},m_{i,j,n}}\alpha'_{j,n}$ also bounded arbitrarily close to $1$ and away from  $|\lambda_{1,n}|^2$ for sufficiently many repetitions. Hence
we can find a Mobius map (scaling) such that, fixes $0,\infty$ and moves all isometric centers of $F_{i,l_{i,j,n},m_{i,j,n}}\alpha'_{j,n}$ to be bounded within 
$1,|\lambda_{1,n}|^2.$
Therefore the repetition process must terminate 
with $k-1$ classical generators $\alpha'_{j,n}$ with centers bounded within $1, |\lambda_{1,n}|^2,$ hence a generating set of $k-1$ classical and standard form.\par

\end{proof}
Given $\G_n$ sequence of $k-1$ classical Schottky group with $\mathfrak{D}_n\to 0,$  Lemma \ref{k-1-form} provides generators of standard form with $k-1$ classical for large $n.$ However note that, the circles of the $k-1$ generators may not be disjointed from circles of $\alpha_{1,n}$, hence they don't necessarily constitute classical generating set for $\G_n$.
This is due to the fact we have degenerate types I, II, III, IV. 
\begin{thm}\label{t-space}
Let $\mf{J}_k$ be the rank $k$ Schottky space.  
For each $\tau>0$ there exists a $\nu>0$ such that, 
\[\{[\G]\in\mf{J}_k(\tau)|  \mathfrak{D}_\G\le\nu, \G \text{ is } k-1 \text{ classical},  \text{ for some }\G\in[\G] \}\subset\mf{J}_{k,o}.\]
\end{thm}

\begin{proof}
We proof by contradiction. It follows from argument before for $(I)$ we only need to assume that $|\alpha_{1,n}|<C$ for some $C>0.$ \par
Now it follows from Lemma \ref{k-1-form}, we have sequence of generating sets $S_n$ of $\G_n$ such that it is $k-1$ classical and of standard form. Since 
$\alpha_{j,n}$, for $2\le j\le k$ are classical generators and have centers bounded within $1$ and $|\lambda_{\alpha_{1,n}}|^2$ (standard form), the only issues that
prevents $<\alpha_{1,n},...,\alpha_{k,n}>$ been a classical generators are degeneracies described earlier. Hence, we have reduced the general case to degenerate type I, II, III, IV. Now it follows from that degenerate types are all in fact classical generators for $\mathfrak{D}_{\G_n}$ sufficiently small, i.e large $n$, which gives our result.

\end{proof}

\section{Collapsing fixed points types} 
This section and following sections are devoted in proving Theorem \ref{2-fixed-point}. The Theorem \ref{2-fixed-point} will enable us to remove the constraint on $Z_{\G_n}.$ that was placed in the Theorem \ref{t-space}. The idea of the proof is based on analysis of the behavior of the set of \emph{fixed points} of generators. This should be compared to 
the previous section where analysis of centers of isometric circles were used.\par 
Recall Theorem \ref{t-space} states that a given Schottky group $\G$ is classical Schottky group if: $Z_{<\alpha_{1},...,\alpha_{k}>}>\tau$ and Hausdorff dimension 
$\mathfrak{D}_\G$ is sufficiently small and $\G$ is $k-1$ classical. Now in this section Theorem \ref{2-fixed-point}, we will state and start it's proof that, for a given Schottky group 
$\G$ which is $k-1$ classical of small Hausdorff dimension $\mathfrak{D}_\G$ we have the following dichotomy: 
\begin{itemize}
\item
$Z_{<\alpha_{1,n},...,\alpha_{j,n}>}>c$ for some $c>0,$ 
\item
 is classical Schottky group. 
\end{itemize}
Therefore, combining with Theorem \ref{t-space}, we have the proof that Schottky groups of fixed rank $k$ which is $k-1$ classical and with sufficiently small Hausdorff dimension are classical Schottky groups.
\begin{thm}\label{2-fixed-point}
There exists $c>0$ such that: Let $\Gamma_n$ be a sequence of rank $k$ Schottky groups that is $k-1$ classical with $\mathfrak{D}_n\to 0$.
Then for sufficiently large $n$, there exists a subsequence (using same index) $\Gamma_{n}$ with generating set 
$<\alpha_{1,n},\alpha_{2,n},...,\alpha_{k,n}>$ such that:
\begin{itemize}
\item
$Z_{<\alpha_{1,n},...,\alpha_{k,n}>}>c$ or, 
\item
$<\alpha_{1,n},\alpha_{2,n},...,\alpha_{k,n}>$ is classical Schottky groups for large $n.$
\end{itemize}
\end{thm}
The idea of the proof of Theorem \ref{2-fixed-point} is to use Theorem \ref{t-space} and show by contradiction. The proof involves analysis of collapsing of fixed points given below as Collapsing I and II. First we pick an sequence of generating sets $S_{\Gamma_n}$ of $\Gamma_n$ which is $k-1$ classical. 
 From $S_{\Gamma_n}$ we will examine when $S_{\Gamma_{n}}$ dose not satisfies conditions of Theorem \ref{t-space}. In each situation where $\Gamma_n\not\in\mathfrak{J}_k(\tau)$ we can always obtain new sequence of generating sets such that it will leads to contradictions.\par
The proof will occupy next two sections of Collapsing fixed points I, II.
\begin{proof}
We prove by contradiction. Suppose there exists $\G_n$ a sequence of Schottky groups which is $k-1$ classical such that
for every generating set $<\alpha_{1,n},...,\alpha_{k,n}>$ of $\Gamma_n$  we have $Z_{<\alpha_{1,n},...,\alpha_{k,n}>}\to 0$.\par
For each $n$, by replacing,
\[<\alpha_{1,n},...,\alpha_{k,n}>\quad\text{with}\quad <\alpha_{1,n},\alpha^{m_{2,n}}_{1,n}\alpha_{2,n},...,\alpha^{m_{k,n}}_{1,n}\alpha_{k,n}>,\] 
for sufficiently large $m_{k,n}$ 
if necessary, we can always assume that every $\Gamma_n$ is generated by generators with 
$|\tr(\alpha_{1,n})|\le|\tr(\alpha_{k,n})|$ and $|\tr(\alpha_{k,n})|\ge\frac{\log3}{\mathfrak{D}_n}$. \par
We use $\mathbb{H}^3$ the upper space model. Conjugating with Mobius transformations
it is sufficient to assume $\alpha_{1,n}$ with fixed points $0,\infty$ and multiplier $\lambda_{1,n}$ and
$\alpha_{i,n}, i\ge 2$ also $|\lambda_{1,n}|^{-2}<|\eta_{\alpha_{i,n}}|\le1$. 
Note that as before we denote by $|z_{\alpha_{i,n},l}|\le|z_{\alpha_{i,n},u}|$ as the two fixed points of
$\alpha_{i,n}$ by $z_{\alpha_{i,n},l},z_{\alpha_{i,n},u}.$ Whenever we have
$\alpha_{i,n}$ in matrix form then we always assume $|a_{i,n}|\le|d_{i,n}|$ otherwise just replace $\alpha_{i,n}$ by $\alpha^{-1}_{i,n}$.\par
Note that by Lemma \ref{k-1-form}, we can assume $S_{\G_n}$ satisfies is $k-1$ classical of standard form. Hence
we will always take such $S_{\G_n}$ to be the generators set.\par
We when $Z_{\G_n}\to 0$, which we say \emph{Collapsing fixed points.} This can be divided into possible collapsing types: 
\begin{itemize}
\item 
Collapsing fixed points I: $|z_{\alpha_{i,n},u}-z_{\alpha_{i,n},l}|\to0$ for some $i\ge 2.$  
\item 
Collapsing fixed points II: $z_{\alpha_{i,n},l}\to 0$ for some $i\ge2.$ 
\end{itemize}
After we analyzed both Collapsing fixed points I, II we will then address the general situation. We will address these collapsing fixed points in next two sections. \par
\section{Collapsing fixed points I} 
We show that in this collapsing situation we have further additional I$_1$ and I$_2$ such that: 
\begin{itemize}
\item
I$_1$ implies $Z_{\G_n}>c$ and, 
\item
I$_2$ implies classical.
\end{itemize}
Where Collapsing fixed points I is dichotomized into: 
\begin{itemize}
\item[I$_1$:] 
$\liminf_n|\lambda_{1,n}|=1.$
\item[I$_2$:] 
There some exists
$\lambda>1$ so that $|\lambda_{1,n}|>\lambda.$ 
\end{itemize}

Let us first analyze subclass I$_2$ which is straightforward. 
\subsection{I$_2$}
Since  $|z_{\alpha_{i,n},l}-z_{\alpha_{i,n},u}|\to 1$ and $|\lambda_{1,n}|>\lambda>1,$ we have
$\frac{1}{\lambda}<|z_{\alpha_{i,n},l}|\le|z_{\alpha_{i,n},u}|<\lambda$ for large $n.$ Hence by $<\alpha_{1,n},...,\alpha_{k,n}>$ is $k-1$ classical of standard form  we have it satisfies Lemma \ref{classical} for large $n.$
\subsection{I$_1$}
In this case we will need to perform transformations on $\alpha_{i,n}, i\ge 2$ based on the rate of  $\lambda_{1,n}$ degenerates into $1,$ i.e. $\alpha_{1,n}$ collapsing into an elliptical element. This is done by divid the rate of degeneracy of  $\lambda_{1,n}$ into smaller and smaller intervals, and then transform the
$\alpha_{i,n},i\ge 2$ based on the divided exponents. \par
By take a subsequence if necessary we can just assume that $|\lambda_{1,n}|$ strictly decreasing
to $1$. For large sufficiently enough $n$, we can choose a sequence of positive integers $m_n$ which depends on index $n$ so that 
$1+\frac{1}{m_n+1}\le |\lambda_{1,n}|\le 1+\frac{1}{m_n}$. Set $\zeta_{\alpha_{i,n}}=\frac{a_{i,n}}{c_{i,n}},\eta_{\alpha_{i,n}}=\frac{-d_{i,n}}{c_{i,n}},$
and since $|\tr(\alpha_{i,n})|\to\infty$ and $|\frac{\sqrt{\tr^2(\alpha_{i,n})-4}}{2c_{i,n}}|=|z_{\alpha_{i,n},+}-z_{\alpha_{i,n},-}|\to 0,$
we have $|\tr(\alpha_{i,n})|<|c_{i,n}|$ then from Lemma \ref{fix-trace} and Remark \ref{rem-2} it follows that,
\[\left|z_{ \alpha^{im_n}_{1,n}\alpha_{i,n}\alpha^{2im_n}_{1,n} ,\pm}-\zeta_{i,n}\lambda^{2im_n}_{i,n}\right|\le \rho\frac{|\lambda^{im_n}_{i,n}|}{|\tr(\alpha_{i,n})|}
\le \rho\frac{e^i}{|\tr(\alpha_{i,n})|}\]
\[\left|z_{\alpha^{im_n}_{1,n}\alpha_{i,n}\alpha^{2im_n}_{1,n} ,\mp}-\eta_{i,n}\lambda^{4im_n}_{i,n}\right|\le \rho\frac{|\lambda^{2im_n}_{i,n}|}{|\tr(\alpha_{i,n})|}
\le \rho\frac{e^{2i}}{|\tr(\alpha_{i,n})|}\]
for sufficiently large $n.$
By Lemma \ref{fix-trace} and Remark \ref{rem-2} and the condition that $|z_{\alpha_{i,n},l}-z_{\alpha_{i,n},u}|\to 0$ we get
both $|\zeta_n|,|\eta_n|\to 1$. We have
\[||z_{\alpha^{im_n}_{1,n}\alpha_{i,n}\alpha^{2im_n}_{1,n} ,\pm}|-|\lambda^{2im_n}_{1,n}||+||z_{\alpha^{im_n}_{1,n}\alpha_{i,n}\alpha^{2im_n}_{1,n},\mp}|
-|\lambda_{1,n}^{4im_n}||\to 0.\]
By the definition of $m_n$ we must have $|\lambda^{2im_n}_n|\to e^{2i},|\lambda^{2im_n}_n|\to e^{4i}.$ Hence it follows that 
$|z_{\alpha^{im_n}_{1,n}\alpha_{i,n}\alpha^{2im_n}_{1,n} ,\pm}|\to e^2$ and also $|z_{\alpha^{im_n}_{1,n}\alpha_{i,n}\alpha^{2im_n}_{1,n} ,\mp}|\to 1.$ 
From this it follows that there exists $c>0$ such that,
\[Z_{<\alpha_{1,n},...,\alpha^{im_n}_{1,n}\alpha_{i,n}\alpha^{2im_n}_{1,n},...,\alpha^{km_n}_{1,n}\alpha_{k,n}\alpha^{2km_n}_{1,n} >}>c 
\quad\text{for sufficiently large}\quad n.\]
\section{Collapsing fixed points II}
We show in this collapsing situation we have: 
\begin{itemize}
\item
II$_1$ implies $Z_{\G_n}>c$ and,
\item
 II$_2$ implies classical for $n$.
 \end{itemize}
Where Collapsing fixed points II is further dichotomized into:
\begin{itemize}
\item[II$_1$:] $\liminf_n|\lambda_{1,n}|< \Lambda$ for some $\Lambda>1.$
\item[II$_2$:] $\liminf_n|\lambda_{1,n}|\to\infty$. 
\end{itemize}
We also assume that $|\zeta_{i,n}|\le|\eta_{i,n}|$ as before, otherwise we just simply replace by it's inverse. 
\subsection{II$_1$} 
In this case we prove that there are integers $l_{i,n}$ with,
\[ Z_{<\alpha_{1,n},\alpha^{l_{2,n}}_{1,n}\beta_{2,n},...,\alpha^{l_{k,n}}_{1,n}\beta_{k,n}>}>c\quad\text{for some}\quad  c>0.\]
By taking a subsequence we can just assume we have $|\lambda_{1,n}|\le\Lambda$ for large enough $n$. 
First we consider $\alpha_{2,n}$ and then move to $\alpha_{k,n}$ inductively as follows. \par
Choose positive integers $l_{2,n}$ to be the smallest such that $e^{2}\le|\zeta_{2,n}\lambda^{2l_{2,n}}_{1,n}|.$ Since $|\lambda_{1,n}|\le\Lambda,$
we must have some $\sigma_2>0$ such that $e^2<|\zeta_{2,n}\lambda_{1,n}^{2l_{2,n}}|<\sigma_2.$\par
\begin{lem}\label{claim}  
With above notations, there exists $0<\epsilon<e^2$ and $n>N_\epsilon$ such that 
$e^2-\epsilon<|z_{\alpha^{l_{2,n}}_{1,n}\alpha_{2,n},+}|,|z_{\alpha^{l_{2,n}}_{1,n}\alpha_{2,n},-}|<\sigma_2+\epsilon$.
\end{lem}
To show this Lemma \ref{claim} in the following we can use Remark \ref{4B}.2.B of Lemma \ref{fix-trace}. 
\begin{proof}
Note that since $|z_{\alpha_{2,n},-}-z_{\alpha_{2,n},+}|<1+\epsilon$ for some $\epsilon>0$ and large $n,$ and also $|\tr(\alpha_{2,n})|\to\infty,$ we have
$|c_{2,n}|\to\infty.$ By Remark \ref{4B}.2.B we have $|z_{\alpha_{2,n},u}-\eta_{2,n}|\to 0,$ hence for large $n$ we have $\Lambda^{-2}<|\eta_{2,n}|<1+\delta.$\par
Let us show that $|z_{ \alpha^{l_{2,n}}_{1,n}\alpha_{2,n} ,+}-z_{\alpha^{l_{2,n}}_{1,n}\alpha_{2,n},-}|\not\to\infty.$ Suppose this is not true.
Denote by $\rho_{2,n}$ the center of the circle having 
$z_{\alpha^{l_{2,n}}_{1,n}\alpha_{2,n},+}$ and $z_{\alpha^{l_{2,n}}_{1,n}\alpha_{2,n},-}$ as pair of antipodal points on the boundary sphere. By 
$|\zeta_{\alpha^{l_{2,n}}_{1,n}\alpha_{2,n}}+\eta_{\alpha^{l_{2,n}}_{1,n}\alpha_{2,n}}|<\sigma_2+1+\epsilon'$ for some $\epsilon'>0$ and large $n,$ 
and $\rho_{2,n}=\frac{z_{\alpha^{l_{2,n}}_{1,n}\alpha_{2,n},+}+z_{\alpha^{l_{2,n}}_{1,n}\alpha_{2,n},-}}{2},$ and 
$z_{\alpha^{l_{2,n}}_{1,n}\alpha_{2,n},+}+z_{\alpha^{l_{2,n}}_{1,n}\alpha_{2,n},-}=\zeta_{\alpha^{l_{2,n}}_{1,n}\alpha_{2,n}}+\eta_{\alpha^{l_{2,n}}_{1,n}\alpha_{2,n}}$
we get $|\rho_{2,n}-\sigma'_2|<\kappa$ for some positive number $\kappa>0$ and $\sigma'_2=\sigma+1+\epsilon'.$ Since
$|\rho_{2,n}-\sigma'_2|\le|\rho_{2,n}|+|\sigma'_2|<\frac{\sigma_2+1+\epsilon'}{2}+\sigma_2+1+\epsilon'$ we can take $\kappa_2=\frac{3}{2}(\sigma+1+\epsilon').$
Hence we have that 
$\dis(\mathcal{L}_{\alpha_{1,n}},\mathcal{L}_{\alpha^{l_{2,n}}_{1,n}\alpha_{2,n}})<\delta_2$ for some $\delta_2>0.$ By Remark \ref{4A}.1.A we have 
$|\tr(\alpha^{l_{2,n}}_{1,n}\alpha_{2,n})|\to\infty,$ 
and $|\zeta_{\alpha^{l_{2,n}}_{1,n}\alpha_{2,n}}-\eta_{\alpha^{l_{2,n}}_{1,n}\alpha_{2,n}}|=\frac{|\tr(\alpha^{l_{2,n}}_{1,n}\alpha_{2,n})|}
{|c_{2,n}\lambda^{-l_{2,n}}_{1,n}|},$ we get
$|\tr(\alpha^{l_{2,n}}_{1,n}\alpha_{2,n})|\asymp|c_{2,n}\lambda^{-l_{2,n}}_{1,n}|.$ But this imples that 
$|z_{\alpha^{l_{2,n}}_{1,n}\alpha_{2,n},+}-z_{\alpha^{l_{2,n}}_{1,n}\alpha_{2,n},-}|<C$ for some $C>0,$ hence a contradiction.\par
Note that if $|z_{\alpha^{l_{2,n}}_{1,n}\alpha_{2,n},+}-z_{\alpha^{l_{2,n}}_{1,n}\alpha_{2,n},-}|\to 0$ then
$|z_{\alpha^{l_{2,n}}_{1,n}\alpha_{2,n},\pm}|\to \frac{1}{2}|\zeta_{\alpha^{l_{2,n}}_{1,n}\alpha_{2,n}}+\eta_{\alpha^{l_{2,n}}_{1,n}\alpha_{2,n}}|.$
Since $\frac{1}{2}(e^2-1-\delta)<\frac{1}{2}|\zeta_{ \alpha^{l_{2,n}}_{1,n}\alpha_{2,n} }+\eta_{\alpha^{l_{2,n}}_{1,n}\alpha_{2,n}}|<\frac{1}{2}\sigma_2'$ for 
large $n.$  
This implies that $\frac{1}{2}(e^2-1-\delta_2)<|z_{\alpha^{l_{2,n}}_{1,n}\alpha_{2,n},\pm}|<\frac{1}{2}\sigma'_2$ for large $n.$\par
Now if $c<|z_{,+}-z_{\alpha^{l_{2,n}}_{1,n}\alpha_{2,n},-}|<c'$ for some $c,c'>0,$ then it follows from Remark \ref{4A}.1.A we get
$|\tr(\alpha^{l_{2,n}}_{1,n}\alpha_{2,n})|\to\infty$ and it imples that $|c_{2,n}\lambda_{1,n}^{-l_{2,n}}|\to\infty.$ So by Remark \ref{4B}.2.B we get
$\{z_{\alpha^{l_{2,n}}_{1,n}\alpha_{2,n},+},z_{ \alpha^{l_{2,n}}_{1,n}\alpha_{2,n} ,-}\}\to\{\zeta_{\alpha^{l_{2,n}}_{1,n}\alpha_{2,n}},\eta_{\alpha^{l_{2,n}}_{1,n}\alpha_{2,n}}\}$ 
which gives the lemma.
\end{proof}
Next we proceed down to generator $\alpha_{i+1,n}$ and do the same above procedure as we have done for $\alpha_{i,n}$ with $i\ge 2.$ More precisely, we choose
$l_{i+1,n}$ to be the smallest integer such that $e^{2(i+1)}+|z_{\alpha^{l_{i,n}}_{1,n}\alpha_{2,n},u}|\le|\zeta_{i+1,n}\lambda^{2l_{i+1,n}}_{1,n}|.$
Then it follows from above proof we have,
\[e^{2(i+1)}+|z_{\alpha^{l_{i,n}}_{1,n}\alpha_{2,n},u}|-\epsilon<|z_{\alpha^{l_{i+1,n}}_{1,n}\alpha_{i+1,n},+}|,|z_{\alpha^{l_{i+1,n}}_{1,n}}|\alpha_{i+1,n}|<\sigma_{i+1}+\epsilon, i\ge 2.\]
Now for these transformed generators we have the following two possibilities according to the collapsing or the non-collapsing fixed points
 respectively as,
\begin{itemize}
\item 
II$_{1a}:$ $\liminf\left|z_{\alpha^{l_{2,n}}_{1,n}\alpha_{2,n},+}-z_{\alpha^{l_{2,n}}_{1,n}\alpha_{2,n},-}\right|\to 0.$ 
\item
II$_{1b}:$ $\liminf_n\left|z_{\alpha^{l_{2,n}}_{1,n}\alpha_{2,n},+}-z_{\alpha^{l_{2,n}}_{1,n}\alpha_{2,n},-}\right|>0$. 
\end{itemize}
In the case II$_{1b}$, we have $Z_{<\alpha_{1,n},\alpha^{l_{2,n}}_{1,n}\beta_{2,n},...,\alpha^{l_{k,n}}_{1,n}\alpha_{k,n}>}>c$, for some $c>0$ and large $n$.\par
Hence suppose we have II$_{1a}.$ Passing to a subsequence if necessary we can assume that
$|z_{\alpha^{l_{i+1,n}}_{1,n}\alpha_{i+1,n} ,+}-z_{\alpha^{l_{i+1,n}}_{1,n}\alpha_{i+1,n} ,-}|\to 0.$\par
If $|\lambda_{1,n}|\to 1$ then choose positive integers $m_{i,n},m'_{i,n}$ as defined in Collapsing fixed points I$_1,$ such that 
$e^{2(i)}+|z_{\alpha^{l_{k,n}}_{1,n}\alpha_{k,n},u}| \le|\eta_{i,n}\lambda^{2l_{i,n}+2m_{i,n}}_{1,n}|<e^{M_{i,n}},$ 
$e^{M_{i,n}+2}\le|\zeta_{i,n}\lambda^{2l_{i,n}+2m_{i,n}}_{1,n}|<e^{M'_{i,n}},$
and also 
\[\kappa_{i,n}<\left|z_{ \alpha^{l_{i,n}+m_{i,n}}_{1,n}\alpha_{i,n} ,+}-z_{\alpha^{l_{i,n}+m_{i,n}}_{1,n}\alpha_{i,n} ,-}\right|<\kappa'_{i,n},\]
for some $0<\kappa_{i,n},\kappa'_{i,n}, M_{i,n},M'_{i,n}.$ By Remark \ref{4A}.1.A, $|\tr(\alpha^{l_{i,n}+m_{i,n}}_{1,n}\alpha_{i+1,n})|\to\infty.$ Hence by
Remark \ref{4B}.2.B we have 
\[\{z_{\alpha^{l_{i,n}+m_{i,n}}_{1,n}\alpha_{i,n},+},z_{\alpha^{l_{i,n}+m_{i,n}}_{1,n}\alpha_{i,n},-}\}
\to\{\zeta_{\alpha^{l_{i,n}+m_{i,n}}_{1,n}\alpha_{i,n} },\eta_{\alpha^{\alpha^{l_{i,n}+m_{i,n}}_{1,n}\alpha_{i,n}}}\}.\]
This implies that 
$e^{2i}+|z_{\alpha^{l_{k,n}}_{1,n}\alpha_{k,n},u}|\le|z_{\alpha^{l_{i,n}+m_{i,n}}_{1,n}\alpha_{i,n} ,\pm}|<e^{M_{i,n}}$ 
and $e^{M_{i,n}+2}<|z_{\alpha^{l_{i,n}+m_{i,n}}_{1,n}\alpha_{i,n} ,\mp}|<e^{M'_{i,n}}.$
Hence there exists $c>0$, such that
\[Z_{<\alpha_{1,n},...,\alpha^{l_{i,n}+m_{i,n}}_{1,n}\alpha_{i,n},...,\alpha^{l_{k,n}}_{1,n}\alpha_{k,n}>}>c\] 
for large $n.$\par
If $\lambda_{1,n}|>c>0$ then we do the same as above and treat it as special case by chose appropriate $\tilde{m}_{i,n}.$ Then 
$\delta_1<|z_{\alpha^{l_{i,n}+\tilde{m}_{i,n}}_{1,n}\alpha_{i,n},+}-z_{\alpha^{l_{i,n}+\tilde{m}_{i,n}}_{1,n}\alpha_{i,n},-}|<\delta_2$  
for some $0<\delta_1,\delta_2.$ And it follows from Remark \ref{4A}.1.A and Remark \ref{4B}.2.B we have that,
\[Z_{<\alpha_{1,n},...,\alpha^{l_{i,n}+\tilde{m}_{i,n}}_{1,n}\alpha_{i,n},...,\alpha^{l_{k,n}}_{1,n}\alpha_{k,n}>}>c\quad\text{for large}\quad n.\]
\subsection{II$_2$} 
Take a subsequence of $\alpha_{1,n}$ if necessary we can assume that $|\lambda_{1,n}|$ is strictly increasing to.
Let $l_{i,n}\ge 0$ be a sequence of integers such that $|\zeta_{i,n}\lambda^{2l_{i,n}}_{1,n}|\le 1$.
If $\limsup |\zeta_{\alpha_{i,n}}\lambda^{2l_{i,n}}_{1,n}|=1$ and $\limsup|\zeta_{\alpha_{i,n}}\lambda^{2l_{i,n}}_{1,n}-\eta_{\alpha_{i,n}}|\not=0$ then 
we have a subsequence
such that, 
\[\liminf_j Z_{<\alpha_{1,n_j},..., \alpha^{l_{i,n_j}+2(i-1)}_{n_j} \alpha_{i,n_j}\alpha^{-2(i-1)}_{n_j}  ,...,
\alpha^{l_{k,n_j}+2(k-1)}_{n_j}\alpha_{k,n_j}\alpha^{-2(k-1)}_{n_j} >}>0.\]\par
Suppose $\limsup \zeta_{\alpha_{i,n}}\lambda^{2l_{i,n}}_{1,n}=1$ then, denote by  $<\alpha_{1,n_j},...,\alpha^{l_{k,n_j}}_{1,n_j}\alpha_{k,n_j}>$ the subsequence 
of $<\alpha_{1,n},...,\alpha^{l_{k,n}}_{1,n}\alpha_{k,n}>$ with $\lim_j\zeta_{\alpha_{i,n_j}}\lambda^{2l_{i,n_j}}_{1,n_j}=1$. \par
If $\limsup|\tr(\alpha^{l_{1,n_j}}_{i,n_j}\alpha_{i,n_j})|=\infty$ then by passing to a subsequence if necessary, for large
$j$, $\alpha^{l_{i,n_j}}_{1,n_j}\alpha_{i,n_j}$ will have disjointed isometric circles. 
Since $|\tr(\alpha_{i,n_j})|\to\infty$, $z_{\alpha_{i,n_j},l}\to 0,$ so $|z_{\alpha_{i,n_j},+}-z_{\alpha_{i,n_j},-}|<c$ for $c>0,$ we have 
$|c_{i,n_j}|\to\infty.$
By Remark \ref{4B}.2.B, $\lim_j\min\{|\eta_{\alpha_{i,n_j}}-z_{\alpha_{i,n_j},-}|,|\eta_{\alpha_{i,n_j}}-z_{\alpha_{i,n_j},+}|\}\to 0$. 
Since $\zeta_{\alpha^{l_{i,n_j}}_{1,n_j}\alpha_{i,n_j}}\to 1$ and $|\tr(\alpha^{l_{1,n_j}}_{1,n_j}\alpha_{i,n_j})|\to\infty$ we have,
\[\left|\zeta_{\alpha^{l_{i,n_j}}_{1,n_j}\alpha_{i,n_j}}-\eta_{\alpha^{l_{i,n_j}}_{1,n_j}\alpha_{1,n_j}}\right|=\frac{\left|\tr(\alpha^{l_{i,n_j}}_{1,n_j}\alpha_{1,n_j})\right|}
{|\lambda^{-l_{i,n_j}}_{1,n_j}c_{i,n_j}|}\to 0\]
and, $|\lambda^{-l_{i,n_j}}_{1,n_j}c_{i,n_j}|\to\infty.$ Also 
\[\left|z_{\alpha^{l_{i,n_j}}_{1,n_j}\alpha_{i,n_j},+}-z_{\alpha^{l_{i,n_j}}_{1,n_j}\alpha_{i,n_j},-}\right|=\frac{\left|\sqrt{\tr^2(\alpha^{l_{i,n_j}}_{1,n_j}\alpha_{i,n_j})-4}\right|}
{2|\lambda^{-l_{i,n_j}}_{1,n_j}c_{i,n_j}|}\to 0.\]
Hence it follows from Remark \ref{4B}.2.B, there exists a $\kappa>0$ such that for large $j$ we have,
 $\kappa^{-1}<|z_{\alpha^{l_{i,n_j}}_{1,n_j}\alpha_{i,n_j},l}|\le|z_{\alpha^{l_{i,n_j}}_{1,n_j}\alpha_{i,n_j},u}|<\kappa$.
And since $|\lambda_{1,n_j}|\to\infty$,
we can choose Mobius transformations $\psi_i$ such that 
\[S_j=\psi_j<\alpha_{1,n_j},...,\alpha^{l_{i,n_j}}_{1,n_j}\alpha_{i,n_j},...,\alpha^{l_{i,n_j}}_{1,n_j}\alpha_{i,n_j}>\psi^{-1}_j\] 
satisfies Lemma \ref{classical}.
And since it is $k-1$ classical of standard form, we have $S_j$ generates classical Schottky groups for large $j.$  \par
Let us arrange $\alpha_{i,n},i\ge 2$ such that $|\tr(\alpha_{i+1,n})|\le|\tr(\alpha_{i,n})|.$
If we have that $\limsup|\tr(\alpha^{l_{k,n_j}}_{1,n_j}\alpha_{k,n_j})|<\infty$ then set $\phi_j$ be Mobius transformations
so that $\phi_j\alpha^{l_{k,n_j}}_{1,n_j}\alpha_{k,n_j}\phi^{-1}_{j}$ have fixed points $0,\infty$ and
fixed points $z_{\phi_{j}\alpha_{1,n_j}\phi^{-1}_j,\pm}$ be $z_{\alpha^{l_{k,n_j}}\alpha_{k,n_j},\pm}.$ Then it follows that 
$z_{\phi_j\alpha_{i,n_j}\phi^{-1}_j,l}\to 1$. Hence we can reduced this case to Collapsing fixed points I which has already been considered.\par
If $\limsup|\zeta_{\alpha_{i,n}}\lambda^{2l_{i,n}}_{1,n}|<1$, then we have:
\begin{itemize}
\item
II$_{2a}:$ $\liminf|\zeta_{\alpha_{i,n}}\lambda^{2l_{i,n}+2}_{1,n}|=1.$ 
\item
II$_{2b}:$ $\liminf|\zeta_{\alpha_{i,n}}\lambda^{2l_{i,n}+2}_{1,n}|>1.$
\end{itemize}

\subsubsection{II$_{2a}$} 
 We arrange so that $<\alpha_{1,n_j},...,\alpha^{l_{k,n_j}}_{1,n_j}\alpha_{k,n_j}>$ be a subsequence of
such that $\lim_j|\zeta_{\alpha_{k,n_j}}\lambda^{2l_{k,n_j}+2}_{1,n_j}|\to 1$. 
If $\sup_j|\tr(\alpha^{l_{k,n_j}+1}_{1,n_j}\alpha_{k,n_j})|<\infty$, then we conjugate 
$\alpha_{1,n_j},\alpha^{l_{k,n_j}}_{1,n_j}\alpha_{k,n_j}$ to
$\hat{\alpha}_{k,j}=\phi_j\alpha_{1,n_j}\phi^{-1}_j,$ and $\hat{\alpha}_{1,j}=\phi_j\alpha^{l_{k,n_j}+1}_{1,n_j}\alpha_{1,n_j}\phi^{-1}_j$ with $\hat{\alpha}_{1,i}$ 
have fixed points $0,\infty$ and $\hat{\alpha}_{k,j}$ have $z_{\alpha_{1,n_j}^{l_{k,n_j}+1}\alpha_{k,n_j},\pm}.$ 
Since $\sup_j|\hat{\lambda}_{1,j}|<\infty$, it follows that, if $z_{\hat{\alpha}_{i,j},l}\to 1$ then 
$<\hat{\alpha}_{1,j},...,\hat{\alpha}_{k,j}>$, falls under Collapsing fixed points I, and if
$z_{\hat{\beta}_i,l}\to 0$ then $<\hat{\alpha}_{1,j},...,\hat{\alpha}_{k,j}>$ falls under Collapsing fixed points II$_1$ Otherwise there exists $\epsilon>0$ such that
$\epsilon<|z_{\hat{\alpha}_{i,j},l}|<1-\epsilon,$ hence as in Collapsing fixed points I we can choose integers $N_{i,n}$ such that
$Z_{<\hat{\alpha}_{1,j},...,\hat{\alpha}^{N_{i,j}}_{1,j}\hat{\alpha}_{i,j}\hat{\alpha}_{1,j}^{-N_{i,n}},...,
\hat{\alpha}^{N_{k,j}}_{1,j}\hat{\alpha}_{k,j}\hat{\alpha}_{1,j}^{-N_{k,n}}>}>c$ for some $c>0.$\par
Now if we assume that $\sup_j|\tr(\alpha^{l_{i,n_j}+1}_{1,n_j}\alpha_{i,n_j})|=\infty$ then for large $j,$
since the radius of isometric circles are
$\mathfrak{R}_{\alpha^{l_{i,n_j}+1}_{1,n_j}\alpha_{i,n_j}}=\frac{\left|z_{\alpha^{l_{i,n_j}+1}_{1,n_j}\alpha_{i,n_j},u}-z_{\alpha^{l_{i,n_j}+1}_{i,n_j}\alpha_{i,n_j},l}\right|}
{|\tr(\alpha^{l_{i,n_j}+1}_{1,n_j}\alpha_{i,n_j})|}$ and the distance between the centers is 
$|\zeta_{\alpha^{l_{i,n_j}+1}_{1,n_j}\alpha_{1,n_j}}-\eta_{\alpha^{l_{i,n_j}+1}_{1,n_j}\alpha_{i,n_j}}|=$
$\frac{|\tr(\alpha^{l_{i,n_j}+1}_{1,n_j}\alpha_{i,n_j})|}{|c_{i,n_j}\lambda_{1,n_j}^{-l_{i,n_j}-1}|},$ and by
$|z_{\alpha^{l_{i,n_j}+1}_{1,n_j}\alpha_{i,n_j},u}-z_{\alpha^{l_{i,n_j}+1}_{1,n_j}\alpha_{i,n_j},l}|=$
$\frac{|\sqrt{\tr^2(\alpha^{l_{i,n_j}+1}_{1,n_j}\alpha_{i,n_j})-4}|}{|2c_{i,n_j}\lambda_{1,n_j}^{-l_{i,n_j}-1}|}$
we have,
\[\lim_j\frac{|\zeta_{\alpha^{l_{i,n_j}+1}_{1,n_j}\alpha_{i,n_j}}-\eta_{\alpha^{l_{i,n_j}+1}_{1,n_j}\alpha_{i,n_j}}|}{\mathfrak{R}_{\alpha^{l_{i,n_j}+1}_{1,n_j}\alpha_{i,n_j}}}
=\lim_i\frac{2|\tr(\alpha^{l_{i,n_j}+1}_{1,n_j}\alpha_{i,n_j})|^2}{|\sqrt{\tr^2(\alpha^{l_{i,n_j}+1}_{1,n_j}\alpha_{i,n_j})-4}|}>\delta
|\tr(\alpha^{l_{i,n_j}+1}_{1,n_j}\alpha_{i,n_j})|\] for some $\delta>0.$ Hence
$|\zeta_{\alpha^{l_{i,n_j}+1}_{1,n_j}\alpha_{i,n_j}}-\eta_{\alpha^{l_{i,n_j}+1}_{1,n_j}\alpha_{i,n_j}}|>2\mathfrak{R}_{\alpha^{l_{i,n_j}+1}_{1,n_j}\alpha_{i,n_j}}$
for large $j.$
From this it implies that $\alpha^{l_{i,n_j}+1}_{1,n_j}\alpha_{i,n_j}$ have disjointed isometric circles for large $j.$ 
By Lemma \ref{fix-trace} we have  $\eta_{\alpha_{i,n_j}}\to 1.$ And since $\eta_{\alpha^{l_{i,n_j}+1}_{1,n_i}\alpha_{i,n_j}}=
\eta_{\alpha_{i,n_j}},$ implies that $\eta_{\alpha^{l_{i,n_j}+1}_{1,n_j}\alpha_{i,n_j}}\to 1.$ 
Note that if $\inf_j|\zeta_{\alpha^{l_{i,n_j}+1}_{1,n_j}\alpha_{i,n_j}}-1|>0$ and $|\zeta_{\alpha^{l_{i,n_j}+1}_{1,n_j}\alpha_{i,n_j}}|\to 1,$ then
by $|\tr(\alpha^{l_{i,n_j}+1}_{1,n_j}\alpha_{i,n_j})|\to\infty$ we have 
$|1-z_{\alpha^{l_{i,n_j}+1}_{1,n_j}\alpha_{i,n_j},l}|,$ 
$|\zeta_{\alpha^{l_{i,n_j}+1}_{1,n_j}\alpha_{i,n_j}}-z_{\alpha^{l_{i,n_j}+1}_{1,n_j}\alpha_{i,n_j},u}|\to 0,$ and 
$\inf_j|z_{\alpha^{l_{i,n_j}+1}_{1,n_j}\alpha_{i,n_j},u}-z_{\alpha^{l_{i,n_j}+1}_{1,n_j}\alpha_{i,n_j},l}|>0.$ Hence for large $j$ there exits
$\epsilon>0$ such that 
\[1-\epsilon<\left|z_{\alpha^{l_{i,n_j}+1}_{1,n_j}\alpha_{i,n_j},l}\right|\le\left|z_{\alpha^{l_{i,n_j}+1}_{1,n_j}\alpha_{i,n_j},u}\right|<1+\epsilon.\]
Hence we can choose integers $N_{i,n}$ such that,
\[\inf_jZ_{<\alpha_{1,n_j},..., \alpha^{l_{i,n_i}+N_{i,n_j}+1}_{1,n_j} \alpha_{i,n_j}\alpha^{-N_{i,n_j}}_{1,n_j} ,..., 
\alpha^{l_{k,n_i}+N_{k,n_j}+1}_{1,n_j} \alpha_{k,n_j}\alpha^{-N_{k,n_j}}_{1,n_j}  >}>0.\]
Now suppose that $\zeta_{\alpha^{l_{i,n_j}+1}_{1,n_j}\alpha_{i,n_j}}\to 1.$ Then
we have 
\[|\zeta_{\alpha^{l_{i,n_j}+1}_{1,n_j}\alpha_{i,n_j}}-\eta_{\alpha^{l_{i,n_j}+1}_{1,n_j}\alpha_{i,n_j}}|\to 0.\] 
Hence 
$\frac{|\tr(\alpha^{l_{i,n_j}+1}_{1,n_j}\alpha_{i,n_j})|}{|c_{i,n_j}\lambda_{1,n_j}^{-l_{i,n_j}-1}|}\to 0.$ Since 
$|\tr(\alpha^{l_{i,n_j}+1}_{1,n_j}\alpha_{i,n_j})|\to\infty,$ which gives
\[|z_{\alpha^{k_{n_i}+1}_{n_i}\beta_{n_i},u}-z_{\alpha^{k_{n_i}+1}_{n_i}\beta_{n_i},l}|=
\frac{\sqrt{\tr^2(\alpha^{l_{n_i}+1}_{n_i}\beta_{n_i})-4}}{|2c_{n_i}\lambda_{n_i}^{-l_{n_i}-1}|}\to 0.\]
Therefore the distance between the centers of these isometric circles decreases to $0$ and radius $\mathfrak{R}_{\alpha^{l_{i,n_j}+1}_{1,n_j}\alpha_{i,n_j}}\to 0.$
Since $\inf_j|\lambda_{1,n_j}|^2>c>1,$ we have for large $j$ that the isometric circles of $\alpha^{l_{i,n_j}+1}_{1,n_j}\alpha_{i,n_j}$ disjointed and
lies between $c^{-1}$ and $c.$ In particular, $c^{-1}<|z_{ \alpha^{l_{i,n_j}+1}_{1,n_j}\alpha_{i,n_j} ,l}|\le|z_{\alpha^{l_{i,n_j}+1}_{1,n_j}\alpha_{i,n_j},u}|
<c.$ Set, 
\[S_j=<\alpha_{1,n_j},...,\alpha^{l_{i,n_j}+1}_{1,n_j}\alpha_{i,n_j},..., \alpha^{l_{k,n_j}+1}_{1,n_j}\alpha_{k,n_j}>.\]
$S_j$ satisfies Lemma \ref{classical}, hence classical.\par

\subsubsection{II$_{2b}$} 

 First we define a new sequence of $<\tilde{\alpha}_{1,n},...,\tilde{\alpha}_{k,n}>$ as follows:
Consider $S'_n=<\alpha_{1,n},..., \alpha^{l_{i,n}}_{1,n}\alpha_{i,n} ,...,\alpha^{l_{k,n}}_{1,n}\alpha_{k,n}>$. 
We arrange $< \alpha^{l_{2,n}}_{1,n}\alpha_{2,n} ,..., \alpha^{l_{k,n}}_{1,n}\alpha_{k,n} >$ so that 
$|\tr(\alpha^{l_{i+1,n}}_{1,n}\alpha_{i+1,n})|\le|\tr(\alpha^{l_{i,n}}_{1,n}\alpha_{i,n}).$ 
If $|\tr(\alpha^{l_{k,n}}_{1,n}\alpha_{k,n})|\ge|\tr(\alpha_{1,n})|$,
then set $\tilde{\alpha}_{i,n}=\alpha_{i,n}, i\ge 1$. Otherwise, let $\phi_n$ be the Mobius map 
so that $\phi_n\alpha^{l_{k,n}}_{1,n}\alpha_{k,n}\phi^{-1}_n$ 
have fixed points $0,\infty$, and $\phi_n\alpha_{1,n}\phi^{-1}_n$ have $z_{\phi_n\alpha_{1,n}\phi^{-1}_n,\pm}=z_{\alpha_{k,n},\pm}.$
Set $\alpha_{1,n;1}=\phi_n\alpha^{l_{k,n}}_{1,n}\alpha_{k,n}\phi^{-1}_n, \alpha_{i,n;1}=\phi_n\alpha_{i,n}\phi^{-1}_n, i\not=k.$
We define integer $l_{i,n;1}$ with respect to $<\alpha_{1,n;1},...,\alpha_{k,n;1}>$ the same way as we defined $l_{i,n}$ before. \par
Suppose that $|\tr(\alpha^{l_{k,n;1}}_{1,n;1}\alpha_{k,n;1})|\ge|\tr(\alpha_{1,n;1})|$ then 
we set $\tilde{\alpha}_{i,n}=\alpha_{i,n;1},i\ge 2$ and $\tilde{\alpha}_{1,n}=\alpha_{1,n;1}$. Otherwise, we repeat this construction
to get a sequence $<\alpha_{1,n;m},...,\alpha_{k,n;m}>.$
By construction for a each $n$, either there exists a $m$ such that $|\tr(\alpha^{l_{k,n;m}}_{1,n;m}\alpha_{k,n;m})|\ge|\tr(\alpha_{1,n;m})|$ 
or we have $|\tr(\alpha^{l_{k,n;m}}_{1,n;m}\alpha_{k,n;m})|\ge|\tr(\alpha_{1,n;m})|$ for all $m.$ Assume the latter, since
$\alpha_{1,n,m+1}=\phi_{n,m}\alpha^{l_{k,n;m}}_{1,n;m}\alpha_{k,n;m}\phi^{-1}_{n,m}$ we have 
$|\tr(\alpha_{1,n;m+1})|<|\tr(\alpha_{1,n;m})|$ for all $m.$ If $\lim_m|\tr(\alpha_{1,n;m})|=0$ then take $m_n$ to be the first integer $m$ with
with $|\tr(\alpha_{1,n;m})|<\frac{1}{n}.$ If $\lim_m|\tr(\alpha_{1,n;m})| >0$ then take $m_n$ to be the first integer $m$ with 
$|\tr(\alpha_{1,n;m+1})|>|\tr(\alpha_{1,n;m})|-\frac{1}{n}.$ If the former holds, we set $m_n$ to be the first integer $m$ with
$|\tr(\alpha^{l_{k,n;m}}_{1,n;m}\alpha_{k,n;m})|\ge|\tr(\alpha_{1,n;m})|.$ Hence there exists a $m_n$, such that either
$|\tr(\alpha^{l_{k,n;m_n}}_{1,n;m_n}\alpha_{k,n;m_n})|>|\tr(\alpha_{1,n;m_n})|-\frac{1}{n}$, or 
$|\tr(\alpha_{1,n;m_n})|<\frac{1}{n}$. Note that by construction arrangement, we have 
$|\tr(\alpha^{l_{i,n;m}}_{1,n;m}\alpha_{i,n;m})|\ge|\tr(\alpha^{l_{i+1,n;m}}_{1,n;m}\alpha_{i+1,n;m})|, i\ge 2.$ This also implies that
$|\tr(\alpha^{l_{i,n;m_n}}_{1,n;m_n}\alpha_{i,n;m_n})|>|\tr(\alpha_{1,n;m_n})|-\frac{1}{n}, i\ge 2$ whenever we have
$|\tr(\alpha^{l_{k,n;m_n}}_{1,n;m_n}\alpha_{k,n;m_n})|>|\tr(\alpha_{1,n;m_n})|-\frac{1}{n}.$ 
We define $\tilde{\alpha}_{1,n}=\alpha_{1,n;m_n},\tilde{\alpha}_{i,n}=\alpha_{i,n,m_n}, i\ge 2.$ 
For $S'_n=<\alpha_{1,n;m},...,\alpha^{l_{k,n;m}}_{1,n;m}\alpha_{k,n;m}>,$ note that $S'_n$ is also of $k-1$ standard form.
\par

First let us consider $<\tilde{\alpha}_{1,n},...,\tilde{\alpha}_{k,n}>$. If $\liminf_n|\tr(\tilde{\alpha}_{1,n})|<\infty$, we choose a subsequence
with $|\tr(\tilde{\alpha}_{1,n_i})|<c$ for all large $i$ and some $c>0$. Let $p_{i,j},2\le i\le k$ be a sequence of least positive integers 
such that $|\tr(\tilde{\alpha}^{p_{i,j}}_{1,n_j}\tilde{\alpha}_{i,n_j})|>\frac{1}{\mathfrak{D}_{n_j}}$. We conjugate $\tilde{\alpha}^{p_{i,j}}_{1,n_j}\tilde{\alpha}_{i,n_j}$
by $\psi_j$ that fixes $0,\infty$ and $\max_{2\le i\le k}\{z_{\psi_j\tilde{\alpha}^{p_{i,j}}_{1,n_j}\tilde{\alpha}_{i,n_j}\psi^{-1}_j,u}\}=1.$
Set $\bar{\alpha}_{1,j}=\psi_j\tilde{\alpha}_{1,n_j}\psi^{-1}_j, \bar{\alpha}_{i,j}=\psi_j\tilde{\alpha}^{p_{i,j}}_{1,n_j}\tilde{\alpha}_{i,n_j}\psi^{-1}_j.$
By construction, if $z_{\bar{\alpha}_{i,j},l}\to 0$ then $<\bar{\alpha}_{1,j},...,\bar{\alpha}_{k,j}>$  satisfies Collapsing fixed points II$_1$ and if
$z_{\bar{\alpha}_{i,j},l}\to 1$ then $<\bar{\alpha}_{1,j},...,\bar{\alpha}_{k,j}>$  satisfies Collapsing fixed points II. Otherwise there exists $\epsilon>0$
such that $\epsilon<|z_{\bar{\alpha}_{i,j},l}|<1-\epsilon$ then we can chose $N_i>0$ integers such that 
$Z_{<\bar{\alpha}_{1,j},...,\bar{\alpha}^{N_i}_{1,j}\bar{\alpha}_{i,j}\bar{\alpha}^{-N_i}_{1,j},...,\bar{\alpha}^{N_k}_{1,j}\bar{\alpha}_{k,j}\bar{\alpha}^{-N_k}_{1,j} >}>c$ for some $c>0.$
Hence in either case, we are done. \par
Now suppose that $\liminf_n|\tr(\tilde{\alpha}_{1,n})|=\infty$. Since $|\tr(\tilde{\alpha}_{i,n})|\ge|\tr(\tilde{\alpha}_{1,n})|,i\ge 2$
it is sufficient to assume that $<\tilde{\alpha}_{1,n},...,\tilde{\alpha}_{k,n}>$ satisfies case II$_{2b}$, otherwise
we are done. We define $\tilde{l}_{i,n}=l_{i,n;m_n}, i\ge 2.$ \par
Set $\beta_{1,n}=\tilde{\alpha}_{1,n}, \beta_{i,n}= \tilde{\alpha}^{\tilde{l}_{i,n}}_{1,n}\tilde{\alpha}_{i,n} .$
Since $|\tr(\tilde{\alpha}_{i,n})|\ge|\tr(\tilde{\alpha}_{1,n})|$ and 
$|z_{\tilde{\alpha}_{i,n},-}-z_{\tilde{\alpha}_{i,n},+}|\le 1+\delta_n$ with $\delta_n\to 0$ (this follows from that 
$|\lambda_{1,n}|^{-2}\le|z_{\tilde{\alpha}_{i,n},u}|\le 1, z_{\tilde{\alpha}_{1,n},l}\to 0$), 
it follows from Lemma \ref{fix-trace} with Remark \ref{rem-2} and the inequality 
$|\tr(\alpha^{l_{i,n;m_n}}_{1,n;m_n}\alpha_{i,n;m_n})|>|\tr(\alpha_{1,n;m_n})|-\frac{1}{n},$
which implies $\lim_n\frac{|\tr(\beta_{i,n})|}{|\tilde{\lambda}_{1,n}|}>\epsilon$ for $\epsilon>0$ where
$\tilde{\lambda}_{1,n}$ is the multiplier of $\beta_{1,n},$ and 
$|\eta_{\beta_{i,n}}-\zeta_{\beta_{i,n}}|<\delta|\tilde{\lambda}_{1,n}|^{-2},$ for some $\delta>0.$ 
Since $\eta_{\beta_{i,n}\beta_{1,n}}=\eta_{\beta_{i,n}}\tilde{\lambda}_{1,n}^{-2}$ and $\eta_{\beta_{1,n}\beta_{i,n}}=\eta_{\beta_{i,n}},$ we have
$|\eta_{\beta_{i,n}\beta_{1,n}}-\tilde{\lambda}_{1,n}^{-2}|=|\eta_{\beta_{i,n}}\tilde{\lambda}_{1,n}^{-2}-\tilde{\lambda}_{1,n}^{-2}|<\delta|\tilde{\lambda}_{1,n}|^{-4}$ and
$|\eta_{\beta_{1,n}\beta_{i,n}}-\zeta_{\beta_{i,n}}|=|\eta_{\beta_{i,n}}-\zeta_{\beta_{i,n}}|<\delta|\tilde{\lambda}_{1,n}|^{-2},$ for large $n.$\par
Now we have reduced II$_{2b}$ to the collapsing behaviors of $\eta_{\beta_{i,n}},\zeta_{\beta_{i,n}}$ there are four cases that needs to be considered as follows:
\begin{itemize}
\item[(I)]
\[ \lim_j\left\{\frac{|\eta_{\tilde{\alpha}_{i,n_j}}|-|\tilde{\lambda}^{-2}_{1,n_j}|}{|\tilde
{\lambda}_{1,n_j}|^{-2}},
\frac{|\zeta_{\tilde{\alpha}_{i,n_j}}\tilde{\lambda}^{2\tilde{l}_{i,n_j}}_{1,n_j}|-|\tilde{\lambda}^{-2}_{1,n_j}|}{|\tilde
{\lambda}_{1,n_j}|^{-2}}\right\}<C_1.\]
\item[(II)]
\[ \lim_j\frac{|\eta_{\tilde{\alpha}_{i,n_j}}|-|\tilde{\lambda}^{-2}_{1,n_j}|}{|\tilde
{\lambda}_{1,n_j}|^{-2}}<C_2,\quad\lim_j\frac{|\zeta_{\tilde{\alpha}_{i,n_j}}\tilde{\lambda}^{2\tilde{l}_{i,n_j}}_{1,n_j}|-|\tilde{\lambda}^{-2}_{1,n_j}|}{|\tilde
{\lambda}_{1,n_j}|^{-2}}=\infty.\]  
\item[(III)]
\[ \lim_j\frac{|\eta_{\tilde{\alpha}_{i,n_j}}|-|\tilde{\lambda}^{-2}_{1,n_j}|}{|\tilde
{\lambda}_{1,n_j}|^{-2}}=\infty,\quad \lim_j\frac{|\zeta_{\tilde{\alpha}_{i,n_j}}\tilde{\lambda}^{2\tilde{l}_{i,n_j}}_{1,n_j}|-|\tilde{\lambda}^{-2}_{1,n_j}|}{|\tilde
{\lambda}_{1,n_j}|^{-2}}<C_3.\]
\item[(IV)]
\[ \lim_n\left\{\frac{|\eta_{\tilde{\alpha}_{i,n}}|-|\tilde{\lambda}^{-2}_{1,n}|}{|\tilde
{\lambda}_{1,n}|^{-2}},
\frac{|\zeta_{\tilde{\alpha}_{i,n}}\tilde{\lambda}^{2\tilde{l}_{i,n}}_{1,n}|-|\tilde{\lambda}^{-2}_{1,n}|}{|\tilde
{\lambda}_{1,n}|^{-2}}\right\}=\infty.\]
\end{itemize}
{\bf Consider (I)}:\\
In this case, we will produce a set of classical generators for sufficiently large $n_j$ by using condition (I), which is the transformed set of generators from the original set of generators.\par
Since $|\zeta_{\tilde{\alpha}_{i,n_j}}\tilde{\lambda}^{2\tilde{l}_{i,n_j}}_{1,n_j}|>|\tilde{\lambda}_{1,n_j}|^{-2}$ and our assumption 
and $\zeta_{\beta_{1,n_j}\beta_{i,n_j}}=\tilde{\lambda}_{1,n_j}^2\zeta_{\beta_{i,n_j}},$ we have for large $j$,
\[ 1<|\zeta_{\beta_{1,n_j}\beta_{i,n_j}}|<(C_1+1), 2\le i\le k.\]
Since $\eta_{\beta_{i,n_j}}=\eta_{\tilde{\alpha}^{\tilde{l}_{i,n_j}}_{1,n_j}\tilde{\alpha}_{i,n_j}}=\eta_{\tilde{\alpha}_{i,n_j}}$ and
$|\tilde{\lambda}_{1,n_j}|^{-2}<|\eta_{\tilde{\alpha}_{i,n_j}}|\le 1$ implies that $1<|\eta_{\beta_{i,n_j}}\tilde{\lambda}_{1,n_j}^2|.$
By our assumption we have Similarly, $1<|\eta_{\beta_{i,n_j}\beta_{1,n_j}^{-1}}|<(C_1+1).$ 
Since $|\tilde{\lambda}_{1,n_j}|\to\infty$ it follows that
there exists $\kappa>1$ such that, $1-\kappa^{-1}<|\eta_{ \beta_{1,n_j}\beta_{i,n_j}\beta^{-1}_{1,n_j} } |,|\zeta_{\beta_{1,n_j}\beta_{i,n_j}\beta_{1,n_j}^{-1}}|<1+\kappa$ for large $i.$ By Lemma \ref{fix-trace} 
with Remark \ref{rem-2} and
$|\tr(\beta_{i,n_j})|>\epsilon|\tilde{\lambda}_{1,n_j}|$ for large $j$ we have for some $\rho_1,\rho_2>0$ that,
\[1-\kappa^{-1}-\rho_1|\tilde{\lambda}_{1,n_j}|^{-1}<|z_{\beta_{1,n_j}\beta_{i,n_j}\beta^{-1}_{1,n_j} ,l}|,
|z_{\beta_{1,n_j}\beta_{i,n_j}\beta^{-1}_{1,n_j} ,u}|<1+\kappa+\rho_2|\tilde{\lambda}_{1,n_j}|^{-1}.\]
Hence there exists $\kappa'>1$ such that,
\[\kappa'^{-1}<|z_{\beta_{1,n_j}\beta_{i,n_j}\beta^{-1}_{1,n_j} ,l}|\le|z_{\beta_{1,n_j}\beta_{i,n_j}\beta^{-1}_{1,n_j} ,u}|<\kappa'.\]
Consider, 
\[S_j=<\beta_{1,n_j},..., \beta_{1,n_j}\beta_{i,n_j}\beta_{1,n_j}^{-1} ,...,\beta_{1,n_j}\beta_{k,n_j}\beta^{-1}_{1,n_j}>.\]
Since $|\tr(\beta_{1,n_j}\beta_{i,n_j}\beta_{1,n_j}^{-1})|\to\infty$ and by $k-1$ classical we have $S_j$ satisfies the second set of conditions of Lemma \ref{classical},
hence classical.\par
{\bf Consider (IV)}:\\
In this case, we will produce a set of classical generators for sufficiently large $n.$ By using condition (IV), the generators are conjugate to the original set of generators.\par
There exists $0<\rho_n\to\infty$ with $\rho_n<|\tilde{\lambda}_{1,n}|$, such that 
$|\zeta_{\tilde{\alpha}_{i,n}}\tilde{\lambda}^{2\tilde{l}_{i,n}}_{1,n}|-|\tilde{\lambda}_{1,n}|^{-2}>\rho_n|\tilde{\lambda}_{1,n}|^{-2}.$ 
Let $\chi_{i,n}$ be Mobius transformations defined by $\chi^2_{n}(x)=\frac{\tilde{\lambda}_{1,n}}{\sqrt{\rho_n}}x.$ 
We will show that $\chi_n<\beta_{1,n},...,\beta_{i,n},...,\beta_{k,n}>\chi_n^{-1}$ satisfies Remark \ref{classical-rem-2} of Lemma \ref{classical}.  \par
If we have that $1-\delta<|\eta_{\beta_{i,n}}|\le 1$ for some $1>\delta>0$ then,
\[\frac{|\tilde{\lambda}_{1,n}|}{\sqrt{\rho_n}}(1-\delta)<|\eta_{\chi_n\beta_{i,n}\chi_n^{-1}}|<
\frac{|\tilde{\lambda}_{1,n}|}{\sqrt{\rho_n}}.\]
Otherwise we have,
\[\frac{\sqrt{\rho_n}}{|\tilde{\lambda}_{1,n}|}+\frac{1}{|\tilde{\lambda}_{1,n}|\sqrt{\rho_n}}<|\eta_{\chi_n\beta_{i,n}\chi_n^{-1}}|<\frac{|\tilde{\lambda}_{1,n}|}{\sqrt{\rho_n}}.\]
Similarly, by condition of Collapsing fixed points II$_2$ we have $|\zeta_{\tilde{\alpha}_{i,n}}\tilde{\lambda}_{1,n}^{2\tilde{l}_{i,n}}|<1.$ This gives,
\[\frac{\sqrt{\rho_n}}{|\tilde{\lambda}_{1,n}|}+\frac{1}{|\tilde{\lambda}_{1,n}|\sqrt{\rho_n}}<|\zeta_{\chi_n\beta_{i,n}\chi_n^{-1}}|<\frac{|\tilde{\lambda}_{1,n}|}{\sqrt{\rho_n}}.\]

By Lemma \ref{fix-trace} with Remark \ref{rem-2}, $|z_{\beta_{i,n},\pm}-\zeta_{\beta_{i,n}}|<C|\tr(\beta_{i,n})|^{-2}$ and
$|z_{\beta_{i,n},\mp}-\eta_{\beta_{i,n}}|<C|\tr(\beta_{i,n})|^{-2}$ for some $C>0.$ Since $|\tr(\beta_{i,n})|>\epsilon|\tilde{\lambda}_{1,n}|$ for large $n$
we have, 
\[|z_{\chi_n\beta_{i,n}\chi_n^{-1},\pm}-\eta_{\chi_n\beta_{i,n}\chi_n^{-1}}|<\frac{C}{\epsilon|\tilde{\lambda}_{1,n}|\sqrt{\rho_n}},\quad
|z_{\chi_n\beta_{i,n}\chi_n^{-1},\mp}-\zeta_{\chi_n\beta_{i,n}\chi_n^{-1}}|<\frac{C}{\epsilon|\tilde{\lambda}_{1,n}|\sqrt{\rho_n}}.\]
Hence, 
\[|z_{\chi_n\beta_{i,n}\chi_n^{-1},u}|<\frac{|\tilde{\lambda}_{1,n}|}{\sqrt{\rho_n}}+\frac{1}{|\tilde{\lambda}_{1,n}|\sqrt{\rho_n}}
+\frac{C}{\epsilon|\tilde{\lambda}_{1,n}|\sqrt{\rho_n}}\quad\text{and},\]
\[|z_{\chi_n\beta_{i,n}\chi_n^{-1},_l}|>\frac{\sqrt{\rho_n}}{|\tilde{\lambda}_{1,n}|}-\frac{1}{|\tilde{\lambda}_{1,n}|\sqrt{\rho_n}}-
\frac{C}{\epsilon|\tilde{\lambda}_{1,n}|\sqrt{\rho_n}}.\]
We have $|\tilde{\lambda}_{1,n}|^{-1}<|z_{\chi_n\beta_{i,n}\chi_n^{-1},l}|\le|z_{\chi_n\beta_{i,n}\chi^{-1}_n,u}|<|\tilde{\lambda}_{1,n}|$ for large $n.$\par
By above estimates for fixed points of $\chi_n\beta_{i,n}\chi_n^{-1},$ and $|\tr(\beta_{i,n})|>\epsilon|\tilde{\lambda}_{1,n}|$ we have,
\begin{eqnarray*}
\frac{(|z_{\chi_n\beta_{i,n}\chi_n^{-1},u}|+1)(|\tilde{\lambda}_{1,n}|+1)}{|\tr(\beta_{i,n})|(|\tilde{\lambda}_{1,n}|-|z_{\chi_n\beta_{i,n}\chi_n^{-1},u}|)}
&<&\frac{\frac{|\tilde{\lambda}_{1,n}|^2}{\sqrt{\rho_n}} +\frac{|\tilde{\lambda}_{1,n}|}{\sqrt{\rho_n}}+|\tilde{\lambda}_{1,n}|+\delta
}{\epsilon|\tilde{\lambda}_{1,n}|^2(1-\frac{1}{\rho_n}-\frac{1}{|\tilde{\lambda}_{1,n}|^2\sqrt{\rho_n}}
-\frac{C}{\epsilon|\tilde{\lambda}_{1,n}|^2\sqrt{\rho_n}})}\\
&<&\frac{1 +\frac{1}{|\tilde{\lambda}_{1,n}|}+\frac{\sqrt{\rho_n}}{|\tilde{\lambda}_{1,n}|}+\frac{\delta\sqrt{\rho_n}}{|\tilde{\lambda}_{1,n}|^2}
}{\epsilon\sqrt{\rho_n}(1-\frac{1}{\rho_n}-\frac{1}{|\tilde{\lambda}_{1,n}|^2\sqrt{\rho_n}}
-\frac{C}{\epsilon|\tilde{\lambda}_{1,n}|^2\sqrt{\rho_n}})}\\
\text{by $\rho_n<|\tilde{\lambda}_{1,n}|$} &\text{we}& \text{have some $\delta'>1$ such that,}\\
&<&\frac{\delta'}{\epsilon\sqrt{\rho_n}(1-\frac{1}{\rho_n}-\frac{1}{|\tilde{\lambda}_{1,n}|^2\sqrt{\rho_n}}
-\frac{C}{\epsilon|\tilde{\lambda}_{1,n}|^2\sqrt{\rho_n}})}\to\ 0.
\end{eqnarray*}
For the other part of the conditions of Remark \ref{classical-rem-2} we have: \\
If $|z_{\chi_n\beta_{i,n}\chi_n^{-1},l}|<M$ then, 
\begin{eqnarray*}
\frac{(|z_{\chi_n\beta_{i,n}\chi_n^{-1},l}|+1)(|\tilde{\lambda}_{1,n}|^{-1}+1)}{|\tr(\beta_{i,n})|(|z_{\chi_n\beta_{i,n}\chi_n^{-1},l}|-|\tilde{\lambda}_{1,n}|^{-1})}&<&
\frac{(M+1)(|\tilde{\lambda}_{1,n}|^{-1}+1)}
{\epsilon|\tilde{\lambda}_{1,n}|(\frac{\sqrt{\rho_n}}{|\tilde{\lambda}_{1,n}|}-\frac{1}{|\tilde{\lambda}_{1,n}|\sqrt{\rho_n}}-
\frac{C}{\epsilon|\tilde{\lambda}_{1,n}|\sqrt{\rho_n}}-\frac{1}{|\tilde{\lambda}_{1,n}|})}\\
&<&\frac{M'}{\epsilon(\sqrt{\rho_n}-\frac{1}{\sqrt{\rho_n}}-\frac{C}{\epsilon\sqrt{\rho_n}}-1)}
<\frac{M'}{\epsilon''\sqrt{\rho_n}}\to 0.
\end{eqnarray*}
Otherwise we have $|z_{\chi_n\beta_{i,n}\chi_n^{-1},l}|\to\infty$ and,
\[
\frac{(|z_{\chi_n\beta_{i,n}\chi_n^{-1},l}|+1)(|\tilde{\lambda}_{1,n}|^{-1}+1)}{|\tr(\beta_{i,n})|(|z_{\chi_n\beta_{i,n}\chi_n^{-1},l}|-|\tilde{\lambda}_{1,n}|^{-1})}
<\frac{\delta''}{\epsilon|\tilde{\lambda}_{1,n}|}\to 0,\quad\text{for some $\delta''>0$}.\]
Set $S_n=\chi_n<\beta_{1,n},...,\beta_{i,n},...,\beta_{k,n}>\chi_n^{-1}, S'_n=<\tilde{\alpha}_{1,n},...,\tilde{\alpha}_{i,n},...,\tilde{\alpha}_{k,n}>.$
Since $S_n'$ is $k-1$ classical of standard form and $\beta_{i,n}=\alpha^{\tilde{l}_{i,n}}_{1,n}\alpha_{i,n},$  we have $S_n$ satisfies, 
\begin{itemize}
\item
\[|\tilde{\lambda}_{1,n}|^{-1}<|z_{\chi_n\beta_{i,n}\chi_n^{-1},l}|\le|z_{\chi_n\beta_{i,n}\chi^{-1}_n,u}|<|\tilde{\lambda}_{1,n}|\]
\item
\[\lim_n\left\{\frac{(|z_{\chi_n\beta_{i,n}\chi_n^{-1},u}|+1)(|\tilde{\lambda}_{1,n}|+1)}{|\tr(\beta_{i,n})|(|\tilde{\lambda}_{1,n}|-|z_{\chi_n\beta_{i.n}\chi_n^{-1},u}|)},
\frac{(|z_{\chi_n\beta_{i,n}\chi_n^{-1},l}|+1)(|\tilde{\lambda}_{1,n}|^{-1}+1)}{|\tr(\beta_{i,n})|(|z_{\chi_n\beta_{i,n}\chi_n^{-1},l}|-|\tilde{\lambda}_{1,n}|^{-1})}\right\}=0\]
\end{itemize}
 conditions of Remark \ref{classical-rem-2}.\par
{\bf Consider (II)}:\\
In this case, we will produce a set of classical generators for sufficiently large $n.$ Using condition (II), the set of generators are conjugate to the transformed from the original set generators.\par
Since $|\tilde{\lambda}_{1,n_j}|^{-2}<|\eta_{\tilde{\alpha}_{i,n_j}}|\le 1,$ and by our condition we have
$1<|\eta_{\tilde{\alpha}_{i,n_j}}\tilde{\lambda}_{1,n_j}^2|\le 1+C_2.$ As in (IV) and using same notations similarly we have,
\[\frac{|\tilde{\lambda}_{1,n}|}{\sqrt{\rho_n}}<|\eta_{\chi_n\beta_{i,n_j}\beta_{1,n_j}^{-1}\chi_n^{-1}}|<
\frac{|\tilde{\lambda}_{1,n}|}{\sqrt{\rho_n}}(1+C_2),\]
\[\frac{\sqrt{\rho_n}}{|\tilde{\lambda}_{1,n}|}+\frac{1}{|\tilde{\lambda}_{1,n}|\sqrt{\rho_n}}<|\zeta_{\chi_n\beta_{i,n}\chi_n^{-1}}|<\frac{|\tilde{\lambda}_{1,n}|}{\sqrt{\rho_n}}.\]
\par
Now by same argument as in (IV) we have, 
\[\chi_n<\beta_{1,n_j},...,\beta_{i,n}\beta^{-1}_{1,n_j},...,\beta_{k,n_j}\beta^{-1}_{1,n_j}>\chi_n^{-1}\]
are classical Schottky groups for large $j.$\par
{\bf Consider (III)}:\\
Finally in this case, we will produce a set of classical generators which are conjugate to the transformed set of generators from the original set.\par
As in (I) we have, $1<|\zeta_{\beta_{1,n_j}\beta_{i,n_j}}|<(1+C_3).$ And also as in (IV) and (II) we also have,
\[\frac{|\tilde{\lambda}_{1,n}|}{\sqrt{\rho_n}}<|\zeta_{\chi_n\beta_{1,n_j}\beta_{i,n_j}\beta_{1,n_j}^{-1}\chi_n^{-1}}|<
\frac{|\tilde{\lambda}_{1,n}|}{\sqrt{\rho_n}}(1+C_3),\]
\[\frac{\sqrt{\rho_n}}{|\tilde{\lambda}_{1,n}|}+\frac{1}{|\tilde{\lambda}_{1,n}|\sqrt{\rho_n}}<|\zeta_{\chi_n\beta_{1,n_j}\beta_{i,n}\beta_{1,n_j}^{-1}\chi_n^{-1}}|<\frac{|\tilde{\lambda}_{1,n}|}{\sqrt{\rho_n}}.\]
Then by same arguments as in (IV) we have, 
\[\chi_n<\beta_{1,n_j},...,\beta_{1,n_j}\beta_{i,n}\beta^{-1}_{1,n_j},...,\beta_{1,n_j}\beta_{k,n_j}\beta^{-1}_{1,n_j}>\chi_n^{-1}\]
are classical Schottky groups for large $j.$\par
Hence we have completed the proof for all cases (I), (II), (III), (IV).\par
Now from above proofs of (I)-(IV), suppose we have some generators that satisfies some of (I)-(IV) then, there exists 
$a_{2,n},...,a_{k,n},b_{2,n},...,b_{k,n}\in\{1,-1,0\}$ such that,
\[\chi_n<\beta_{1,n_j},...,\beta^{a_{i,n_j}}_{1,n_j}\beta_{i,n}\beta^{b_{i,n_j}}_{1,n_j},...,\beta^{a_{k,n_j}}_{1,n_j}\beta_{k,n_j}\beta^{b_{k,n_j}}_{1,n_j}>\chi_n^{-1},\]
is classical Schottky group.
\par
\subsection{Proof of Theorem \ref{2-fixed-point} and Rank $k$} 
\emph{Completing proof of Theorem \ref{2-fixed-point}}:
For general case, we combine results of Collapsing fixed points I, II to complete the proof as follows. Given a sequence $\G_n$ which is $k-1$ classical with 
$\mathfrak{D}_n\to 0$, it follows that if $Z_{\G_n}$ is not bounded below by some constant $c>0$ then, we have some elements of all generating sets of the sequence having Collapsing fixed points I, or II, or both as defined earlier. In both collapsings fixed points situations, we showed that either the generating set gives classical generators, or it leads to bounded $Z_{\G_n}$ below.
In particular, suppose both collapsing I and II exists for all sequence $S_n$ of generating set of $\G_n$ . If $\G_n$ is sequence of non-classical Schottky groups
with $\mathfrak{D}_n\to 0$ then
it follows from results for Collapsing I and II, we must have a 
uniform lower bound on $Z_{\G_n}.$ 
Hence we have completed proof Theorem \ref{2-fixed-point}.
\end{proof}
\begin{thm}[Rank-k classical Schottky]\label{rank-k}
There exists $\nu_k>0$ such that any rank $k$ non-classical Schottky group $\G$ must have $\mathfrak{D}_\G>\nu_k.$
\end{thm}
\begin{proof}
We proof by induction. By Main theorem of \cite{HS}, there exists $\nu_2>0$ such that all rank $2$ is classical when Hausdorff dimension $<\nu_2.$
Let $k>3.$ Hence by induction and Theorem \ref{t-space} and Thereom \ref{2-fixed-point}, there exists $\nu_k>0$ such that any $\mathfrak{D}_\G<\nu_k$ is classical Schottky group.
\end{proof}
\section{Extension from fixed Rank $k$ to all ranks}
This is section we will prove that our result on fixed rank $k$ can be extended to all finitely generated $\G.$\par
Define $\mathcal{D}_k$ as:
\[\mathcal{D}_k=\sup\{\mathcal{D}'_k|\mbox{any rank $k$ nonclassical Schpottky group $\G$ must have $\mathfrak{D}_\G\ge\mathcal{D}'_k$}\}.\] 
\begin{lem}\label{k-finite}
There exists a positive number $\delta>0$ and $K$ a positive integer such that, given any
$\G_{k+1}$  rank $k+1$ Schottky group with the property that all rank $N<k+1$ subgroups of $\G_{k+1}$ are classical schottky groups then,
$\G_{k+1}$ is classical Schottky group provided that $\mathfrak{D}_{\G_{k+1}}<\delta$ and $k>K$.
\end{lem}
The idea of the proof is induction on rank. To do so we use Proposition \ref{trace-rank}, which say that with respect to $\alpha_{1,n}$ there must exists some generator such that the norms of its trace growth exponentially with respect to the rank $k.$ So assuming that we have all rank $k$ of ``small" Hausdorff dimensions are classical
Schottky groups then we will show that the additional generator of the ``$k+1$" element will have disjointed Schottky circles from rest of the ``$k$" elements when the rank is sufficiently large.\par
\begin{proof}
Let $\{\G_{k+1}\}$ be a sequence of Schottky groups each $\G_{k+1}$ is of rank $k+1,$ such that
any rank $N$ that is $<k+1$ is classical.
We will show that for sufficently large $k,$ $\G_{k+1}$ is classical for small $\mathfrak{D}_{k+1}.$\par
In order to prove $\G_{k+1}$ is eventually classical we will need to use the estimates given by Proposition \ref{trace-rank}. 
We denote generators set of $\G_{k+1}$ as $S_{\G_{k+1}}=<\alpha_{1,k+1},...,\alpha_{k+1,k+1}>.$ We can assume that $\mathfrak{D}_{\G_{k+1}}\to 0,$ otherwise 
we are done with the lemma.\par
If necessary, for $i\ge 2$ replace $\alpha_{i,k+1}$ by $\alpha^{l_{i,k+1}}_{1,k+1}\alpha_{i,k+1}\alpha^{m_{i,k+1}}_{1,k+1}$ 
such that 
\[|\lambda_{1,k+1}|^{-2}\le|\zeta_{\alpha^{l_{i,k+1}}_{1,k+1}\alpha_{i,k+1}\alpha^{m_{i,k+1}}_{1,k+1}}|,|\eta_{\alpha^{l_{i,k+1}}_{1,k+1}\alpha_{i,k+1}\alpha^{m_{i,k+1}}_{1,k+1}}|\le 1 ,\] 
we can assume that all $|\lambda_{1,k+1}|^{-2}\le|\zeta_{\alpha_{i,k+1}}|,|\eta_{\alpha_{i,k+1}}|\le 1, i\ge2.$\par
For each $n\le k+1,$ we choose and arrange the geneartors as follows.
Follows from Corollary \ref{trace-rank-dis} there exists a generators that satisfies the estimates stated in Corollary \ref{trace-rank-dis} 
with respect to $\lambda_{\alpha_{1,k+1}}.$
Set $\alpha_{k+1,k+1}$ that satisfies the estimates provided by Corollary \ref{trace-rank-dis}. We arrange $\alpha_{k,k+1}$ simiarly by using Corollary \ref{trace-rank-dis}
on $<\alpha_{1,k+1},...,\alpha_{k,k+1}>.$ And we arrange generators this way for all generators down to $\alpha_{n+1,k+1}$ and in particular we have,
\begin{multline*}
|\tr(\alpha_{n+1,k+1})|\\>c\left(\frac{(2n-1)|\lambda_{\alpha_{1,k+1}}|^{2\mathfrak{D}_{\G_{k+1}}}+(2n+1)}
{|\lambda_{\alpha_{1,k+1}}|^{2\mathfrak{D}_{\G_{k+1}}}-1}\right)^{\frac{1}{2\mathfrak{D}_{\G_{k+1}}}}
\left(e^{-\dis(\mathcal{L}_{\alpha_{1,k+1}},\mathcal{L}_{\alpha_{n+1,k+1}})}\right). 
\end{multline*}
Let $\G_{n;k+1}\subset\G_{k+1}$ be generated by 
$S_{\G_{n;k+1}}:=<\alpha_{1,k+1},...,\alpha_{n,k+1}>.$ Since $\mathfrak{D}_{\G_{n;k+1}}\to 0$ for $k\to\infty,$ and by assumption that it is $k-1$ classical,
it follows from Lemma \ref{k-1-form} we can choose the generators set
$S_{\G_{n;k+1}}$ such that it is $k-1$ classical of standard form.
\par
Now we will show that the circles of $\g_{N+1,k+1},...,\g_{k+1,k+1}$ are disjoint from circles $C_{|\lambda_{\alpha_{1,k+1}}|^{-1}}, C_{|\lambda_{\alpha_{1,k+1}}|}.$\par
Suppose that, 
\[\frac{|z_{\g_{m,k+1},-}-z_{\g_{m,k+1},+}|}{(|\lambda_{\alpha_{1,k+1}}|^2-1)^2}>C\] 
for some $C>0,m>N.$ 
by Lemma \ref{delta-z} and above inequality and assume that $\mathfrak{D}_{\G_{k+1}}<\frac{1}{6}$ then,
\begin{multline*}
\frac{1}{|\tr(\g_{m,k+1})|(|\lambda_{\alpha_{1,k+1}}|^2-1)}\\
<\rho'\left( \frac{ |\lambda_{\alpha_{1,k+1}}|^{2\mathfrak{D}_{\G_{k+1}}}-1 }
{(2m-1)|\lambda_{\alpha_{1,k+1}}|^{2\mathfrak{D}_{\G_{k+1}}}+(2m+1)}
\right)^{\frac{1}{2\mathfrak{D}_{\G_{k+1}}}}(|\lambda_{\alpha_{1,k+1}}|^2-1)^{-3}\\
<\rho'\left(\frac{|\lambda_{\alpha_{1,k+1}}|^{2\mathfrak{D}_{\G_{k+1}}}-1}
{((2m-1)|\lambda_{\alpha_{1,k+1}}|^{2\mathfrak{D}_{\G_{k+1}}}+(2m+1)) (|\lambda_{\alpha_{1,k+1}}|^2-1)^{6\mathfrak{D}_{\G_{k+1}}}}
\right)^{\frac{1}{2\mathfrak{D}_{\G_{k+1}}}}\\
\mbox{since}\quad\frac{ |\lambda_{\alpha_{1,k+1}}|^{2\mathfrak{D}_{\G_{k+1}}}-1 }{(|\lambda_{\alpha_{1,k+1}}|^2-1)^{6\mathfrak{D}_{\G_{k+1}}}}
<\frac{ |\lambda_{\alpha_{1,k+1}}|^{2}-1 }{(|\lambda_{\alpha_{1,k+1}}|^2-1)^{6\mathfrak{D}_{\G_{k+1}}}}<1 \quad\mbox{we have,}\\
<\rho''\left(\frac{1}
{(2m-1)+(2m+1)}\right)^{3}\to 0,\quad\mbox{for}\quad m\to\infty.
\end{multline*}
for small $\mathfrak{D}_{\G_{k+1}}$ and large $m.$
 
\end{proof}
If $\frac{|z_{\g_{m,k+1},-}-z_{\g_{m,k+1},+}|}{(|\lambda_{\alpha_{1,k+1}}|^2-1)^2}\to 0$ then first let us assume that $|\lambda_{1,k+1}|>C'$ for some $C'>0.$
In this case we have the same result as to Proposition \ref{delta-z-r} that $r_{\g_{m,k+1}}+r'_{\g_{m,k+1}}\to 0.$ To see this we can simply reproduce the same
proof as given before. More precisely, we can assume that $|\tr(\g_{m,k+1})|<C''$ for some $C''>0$ othewise it is trivial, then we have as in last inequality of
proof of Proposition \ref{delta-z-r}:
\begin{align*}
r_{\g_{m,k+1}}+r'_{\g_{m,k+1}}&<\frac{M'|z_{\g_{m,k+1},-}|}
{\left((2m-1)|\lambda_{\g_{m,k+1}}|^{2\mathfrak{D}_{\G_{k+1}}}+(2m+1)\right)^{\frac{1}{2\mathfrak{D}_{\G_{k+1}}}}}\\
&<\frac{M''}{(2m+1)}\to 0,\quad\quad\mbox{for}\quad \mathfrak{D}_{\G_{k+1}}\le 1.
\end{align*}
For $|\lambda_{\alpha_{1,k+1}}|\to 1$ then we can follow the same procedure as in the proof of Lemma \ref{m=0} and Proposition \ref{classical-i-2}. More precisely,
we break down to two cases based on bounds of $f(\g_{m,k+1},\alpha_{1,k+1})\to 0$ or $f(\g_{m,k+1},\alpha_{1,k+1})>M$ for some $M>0.$\par
For $f(\g_{m,k+1},\alpha_{1,k+1})\to 0,$ by repeating the same technique we can show that 
$S''_{\G_{k+1}}$ with $\alpha_{1,k+1}$ replaced by $\alpha^{-1}_{1,k+1}\g_{m,k+1}$ in $S'_{\G_{k+1}}$ generates a classical Schottk group for large $m.$
Similarly for $f(\g_{m,k+1},\alpha_{1,k+1})>M,$ we have $\hat{S}_{\G_{k+1}}$ with $\g_{m,k+1}$ replaces by $\alpha^{-1}_{1,k+1}\g_{m,k+1}$ generates
a classical Schottky group for large $m.$
\par
\begin{cor}\label{k-ext}
$\inf_k\mathcal{D}_k>0.$
\end{cor}
\begin{proof}
Let $\delta>0, K$ be given by Lemma \ref{k-finite} then let $\mathfrak{D}=\min\{\delta,\mathcal{D}_K\}$ then any
nonclassical Schottky group $\G$ must have $\mathfrak{D}_{\G}>\mathfrak{D}.$
\end{proof}
\section{Proof of Main Theorem}
\begin{thm}\label{main-1}
There exists $\epsilon>0$ such that all finitely generated Schottky group $\G$
with $\mathfrak{D}_\G<\epsilon$ is a classical Schottky group.
\end{thm}
\begin{proof}
This follows from Theorem \ref{t-space} and Theorem \ref{2-fixed-point}.
\end{proof}
\begin{proof}[Proof of Theorem \ref{main}]
The final argument that we need to show is that Kleinian groups with limit set of small enough Hausdorff dimension are Schottky groups.
This is a well established fact and detailed arguments can be found in \cite{HS}. For completeness we give a quick review here.\par
Let $\G$ be a non-elementary finitely generated Kleinian
group. 
By Selberg lemma we can assume $\G$ is a torsion-free non-elementary finited generated Kleinian group.
$\mathfrak{D}_{\G}<\frac{1}{2}$ implies $\G$ is geoemtrically finite convex cocompact of second kind. Since 
$\Omega_\G/\G$ consists of finitely many compact compressible Riemann surfaces, we can decomposes $\G$ along compression disks and after
finitely many steps we left with topological balls. Hence $\mathbb{H}^3/\G$ is a handle body.
\end{proof}
\text{}\\
E-mail: yonghou@princeton.edu
\pdfbookmark[1]{Reference}{Reference}
\bibliographystyle{plain}

\end{document}